\documentclass[english]{smfart}
\usepackage[utf8]{inputenc}
\usepackage[english,french]{babel}
\usepackage{smfthm}
\usepackage{mathrsfs}
\usepackage{graphics}
\usepackage{hyperref} 
\usepackage{amssymb, amsmath, mathabx}
\usepackage{lmodern}
\usepackage[all]{xy}
\usepackage{multicol}
\usepackage{color}
\usepackage{tikz-cd}
\usepackage[shortlabels]{enumitem}
\usepackage{scalerel}
\makeindex
\DeclareUnicodeCharacter{00A0}{\relax}

\author{Antoine Ducros}
\address{Sorbonne Université, Université Paris-Diderot, CNRS, Institut de Mathématiques de Jussieu-Paris
Rive Gauche, IMJ-PRG,
F-75005, Paris, France}
\email{antoine.ducros\at imj-prg.fr}
\urladdr{https://webusers.imj-prg.fr/$\sim$antoine.ducros/}

\author{François Loeser}
\address{Sorbonne Université, Institut de Mathématiques de Jussieu-Paris
Rive Gauche, CNRS. Campus Pierre et Marie Curie, case 247, 4 place Jussieu, 75252 Paris cedex 5, France.
}
\email{francois.loeser@imj-prg.fr}
\urladdr{https://webusers.imj-prg.fr/$\sim$francois.loeser/}

\title[Nash structure of curves and tame henselian rationality]{Nash structure of curves over a valued field and
tame henselian rationality}
\dedicatory{To the memory of Zoé Chatzidakis}
\NumberTheoremsAs{subsection}
\SwapTheoremNumbers

\newcommand{\eg}{e.\@g.\@}

\newcommand{\ie}{i.\@e.\@}

\newcommand{\loccit}{loc.\@~cit.\@}
\newcommand{\opcit}{op.\@~cit.\@}
\newcommand{\resp}{resp.\@~}

\newcommand{\rr}{\mathfrak r}
\newcommand{\wrr}{\widehat \rr}
\newcommand{\trr}{\widetilde \rr}
\newcommand{\Tau}{\mathrm T}
\newcommand{\nash}[2]{#1\langle #2\rangle}
\newcommand{\rednash}[2]{\RV(#1)\langle #2\rangle}
\newcommand{\redgrad}[2]{\RV(#1)[#2]}

\renewcommand{\k}{\mathsf k}
\newcommand{\res}{\mathsf {res}}
\newcommand{\RV}{\mathsf{RV}}
\newcommand{\rv}{\mathsf {rv}}
\newcommand{\acl}{\mathsf{acl}}

\newcommand{\abs}[1]{\mathopen|#1\mathclose|}
\newcommand{\an}{^{\mathrm{an}}}
\newcommand{\h}{^{\mathrm h}}

\newcommand{\gm}{\mathbf G_{\mathrm m}}
\newcommand{\gpm}{^{\times}}
\newcommand{\hr}[1]{\mathscr H(#1)}

\def\Hom{\operatorname{Hom}}
\def\spec{\operatorname{Spec}}
\def\colim{\operatorname{colim}}

\def\spf{\operatorname{Spf}}
\def\td{\operatorname{trdeg}}

\def\GL{\operatorname{\mathrm{GL}}}
\newcommand{\acvf}{\mathrm{ACVF}}

\newcommand{\A}{\mathbf A}

\renewcommand{\P}{\mathbf P}
\newcommand{\Q}{\mathbf Q}

\newcommand{\Z}{\mathbf Z}

\renewcommand{\phi}{\varphi}
\renewcommand{\epsilon}{\varepsilon}
\renewcommand{\leq}{\leqslant}
\renewcommand{\geq}{\geqslant}

\renewcommand{\labelitemi}{$\bullet$}

\renewcommand{\theequation}{\alph{equation}}
\SetEnumerateShortLabel{1}{\textnormal{(\arabic{enumi})}}
\SetEnumerateShortLabel{i}{\textnormal{(\roman{enumi})}}
\SetEnumerateShortLabel{a}{\textnormal{(\alph{enumi})}}
\SetEnumerateShortLabel{A}{\textnormal{(\Alph{enumi})}}
\SetEnumerateShortLabel{2}{\textnormal{(\arabic{enumii})}}
\SetEnumerateShortLabel{j}{\textnormal{(\roman{enumii})}}
\SetEnumerateShortLabel{b}{\textnormal{(\arabic{enumi}\alph{enumii})}}
\SetEnumerateShortLabel{c}{\textnormal{(\alph{enumii})}}
\SetEnumerateShortLabel{B}{\textnormal{(\Alph{enumii})}}

\begin{document}
\begin{abstract}
We provide a new proof, model-theoretic and geometric in nature,
of  the tame henselian
rationality theorem of Kuhlmann. The proof takes place within the framework of stable completions of algebraic varieties over a valued field.
It relies on two main results of independent interest:
a Nash structure theorem for the stable completion of an 
algebraic curve 
and a tame descent theorem for abstract open polydiscs.
\end{abstract}


\maketitle

\tableofcontents
\setcounter{section}{-1}
\section{Introduction}
\subsection*{Motivation}
The initial goal of this work was to answer a question asked by Franziska Jahnke and Franz-Viktor Kuhlmann
to the second author after a series of talks he gave in Singapore during Summer 2025. Among various results, 
he had presented a new proof, written down in 
\cite[Appendix]{ducros-h-l-y2024}, of a classical result of valuation theory due to Kuhlmann 
(\cite{kuhlmann2010}): \emph{every Abhyankar extension of a
defectless valued field is still defectless}. This new proof relies on model theory, and more precisely
on the theory
of \emph{stable completions} (the  model-theoretic version of Berkovich geometry)
developped by Hrushovski and the second author 
\cite{hrushovski-l2016} for studying the homotopy type of algebraic Berkovich spaces. 
The question of Jahnke and Kuhlmann was whether one can use the framework of
stable completions 
for getting an alternative proof of
another important result of valuation theory, namely the \emph{tame henselian
rationality theorem} of Kuhlmann
\cite{kuhlmann2019}.

Let us first say a few words about this theorem. A valued field is called tame if it is henselian and
all its finite extensions are tamely ramified; this amounts to requiring that it be
 defectless with perfect residue field
and
$p$-divisible value group, for $p$ the residue characteristic exponent; every henselian valued field with residue
characteristic zero is tame. 
Tame henselian fields were first introduced and studied by  Kuhlmann in \cite{fvk2016}
where he initiated the investigation
 of their algebraic and model theoretic properties, 
showing in particular an
 Ax-Kochen-Ershov principle for this class of fields. 
More recently, 
tame fields have played a crucial role in the groundbreaking work of Jahnke and Kartas on the model theory of perfectoid fields
and tilting equivalence
\cite{jahnke-kartas}, see also
\cite{anscombe2026}, the reason for this being that
a continuous
non-principal ultraproduct of perfectoid fields is 	 a tame field. 
Both \cite{fvk2016} and \cite{jahnke-kartas} use in a crucial way a certain
relative embedding property, which in turn relies on the following tame henselian rationality theorem from Kuhlmann
\cite{kuhlmann2019}, that can be stated as follows:

\emph{Let $L$ be a finitely generated extension of transcendence degree 1
of a tame valued field $K$, equipped with 
an immediate extension $v$ of the structure valuation of $K$. Then there exists an element $t$ of $L$
transcendental over $K$
such that the henselizations of $K(t)$ and $L$ with respect to $v$
are equal.} 

Let us mention that even the case where $K$ is algebraically closed is difficult (and it was handled by Kuhlmann as a first step toward
the general statement). 

\subsection*{Our results}
In this paper, we are indeed able to use the formalism of stable
completions to provide a new proof 
of the tame henselian rationality theorem (Theorem \ref{henselian-rationality}) with a model-theoretic and geometric flavour. 
In our approach, this theorem essentially relies on two results:

\begin{itemize}[label=$\diamond$]
\item A ``Nash structure theorem" for the stable completion $\widehat X$
of a $K$-algebraic curve $X$
(Theorem \ref{structure-theo}). This structure theorem is an analogue 
of a structure theorem for Berkovich analytic curves, which is 
equivalent to the semi-stable
reduction theorem and can be stated and proved
in several ways; the presentation
we mimic here is the one provided by
the first author
in \cite{ducros2024}. 

\item A tame descent theorem for ``abstract
open polydiscs"
(Theorem \ref{theo-tame-descent}), which we use only in dimension 1
and which enables us to deduce the tame case from the algebraically closed one. 
This descent theorem is an analogue of the  tame descent theorem for analytic
open polydiscs
proved by the first author
in \cite{ducros2013b}.
\end{itemize}

Let us say a few words about the reference to Nash. The theory of stable completions provides a very robust model-theoretic
avatar
of the underlying topological space of a Berkovich space, but it lacks a substitute for
the analytic structure. And we definitely needed such a substitute: it
was both
technically essential for adapting the proof of
the descent theorem of \cite{ducros2013b} to our setting
and
conceptually helpful 
for stating,  proving and understanding our results on the structure of stable completions of curves. 
This substitute is introduced in section \ref{section-nash}, where we define 
\emph{Nash isomorphisms} (Definition \ref{defi-nash-iso}) and \emph{Nash functions} (\ref{nash-functions}); 
we have chosen this terminology since our definitions are reminiscent of
those of Nash isomorphisms and Nash functions in 
the theory of real closed fields, in which they also are a good substitute for real-analytic isomorphisms and real-analytic functions.

This notion of a Nash isomorphism is closely related, 
if not equivalent in appropriate settings,  to that of an 
\emph{$h$-isomorphism}, which was introduced by Hübner
and Temkin in their work \cite{hubner-temkin2024} in the adic setting.

\subsection*{Structure of this manuscript}
Section 1 is devoted to stating the  conventions we use and to providing some basic reminders on the model theory of valued fields, 
with an emphasis  on the sort $\RV$
which is used systematically throughout this article.  We have chosen to use the denomination $\RV$
(which is standard in model theory) for the sake of consistency with the general viewpoint of this text, but we could have used equivalently the more
algebraic language of \emph{graded reduction} in the sense of Temkin \cite{temkin2004}; for  similar reasons we borrow the terminology 
of \emph{generalized varieties} from \cite{hrushovski-k2006} but we could have called them \emph{graded schemes} instead.

In Sections \ref{topology-curve} and \ref{topology-curve-complement} we describe and investigate the structure of the 
stable completions of curves and of related spaces: the stable completion $\widehat X$ of a curve $X$ over a valued field is the analogue
of its Berkovich analytification, and we also consider the analogue $\widetilde X$ of its Huber analytification (mainly as an auxiliary tool
for understanding $\widehat X$). Section \ref{topology-curve} is essentially a recollection of statements from \cite{hrushovski-l2016}, the main
one being the existence for any curve $X$ of an \emph{admissible $K$-skeleton} $\Sigma$ on $\widehat X$, that is, a
$K$-skeleton (\ie, a $K$-definable finite graph) which is the 
target of a $K$-definable deformation retraction of
$\widehat X$. In Section \ref{topology-curve-complement} some new results are established
concerning the complement in $\widehat X$ of a
closed skeleton $\Sigma$: we prove 
that
$\pi_0(\widehat X\setminus \Sigma)$ is definable 
and that each connected component of $\widehat X\setminus \Sigma$
is the stable completion of a definable subset of $X$
(Proposition \ref{prop-connected-comp-definable}). 
The most important case, to which the
general one reduces, and for which one has a very nice description 
of the set of components and of the components themselves
(using $\widetilde X$)
is when $\Sigma$ is
an admissible skeleton containing all branches at its simple points
(Definition \ref{defi-allbranches},
Lemma \ref{lemm-mathfrakr-fibers} and \ref{components-are-stable-completions}).

In Section 4 we introduce the notion of Nash isomorphisms and Nash functions. 
Let $X$ be a $K$-algebraic variety and let $U$ be a $k$-definable subset of
$X$. If $Y$ is another $K$-algebraic variety and 
$V$ is a $K$-definable subset of $Y$, 
a $K$-algebraic map $f\colon Y\to X$
is said to induce a Nash isomorphism between 
$V$ and $U$ it $V$ is contained in the étale locus of $f$, if $f$ induces a bijection $V\simeq U$, and if the induced map $\widehat V
\to \widehat U$ is a homeomorphism. 
If this is the case any $K$-regular function on $Y$ gives rise via $f$ to a $K$-definable function 
on $U$, and $K$-definable functions on $U$ arising this way for some $(Y,V,f)$ as above are called Nash functions; we denote by
$\nash KU$ the algebra of such functions. 
Under 
a boundedness assumption on $U$ we can also define an $\RV$-reduction $\rednash KU$ of the algebra of $K$-definable Nash function on $U$; this
is a (generalized) algebra over $\RV(K)$. 
We show that when the ground field $K$
is henselian
 $\nash LU=L\otimes_K \nash KU$ for every finite extension $L$
of $K$ (Lemma \ref{cofinality}) and that $\rednash LU=\RV(L)\otimes_{\RV(K)} \rednash KU$ if $L$
is moreover assumed to be tamely ramified over $K$
(Corollary \ref{coro-tame-reduction}); this
latter fact relies heavily on the characterization of tame ramification in terms
of $\RV$ \cite{ducros2013b}. 

In section \ref{abstract-polydiscs} we assume
that $K$ is henselian 
and we study $\nash KU$ for $U$  a $K$-definable abstract polydisc, that is, a
$K$-definable subset of a variety Nash-isomorphic through some 
$n$-uple of $K$-regular functions $(f_1,\ldots, f_n)$ 
to an open polydisc centered at the origin, say of polyradius $r$. Any Nash function on $U$
has a well-defined \emph{formal}
Taylor development $\sum a_I f^I$ at the unique point $u$ of $U$ at
which all $f_i$ vanish (because Nash isomorphisms preserve the henselizations of local rings, and a fortiori
their completions), enabling us to see elements of $\nash KU$ as formal power series in the $f_i$. 
But be aware that the valued fields we deal with are not assumed to be complete in any sense, and can be of arbitrary height, 
so that there is no general theory here allowing for the evaluation of arbitrary power series in $f$ at points of $U$, even under suitable growth
conditions on the coefficients. We are nonetheless able to prove 
that Nash functions on $U$ viewed as power series behave as expected
(but this requires
some work!).
For instance we prove that if $g=\sum a_I f^I$ 
is such a function, if $s$ is some polyradius $<r$ and if 
and if $u_s$ denotes the unique pre-image under $f$ of the Gauss point $\eta_s$, then 
\[\sup_{\{\abs f\leq s\}}\abs g=\abs{g(u_s)}=\max_I \abs{a_I}\cdot s^I,\]
this maximum being achieved only for finitely many indices
(Theorem \ref{theo-norme-gauss}); Corollary \ref{gauss-norm-open-disc}
of this theorem then asserts that
$\sup_U \abs g=\max_I  \abs{a_I}\cdot r^I$ (but here the maximum can be achieved for
infinitely many indices). Theorem \ref{theo-norme-gauss} is proved by induction on dimension, 
so the core is the proof in dimension 1, which uses the topology of curves investigated in Sections
\ref{topology-curve} and \ref{topology-curve-complement}. And we also prove that if $n\geq 1$
and $J$ denotes
the set of multi-expoents $I$ with last coordinate $0$ (otherwise said the set of those $I$ such that $f_n$ does not occur
in the monomial $f^I$) then $\sum_{I\in J}a_I f^I$ is (the Taylor developement of) a Nash function $h$ on $U$ 
which can be described as the composition of a ``projection" onto the locus
$\{f_n=0\}$ and of a Nash function on the latter
(Proposition \ref{restriction-nash-functions}). We warn the reader
of the following: by definition of a Nash function, $g$ is defined on some 
variety $Y$ equipped with an étale map to $X$, 
but for defining $h$ a further étale refinement $Z\to Y$ is likely unavoidable 
in general, see Remark \ref{etale-unavoidable}.

In Section \ref{tame-descent} we  still assume that $K$
is henselian and
describe $\rednash KU$ for $U$ a $K$-definable abstract polydisc, relying upon
the results on $\nash KU$ proved in Section \ref{abstract-polydiscs}. 
If $(f_1,\ldots, f_n)$ are
as above, we show that $\rednash KU$ is a local (generalized) sub-algebra of the (generalized)
ring of (generalized) formal power series
$\RV[\![\rv(f_1),\ldots, \rv(f_n)]\!]$ and that a family $(g_1,\ldots, g_n)$ of Nash functions on $U$ induces a Nash isomorphism with
an open polydisc if and only if the $\rv(g_i)$ generate the maximal ideal of $\rednash KU$ (Theorem
\ref{reduction-abstract-disc}). We can then prove our tame descent theorem for open polydiscs (Theorem \ref{theo-tame-descent}): 
if $L$ is a tamely ramified finite extension of $K$ and if $U$ is a $K$-definable subset of a $K$-variety $X$ such that there exists
$(g_1,\ldots, g_n)\in \nash LU$ inducing a Nash isomorphism between $U$ and an open polydisc, then there exist
$(f_1,\ldots, f_n)\in \nash KU$ inducing a Nash isomorphism between $U$ and an open polydisc; the proof
of this Theorem is completely parallel to that of \cite[Théorème 3.5]{ducros2013b}. 
Let us mention that if
$n=1$ one can do slightly better: in this case
if $g_1$ is an $L$-regular function on some $L$-Zariski open neighborhood of $U$
in $X$, one can choose $f_1$ in the ring of $K$-regular functions on some $K$-Zariski open neighborhood of $U$ in $X$
(no serious étale refinement is needed in this case). 

The purpose of Section 7 is to  construct  ``generic" abstract open polydiscs. More precisely we 
assume that the ground field $K$ is algebraically 
closed. We start from a 
$K$-variety (say, integral)
of dimension $n$
and a $K$-Zariski generic point $x$ of $X$ such that the valued field $K(x)$ is Abhyankar over $K$. Let $f_1,\ldots, f_m$ be 
non-zero $K$-rational functions on $X$ such that the $\rv(f_i(x))$ generate $\RV(K(x))$. 
By generic smoothness it is of course very easy to exhibit a $K(x)$-definable
abstract open polydisc in $X$ containing $x$, but we get a better and more useful result:  indeed, we prove
(Theorem \ref{maintheo-balls})
that
there exists an abstract 
open polydisc $U\subset X$ containing $x$ which is definable over $K\cup\{\rv(f(x))\}$ and which is moreover
\emph{atomic} (\ie, it has no proper non-empty definable subsets) over this set of parameters. 
The proof relies on generic smoothness but only at the level of the residue field, and makes a crucial use
of the fact that if $c$ is a point of $\RV^n$ with at least one coordinate transcendental over $\RV(K)$,
the open polydisc  $\rv^{-1}(c)\subset \A^n$ is atomic over $K\cup\{c\}$ and even over its model-theoretic
algebraic closure (\cite{hrushovski-k2006}; see the comments in \ref{comments-atomic}). 

Sections \ref{nash-structure-coverings} and \ref{nash-structure-triangulations}
are devoted to the Nash structure of the stable completion $\widehat X$ 
of a $K$-algebraic curve $X$. We prove 
more precisely (Theorem \ref{structure-theo})
that 
$\widehat X$ admits a \emph{$K$-triangulation}. This is essentially a finite $K$-definable subset $\mathscr V$ such that 
there exists a finite set $E$ of $\overline K$-regular functions (each of them being defined on a
$\overline K$-Zariski open subset of $X$) fulfilling the following property: 
every connected component 
of $\widehat X\setminus \mathscr V$ is of the form $\widehat U$ for some definable subset $U$ of $X$
which is Nash-isomorphic
via some $f\in E$
to an open disc or an open annulus, the second case occuring only for finitely many
components (if the component $U$ is Nash-isomorphic to an annulus, we also require 
that its two orientations be individually inter-definable with $U$; otherwise said if $U$ is defined over
some finite extension $L$ of $K$, so are its two orientations). 
Such a triangulation gives rise to an admissible $K$-skeleton $\Sigma$, 
obtained by taking the union of $\mathscr
V$ and of the skeletons of the finitely many connected components of
$\widehat  X\setminus \Sigma$ which are abstract annuli; then every connected component of 
$\widehat X\setminus \Sigma$ is of the form $\widehat U$ for $U$
a definable subset of $X$ which is Nash-isomorphic to an open disc
via some $f\in E$ (we call such skeletons \emph{Nash-admissible}). 
The proof of  
Theorem \ref{structure-theo} reduces through 
easy Galois descent
to the case where $K$ is algebraically closed, in which case 
it is deduced from
a first structure theorem (Theorem \ref{theo-core}) asserting that $\widehat X$ can be covered by the stable completions
of finitely
many \emph{nice charts} (Definition \ref{defi-nice-charts}), which itself relies in an essential way on the existence of 
atomic open (poly)discs proved in Section \ref{atomic-polydiscs}.
Let us mention that the existence of triangulations in this sense and in this
setting is the exact analogue of the existence of triangulations on Berkovich analytic curves 
\cite[Théorème 5.1.14]{ducros2024}. 

Section \ref{applications} is devoted to some applications of our former results. The first one 
(Theorem \ref{definability}) still concerns the structure of curves: 
we deduce from the existence of triangulations that
the pro-definable set $\widetilde X$ is actually definable, 
and also that if $\Sigma$ is a closed skeleton of $\widehat X$, the definable set
$\pi_0(\widehat X\setminus \Sigma)$ is $\RV$-internal.
The second 
one explains how our theorem on the existence 
of triangulations implies the corresponding theorem in the Berkovich setting, and the
semi-stable reduction theorem over arbitrary height 1 valuation rings (Theorem \ref{semi-stable}). 
The third one is
the tame henselian rationality
Theorem (Theorem \ref{henselian-rationality}), which is proved as follows. If $L$ is an immediate
one-dimensional finitely generated extension of a tame valued field $K$, we write $L=K(x)$ where $x$
is some $K$-Zariski generic point of a
$K$-curve $X$. Choose a
$K$-triangulation on $\widehat X$
and denote by $\Sigma$ the corresponding Nash-admissible skeleton. We view $x$ as a simple point of $\widehat X$; it then belongs
to a connected component of $\widehat X\setminus \Sigma$, which is of the form $\widehat U$ for some 
definable subset $U$ of $X$. As the set of connected components of $\widehat X\setminus \Sigma$ is $\RV$-internal, 
$U$ is definable over $\overline K\RV(K(x)))=\overline K$. Now the fact that $K$
is tame and that generalized residue fields behave well with respect to
tamely ramified
extensions implies that $L\otimes_K \overline K$ is a field to which
$v$ extends uniquely, which means that
$\mathrm{tp}(x/\overline K)$ is
Galois-invariant
and implies that $U$ is
actually
$K$-definable. 
 Since $\Sigma$ is Nash-admissible there exists a Nash isomorphism
between the $K$-definable set
$U$ and an open disc which is induced by an $F$-regular function on an
$F$-Zariski
open neigborhood of $U$ for some finite extension $F$ of $K$.  As $K$ is tame, $F$ is tamely ramified
over $K$, so by our tame descent theorem for abstract open discs (and more precisely by its stronger version specific
to dimension 1)
we can assume that $F=K$; then $f$ induces an isomorphism between the henselizations
of $K(f(x))$ and $K(x)$, and we are done. 

Let us end this description of our work by a remark: Sections \ref{section-nash} to \ref{atomic-polydiscs}
go far beyond what was strictly needed for the structure of curves (Sections \ref{nash-structure-coverings} and
\ref{nash-structure-triangulations}) and tame henselian rationality (Section \ref{henselian-rationality}), if only because
they apply to varieties of arbitrary dimension. But we hope that the resuts therein are intrinsically interesting, and that they
will be useful for future works in geometry over valued fields.

\subsection*{Acknowledgements}
We are extremely
grateful to Franziska Jahnke and Franz-Viktor Kuhlmann for their interest in our former work
\cite{ducros-h-l-y2024} (joint with E. Hrushovski and J. Ye) and for having asked the question that eventually
gave rise
to this article. F.L. would like to thank the organizers of the period ``Recent Applications of Model Theory'' that took place at the Institute for Mathematical Sciences of the National University of Singapore
in Summer 2025 for their kind invitation that made this work possible. 
During the preparation of this article, A.D. 
was supported by the project AdAnAr (ANR-24-CE40-6184) of the Agence nationale de la recherche
and F.L. was partially supported by the Institut Universitaire de France.

\medskip
\emph{We dedicate this paper to the memory of our friend and colleague Zoé Chatzidakis,
who passed away on January 22, 2025.
She was always keen on seeing new applications of model theory to algebra and geometry.
We hope that this paper will serve as a fitting, if modest, tribute to her memory.}

\section{Model theory of valued fields: reminders and basic facts}\label{reminders}

\subsection{}
Throughout this paper we shall use the multiplicative notation for valued fields: a valuation on a field $F$ is a map
$\abs \cdot \colon F\to G_0$ where $G$ is a multiplicative ordered abelian group (with unit element $1$) and $G_0=G\cup\{0\}$
is the monoid obtained by formally adjoining to $G$ a smallest absorbing element.

\subsection{}
We will use the multi-sorted language of valued fields introduced in 
\cite{haskell-h-m2006}. Besides the three standard sorts, namely, the valued field itself, the residue field and the value group, it
encompasses a countable family of so-called \textit{geometric sorts}. We will not use the latter explicitly
but adjoining these geometric sorts ensures that the theory $\acvf$ of algebraically closed
(non-trivially) valued fields eliminates imaginaries \cite[Theorem 3.4.10]{haskell-h-m2006}, so that 
objects which are merely interpretable in the classical setting will be definable, like $\RV$ (see below)
or the \textit{stable completion} of an algebraic curve
defined by Hrushovski and the second author  in their work \cite{hrushovski-l2016}.

\subsection{}
In this text, definable ``sets" will be considered as functors on the category whose objects are models of $\acvf$ and whose arrows are isometric
embeddings, and definable ``maps" as natural transformations. Basic examples will be the value group
$\Gamma$ that maps $F$ to $\abs{F\gpm}$, its variant $\Gamma_0=\Gamma\cup\{0\}$ that maps $F$ to $\abs F$, and 
the residue field $\k$ that maps $F$  to $\{z\in F, \abs z\leq 1\}/\{z\in F, \abs F<1\}$. 

We shall also need the functor 
\[\RV\colon F\longmapsto \coprod_{\gamma \in \abs{F\gpm }}
\{z\in F,\abs z\leq \gamma\}/\{z\in F,\abs z<\gamma\}.\]
Note that $\RV(F)$ inherits a multiplication, as well as a partially defined addition: one can only add up two terms
belonging to the same summand,  and each summand is an abelian group (with its own zero-element); the summand
of degree one is nothing but the residue field. 
Therefore
$\RV(F)$ appears as kind of a a generalized ring, and even as a generalized field since 
$1$ is non-zero and its group $\RV\gpm(F)$ of invertible elements coincides with its subset of non-zero
elements. 

The image of an element $x\in F\gpm$ in the group $\RV\gpm (F)$ will be denoted by $\rv(x)$. One has a natural exact sequence
\[\begin{tikzcd}
1\ar[r] &\k\gpm\ar[r]&\RV\gpm\ar[rr,"\abs \cdot"]&&\Gamma\ar[r]& 1\end{tikzcd}.\]

Most notions of commutative algebra have straighforward analogues to this context of generalized rings (or fields). 
We will use them freely, and refer the reader to the first section of \cite{ducros-2021a} for more details. 
For instance there is a notion of a generalized polynomial ring over $\RV(F)$, but the indeterminates have to be assigned a degree. 
This gives rise to a theory of generalized algebraic extensions, transcendence degree, etc. 

\subsection{}
Let $K$ be a valued field. A  $K$-variety $X$ (\ie, a separated $K$-scheme of finite type) will be here mainly considered as a functor
on the category of  models of $\acvf$ ``containing" $K$ (\ie, equipped with an isometric embedding from
$K$), or sometimes even implicitly containing a given valued extension of $K$. 
By a point of $X$ we will mean 
an element of $X(F)$ for such a model $F$, which will in general not be specified, and could be enlarged if needed. 
(The reader who likes the model-theoretic version of Weil's viewpoint might also consider that one has fixed a huge ``monster model"
$\mathscr U$ of $\acvf$, that all objects considered in this paper are embedded in $\mathscr U$, and that points of $X$ are simply
elements of $X(\mathscr U)$.) If we need to consider scheme-theoretic points or types in $\acvf$ rather than 
``naive" points living on some model, 
we will say so explicitly. 

The ring of regular functions on the $K$-scheme $X$ will usually be denoted by $K[X]$; if the scheme
$X$ is integral, its field of rational functions will be denoted by $K(X)$.

\subsection{}
For seing definable sets as functors one needs in general to restrict the evaluation to models of $\acvf$, to ensure that one only 
deals with elementary embedding of valued fields. But definable sets coming with a canonical quantifier-free description
can be seen as functors on the category of all valued fields, or valued extensions of a given ground field. For instance 
$\k, \Gamma$ and $\RV$ extend to the category of all
valued fields, and if $X$ is a $K$-variety, 
it can be seen as a functor on the category of all valued extensions of $K$. 

\subsection{}
In general, our sets of parameters for definable subsets of a $K$-variety $X$
will consist of elements of one of the following sorts: the ground field, the value group or $\RV$ (remind that the latter
contains the residue field). 
This means that such a set $A$
will be contained in $F\coprod \abs F\coprod\RV(F)$ for some (not necessarily specified) valued extension 
$F$ of $K$, and we shall say for short that a valued extension $M$ of $K$ contains $A$
if it contains the smallest subfield $E$ of $F$ containing $K$
and such that 
$A\subset E\coprod\abs  E\coprod \RV(E)$.
For such an $A$, an $A$-definable subset $D$ of $X$ will be understood as a functor
on the category of algebraically closed, non-trivially valued extensions of $K$  containing $A$ -- or on all
valued extensions of $K$ containing $A$ if $D$ admits a canonical quantifier-free description. 

An $A$-definable subset $E$ of $X$ will be called \textit{$A$-atomic} if the only $A$-definable subsets of $E$
are $\emptyset$ and $E$. 

\subsection{}
Let $K$ be a valued field and let $F$ be 
a valued extension of $K$. Then $\RV(F)$ is a (generalized) 
field extension of $\RV(K)$
and 
\[[\RV(F):\RV(K)]=(\abs{F\gpm}:\abs{K\gpm})\cdot[\k(F):\k(K)],\]
and $F$ is an immediate extension of $K$ if and only if 
$\RV(F)=\RV(K)$. 
If $(a_i)$ is a family of elements of $F\gpm$, the $\rv(a_i)$ are linearly independent over $\RV(K)$ if and only if 
$\abs{\sum \lambda_ia_i}=\sup_i \abs{\lambda_i}\cdot \abs{a_i}$ for all (finitely supported) families $(\lambda_i)$ of elements of 
$K$. If this is the case, the $a_i$ are linearly independent over $K$. We thus see that $[\RV(F):\RV(K)]\leq [F:K]$. 

Now denote by $L$ a finite extension of $K$ and by $L_1,\ldots, L_r$ all the valued extensions of $K$ with 
underlying field $L$. By the above we have for all $i$ the inequality 
\[[\RV(L_i):\RV(K)]=[\RV(L_i\h):\RV(K\h)]\leq [L_i\h:K\h]\]
where $(\cdot)\h$ stands for the henselization. It follows that 
\[\sum_i [\RV(L_i):\RV(K)]\leq \sum_i [L_i\h:K\h]=[(L\otimes_K K\h):K\h]
=[L:K].\]
The extension $K\hookrightarrow L$ is called \textit{defectless} if $\sum_i [\RV(L_i):\RV(K)]=[L:K]$, and a valued field
is called defectless\footnote{Such valued fields are also often called \emph{stable}, but to avoid any confusion with the model-theoretic notion
of stability, we prefer to use ``defectless" in this article.}
if all its finite extensions are defectless. Algebraically closed
valued fields, valued fields with residue characteristic zero, complete discretely valued
fields, function fields of normal irreducible varieties equipped with the discrete valuation associated to a divisor, 
are defectless. A valued field is defectless if and only if its henselization is defectless. 

\subsection{}
Let $K\hookrightarrow F$ be an extension of valued fields. If $(a_i)$ is 
a family of elements of $F^\times$, the $\rv(a_i)$ are algebraically 
independent over $\RV(K)$ if and only if 
$\abs{\sum \lambda_Ia^I}=\sup_i \abs{\lambda_I}\cdot \abs a^I$ for all polynomials $\sum \lambda_I T^I$
in the multi-variable $T=(T_i)$. If this is the case, the $a_i$ are algebraically
independent over $K$. We thus see that $\td(\RV(F)/\RV(K))
\leq \td(F/K)$. 
We say that $F$ is an \textit{Abhyankar extension} of $K$ if it is finitely generated over $K$ and 
 $\td(\RV(F)/\RV(K))
=\td(F/K)$. If this is the case, 
an \textit{Abhyankar basis} of $F$ over $K$ will be any family $(a_i)$ of elements of $F\gpm$ such that 
the $\rv(a_i)$ are a transcendence basis of $\RV(F)$ over $\RV(K)$. Otherwise said, this is a family $(a_i)$ of elements of 
$F\gpm$ whose cardinality is equal to $\td(F/K)$ and such that $\abs{\sum \lambda_Ia^I}
=\sup_i \abs{\lambda_I}\cdot \abs a^I$ for all polynomials $\sum \lambda_I T^I$. 

An important result of valuation theory, which plays a key role in this manuscript,  asserts that if $K$ is a defectless field,
every Abhyankar valued extension
of $K$ is still defectless. Its first proof was given by Franz-Viktor Kuhlmann \cite{kuhlmann2010};
a purely model-theoretic proof can be found in the Appendix of \cite{ducros-h-l-y2024}.

\begin{exem}
Let $K$ be a valued field, let $G$ be an ordered abelian group containing $\abs{K^\times}$, and let 
$g=(g_1,\ldots, g_n)\in G^n$. Let $\eta_g$ be the valuation 
on $K(T_1,\ldots, T_n)$ that maps $\sum a_I T^I$ to $\max \abs{a_I} g^I$. 
Then $(K(T),\eta_g)$ is an Abhyankar extension of $K$, and 
$\RV(K(T),\eta_g)$ is the rational function field over $\RV(K)$ generated by 
he $\rv(T_i)$, which are algebraically independent
over $K$. We call $\eta_g$ the \textit{Gauss valuation} 
with parameter $g$.  
\end{exem}

\subsection{Generalized varieties}\label{generalized-variety}
We shall need the notion of 
a \textit{generalized algebraic variety}, see \cite[\S 4.1]{hrushovski-k2006}; this is essentially the $\RV$ (or graded)
version of a classical variery over the residue field.
In particular for every finite tuple $\gamma=(\gamma_1,\ldots, \gamma_e)$ of elements of $\Gamma$
one can define the generalized affine space $\A^\gamma_\RV $; this is simply the 
$\gamma$-definable subset of $\RV^e$ consisting 
of those elements $(x_1,\ldots, x_e)$ with each $x_i$ of degree $\gamma_i$.
There is a projective version of this construction: $\P^\gamma_\RV$ is the set of 
tuples $(x_0,\ldots, x_e)$ of elements of $\RV$ with $x_0$ of degre $1$ and $x_i$ of degree $\gamma_i$
for $i>0$, modulo multiplication by $\k(K)^\times$. 
Note that if $\gamma=(1,\ldots, 1)$ then $\A^\gamma_\RV$ is nothing but the classical affine 
space $\A^e_\k$ over the residue field;
more generally if $\gamma_i=\abs{\lambda_i}$ for some family
$\lambda_i$ of elements of a valued field $K$ then 
$\A^\gamma_\RV$ is $K$-definably isomorphic to $\A^e_\k$ using the multiplication by $(\rv(\lambda_i^{-1})_i$.

Using the notion of a generalized polynomial we can define
generalized
algebraic  functions from a generalized algebraic variety to 
$\RV$. 

\subsection{Stable completions}
We will use freely the theory of \textit{stable completions}, developped by Hrushovski and
the second author in \cite{hrushovski-l2016}. We will need it only for curves, for which it
is far simpler than for higher dimensional varieties: indeed if $K$
is a valued field and $V$ is a $K$-definable 
subset of some algebraic $K$-variety, 
the stable completion $\widehat V$ (which is a functor on the category of all models
of $\acvf$ containing $K$) is $K$-definable as soon as $\dim V\leq 1$, while it is merely
pro-definable if $\dim V>1$. More precisely, the references inside
\cite{hrushovski-l2016} are  \S3.1 and Theorem 3.1.1
for the definition and the (strict) pro-definability of $\widehat V$, Theorem 7.1.1 for the
definability in dimension $\leq 1$ (which ultimately
relies on Riemann-Roch Theorem), and Remark 7.1.3 
for the fact that one cannot hope for definability in dimension $>1$. 

In fact $\widehat V$ is not only a pro-definable set but a pro-definable topological space: for every model 
$M$ of $\acvf$ containing $K$ the set $\widehat V(M)$ is equipped with a natural topology admitting a basis of relatively
$M$-definable open subsets, inducing the valuative topology 
on $V(M)$, and for which strict inequalities give rise to open subsets. 
It therefore makes sense to say that some pro-definable subset of
$\widehat V$ is definably closed (\resp open, \resp compact, \resp connected).
We shall also use the relative version of definable compactness, that of definable properness: a pro-definable 
map between pro-definable topological spaces will be called definably proper 
if the pre-image of every definably compact subset is definably compact.

Since we shall always use
the definable variants of the usual topological notions in this setting, we will allow ourselves to omit ``definably"
when refering to them, so we will say compact for definably compact, proper for definably proper, etc.

\section{Topology of curves: admissible skeletons}\label{topology-curve}
%
We fix a valued field $K$ and an algebraic closure $\overline K$ of $K$, endowed with an extension of the valuation
of $K$.

\subsection{}
Let $X$ be an algebraic $K$-variety. We shall denote by $\breve X$ the functor that maps a model $M$ 
of $\acvf$ containing $\overline K$ to the set of $M$-definable types over $M$ on $X$. Note that $\breve X$ contains $\widehat X$
by definition. More precisely, 
$\widehat X(M)$ is the set of $x\in \breve X(M)$ such that for every $M$-Zariski open subset $U$ of $X$ with $x\in \breve U(M)$ and every 
$M$-regular function $f$ on $U$ there exists $\lambda\in M$ with $\abs{f(x)}=\abs \lambda$. 
There is a natural intermediate set $\widetilde X$ between $\widehat X$ and $\breve X$, that of \textit{bounded} definable types: $\widetilde
X(M)$ is the set of $x\in \breve X(M)$ such that for every $M$-Zariski open subset $U$ of $X$ with $x\in \breve U(M)$ and every 
$M$-regular function $f$ on $U$ there exists $\lambda\in M$ with $\abs{f(x)}\leq \abs \lambda$.
While $X\mapsto \widehat X$ is a model-theoretic avatar of Berkovich's analytification, $X\mapsto \widetilde X$ is that of Huber's.
It follows from \cite[Theorem 7.4.3]{cubides-h-y2021}
that $\breve X$ and $\widetilde X$ are strictly pro-definable.

\subsubsection{}\label{defsigma}
The embedding $\widehat X\subset \widetilde X$ admits a retraction
$\varpi$. If one sees a point of $\breve X(M)$ as a scheme-theoretic
point $\xi$ on the $M$-scheme $X_M$ together with an $M$-valuation $v$ on the field $M(\xi)$, then $\varpi$
maps
$(\xi,v)$ to $(\xi,v')$ where $v'$ is the coarsening of $v$ obtained by modding out $v(M(\xi)^\times)$ by its convex subgroup
consisting of the elements infinitesimally close to $1$ with respect to $\abs{M^\times}$. 

\subsubsection{}
The map $\varpi$ is pro-definable. Indeed, this follows from the general ``explicit" description of a set of uniformly definable types
as a pro-definable set \cite[Prop. 4.1]{cubides-y2021}, and the following
facts (for $M$ a model of $\acvf$ containing $K$): 

\begin{itemize}
\item [$\diamond$] A type $x$ on $X$  over $M$ (in $\acvf$) is entirely determined once one knows the set of all $M$-Zariski open subsets $U$ of $X$ such that $x$ lies on $U$ and, for every such $U$, the set of all $M$-regular functions $f$ on $U$ such that $\abs{f(x)}\leq 1$
and that of all  $M$-regular functions $f$ on $U$ such that $\abs{f(x)}<1$. 
\item[$\diamond$] Let $U$ be an $M$-Zariski open subset of $X$ and let $x$ be a point of $\widetilde X(M)$ with image $x'$ on $\widehat X(M)$. 
The point $x$ lies on $U$ if and only if $x'$ lies on $U$. If this is the case then for every $M$-regular function $f$ on $U$
one has: 

\begin{itemize}
\item[$\bullet$] $\abs{f(x')}\leq 1$ if and only if $\abs{\lambda f(x)}<1$ for all $\lambda \in M$ with $\abs \lambda <1$; 
\item[$\bullet$]  $\abs{f(x')}<1$ if and only if there exists $\lambda \in M$ with $\abs \lambda >1$ and $\abs{\lambda f(x)}<1$.
\end{itemize}
\end{itemize}

\subsubsection{}
If $U$ is a definable subset of $X$, we shall denote by $\breve U, \widetilde U$ and $\widehat U$ the subsets of
$\breve X, \widehat X$ respectively consisting of those types that lie on $U$. 
We warn the reader that there is no general inclusion relation between $\widetilde U$ and 
$\varpi^{-1}(\widehat U)$
in $\widetilde X$. Indeed, let us take for $X$ the affine line, for $U$ the open unit ball $\abs T<1$ and for $V$ 
its complement $\abs T\geq 1$. 
Then using the notation introduced at the beginning of \ref{projective-line-explicit} below 
$\eta_{0,1^-}$ belongs to $\widetilde U$ but $\varpi(\eta_{0,1^-})=\eta_{0,1}$ does not lie on $\widehat U$; 
so $\eta_{0,1^-}$ does not belong to $\widetilde V$ but $\varpi(\eta_{0,1^-})$ belongs to $\widehat V$.

\subsection{The case of the
projective line}\label{projective-line-explicit}
Our purpose is to show explicitly that
the pro-definable sets $\breve {\P^1}$ and $\widetilde {\P^1}$ are actually definable, like $\widehat {\P^1}$. 
We fix a model
$M$ of $\acvf$ containing $K$, we denote by $R$ the valuation ring of $M$ and by $\mathfrak m$
its maximal ideal,  and
we set $E=\{a+bT\}_{a,b\in M}\subset M[T]$.  

Let $a$ be an element of $M$. For every $r$ in an ordered abelian group containing
$\abs{M^\times}$ 
we denote by $\eta_{a,r}$ the type over $M$ lying on $\P^1$ and given by the Gauss norm
$\sum a_i(T-a)^i\mapsto \max \abs{a_i}r^i$ on $M[T]$. One has $\eta_{a,r}=\eta_{a,s}$ if and
only if $r$ and $s$ have the same type over $\abs{M^\times}$
and $\abs{a-b}\leq r$. The type $\eta_{a,r}$ 
is $M$-definable
if and only if the type of $r$ over $\abs{M^\times}$ is $\abs{M^\times}$-definable. 
This occurs if and only if one (and only one) of the following hold: 
\begin{itemize}[label=$\diamond$]
\item $r\in \abs{M^\times}$; 
\item $r$ is infinitesimally lower (with respect to
$\abs{M^\times}$) than some $\gamma\in \abs{M^\times}$; then we also write $\eta_{a,\gamma^-}$ for the type
$\eta_{a,r}$; 
\item $r$ is infinitesimally larger than some $\gamma\in \abs{M^\times}$; then we also write $\eta_{a,\gamma^+}$ for the type
$\eta_{a,r}$; 
\item $r$ is infinitely small; then we also write $\eta_{a,0^+}$ for the type $\eta_{a,r}$; 
\item $r$ is infinitely large; then we also write  $\eta_{a,\infty^-}$ for the type $\eta_{a,r}$
(note that $\eta_{a,\infty^-}$ is
equal to $\eta_{0,\infty^-}$). 
\end{itemize}

\subsubsection{Reminders on $\widehat {\P^1}$}
\label{reminders-xhat}
Let us first remind the explicit description of $\widehat {\P^1}$ as a definable set: there is 
a natural bijection from $\widehat \P^1(M)$ to 
the definable set $\Lambda$ of $M$-definable sub-modules
$E$ of the form $R\oplus S(T-a)/\lambda$ with $S=0,R$ or $M$, with
$a\in M$ and with $\lambda \in M^\times$ (for the definability of $\Lambda$ to
make sense, one needs to use the so-called
geometric sorts, at least $S_2$). This natural bijection maps $x\in \widehat {\P^1}(M)$ to $L_x:=\{f\in E, \abs{f(x)}\leq 1\}$; the converse bijection maps 
$L_x$ to $\eta_{a,\abs \lambda }$ if $L_x$ is a lattice
 and $(1,(T-a)/\lambda)$ is a basis of $L_x$,
 to the simple point given by the equation $T=a$ if $M(T-a)\subset L_x$ (this determines $a$ uniquely) and to $\infty$ if $L_x=R$. 

\subsubsection{First description of $\breve {\P^1}$ and 
$\widetilde {\P^1}$}
An $M$-type $x$ lying on $\P^1$ which does not belong to $\widehat \P^1(M)$ is either realized by some element of $M'\setminus M$ for $M'$
an immediate extension of $M$, either of the form $\eta_{a,r}$ for some $r$ in $G\setminus \abs{M^\times}$
for some ordered group $G\supset \abs{M^\times}$; but in the first
case the type $x$ is not definable (for the set of $f\in E$ such that $\abs{f(x)}\leq 1$ is not definable). We thus see that 
$\breve {\P^1}(M)\setminus \widehat {\P^1}(M)$ consists of the types of the form
$\eta_{a,\gamma^-}, \eta_{a,\gamma^+}, \eta_{a,0^+}$ and $\eta_{0,\infty^-}$ for $a\in M$ and $\gamma\in \abs{M^\times}$; 
among these types, only those of the form $\eta_{a,\gamma^-}$ and $\eta_{a,\gamma^+}$ are bounded, 
so $\widetilde {\P^1}(M)\setminus \widehat {\P^1}(M)$ consists of all types of the form $\eta_{a,\gamma^-}$ and $\eta_{a,\gamma^+}$.

\subsubsection{Definability of $\breve {\P^1}$ and $\widetilde {\P^1}$}
Let $x\in \breve {\P^1}(M)$. Let $L_x$ be the sub-$R$-module $\{f\in E, \abs{\lambda f(x)}<1$ for all $\lambda\in M$ with $\abs \lambda <1\}$
of $E$.
Let us describe $L_x$ more explicitly. 

\begin{itemize}
\item[$\diamond$] Assume that $x\in \widehat {\P^1}(M)$. Then $L_x$ coincides with the module
$L_x$ defined in \ref{reminders-xhat}. 

\item[$\diamond$] If $x=\eta_{a,\gamma^-}$ or $\eta_{a,\gamma^+}$ 
then $L_x=R\oplus R (T-a)/\lambda$ where $\lambda$ is any element of $M$ with $\abs \lambda=\gamma$
(this is also the lattice associated to
$\eta_{a,\gamma}$). 
In this case denote by $L'_x$ the subset of $L_x$ consisting of those $f$ such that $\abs{f(x)}<1$ and $\abs{\lambda f(x)}>1$ for every
$\lambda\in M$ with $\abs \lambda >1$. Then $L'_x=R(T-a)/\lambda+\mathfrak m_xL_x$ if $x=\eta_{a,\gamma^-}$, and $L'_x=\emptyset$
if $r=\eta_{a,\gamma^+}$. 

\item[$\diamond$] If $x=\eta_{a,0^+}$ then $L_x=R\oplus M(T-a)$
(this is also the lattice associated to the simple point $T=a$);

\item[$\diamond$] If $x=\eta_{0,\infty^-}$ then $L_x=R$ (this is also 
the lattice associated to $\infty$). 
\end{itemize}

The discussion above furnishes a pro-definable partition 
$\breve {\P^1}=\widehat {\P^1}\coprod \Omega\coprod \Omega'$ where: 

\begin{itemize}
\item[$\diamond$] $\Omega(M)$ is the set of types of the form $\eta_{a,\gamma^-}$ or $\eta_{a,\gamma^+}$. It
is in pro-definable 
bijection with the definable set of pairs $(L, D)$ where $L$ is a lattice of $E$ admitting $R$ as a direct summand  and $D$ 
is an element of $\mathbb P(L/\mathfrak mL)$ (note that for this set of pairs to be definable, one needs the geometric sorts $S_2$ and $T_2$). 
The bijection maps $x$ to $(L_x, L'_x/\mathfrak mL_x)$ if $L'_x\neq \emptyset$ and
to $(L_x, (R+\mathfrak mL_x)/\mathfrak mL_x)$ otherwise. The converse bijection 
maps a pair $(L,D)$ to $\eta_{a,\abs \lambda ^-}$ if $D\neq (R+\mathfrak mL_x)/\mathfrak mL_x$ and $(T-a)/\lambda$
lifts a generator of $D$,
and to $\eta_{a,\abs \lambda ^+}$ if $D=( R+\mathfrak m L_x)\mathfrak mL_x$ and $(1,(T-a)/\lambda)$ is a basis of $L_x$.

\item[$\diamond$] $\Omega'(M)$ is the set of types of the form $\eta_{a,0^+}$ or $\eta_{0,\infty^-}$. It is in pro-definable
bijection with the definable set $\P^1(M)$ by the map 
$\eta_{a,0^+}\mapsto a$ and $\eta_{0,\infty^-}\mapsto
\infty$. 
\end{itemize}

It follows that the strictly pro-definable set $\breve {\P^1}$ is actually definable, and so is $\widetilde {\P^1}=\widehat {\P^1}\coprod \Omega$. 
The natural retraction from $\widetilde {\P^1}$ to $\widehat {\P^1}$ is then definable. Explicitly if $x\in \widehat {\P^1}(M)$ it maps $x$ to
$x$; and if $x\in \Omega(M)$ it maps $x$ to the unique point $x'$ of $\widehat {\P^1(M)}$ such that $L_{x'}=L_x$, which can be rephrased by saying 
that it maps $\eta_{a,\gamma^\pm}$ to $\eta_{a,\gamma}$.

\subsubsection{The $\RV$-skeleton of $\widetilde  {\P^1}$}\label{rv-internal-A1}
We define $\Sigma(M)$ as the subset of  $\widehat {\P^1}(M)$
consisting of the origin and all points of the form $\eta_{0,\gamma}$ with $\gamma\in 
\abs{M^\times}$. The ``skeleton" $\Sigma$
is a definable subfunctor of $\widehat {\P^1}$ which is $\Gamma_0$-internal; in fact $\gamma\mapsto
\eta_{0,\gamma}$ (with the convention that $\eta_{0,0}=0$ and
$\eta_{0,\infty=}\infty$) establishes a definable homeomorphism $[0,\infty]\simeq
\Sigma$.

Let $\Sigma_\RV$ be the subset of $\widetilde {\P^1}$ consisting of 
points 
of the form $\eta_{a,\abs{a}^-}$
where $a$
lies in $\gm=\P^1\setminus \{0,\infty\}$. Note that for $a$ and $b$ in
$\gm$ we have $\eta_{a,\abs a^-}=\eta_{b,\abs b^-}$
if and only if $\abs a=\abs b$ and $\abs{a-b}<\abs a$, so we have a natural
$K$-definable bijection 
$\Sigma_\RV\simeq \RV^\times$. Note also that if $s\in \Sigma_\RV$ then
$s$ does not belong to
$\widehat {\P^1}$ and $\varpi(s)$ belongs to $\Sigma$: if $s=\eta_{a,\abs {a^-}}$ 
then $\varpi(s)=\eta_{a,\abs a}$.

Let $I$ be the interval $[0,1]$. 
Let $\breve I$ be the set of all definable types on $I$
(it contains $I$, $0^+$, $1^-$, and $r^+$ and $r^-$ for all
$0<r<1$)
and $\widetilde I$ that of ``bounded" such types, which is merely
$\breve I\setminus \{0^+\}$. 

Let $h\colon \P^1\times [0,1]\to \widehat {\P^1}$ be the $K$-definable map
$(a,t)\mapsto=\eta_{a,t\abs a}$. 
By applying $h$ on realizations of types we obtain maps 
\[\widehat h\colon \widehat {\P^1}\times I\to \widehat {\P^1},\; 
\widetilde h\colon \widetilde {\P^1}\times \widetilde I\to \widetilde {\P^1}\; 
\;\text{and}\;\breve h\colon \breve {\P^1}\times \breve I\to \breve {\P^1}.\]

The definable map $\widehat h$
is the unique continuous pro-definable extension 
of $h$ to $\widehat {\P^1}\times I$
(see \cite{hrushovski-l2016}, Lemma 3.8.5) 
and this is by a direct computation a strong deformation retraction from $\widehat {\P^1}$
to $\Sigma$. 
We set $\rho(a)=h(a,1)$ for every $a\in \P^1$, so $\rho(0)=0,\rho(\infty)=\infty$
and $\rho(a)=\eta_{a,\abs a}$ for $a
\in \gm$. By applying $\rho$ on realization of types we get maps
$\widehat \rho$ (\resp $\widetilde \rho$, \resp $\breve \rho$) from $\widehat {\P^1}$
(\resp $\widetilde {\P^1}$, \resp $\breve {\P^1}$) to itself; one has
$\widehat \rho(a)=\widehat h(a,1)$ for all $a\in \widehat {\P^1}$, and the analogous formulas
for $\widetilde {\P^1}$ and $\breve {\P^1}$. The image of $\widehat \rho$
is  still equal to $\Sigma$; the image of $\tilde{\rho}$ is
equal to $\widetilde \Sigma:=\Sigma \cup\{\eta_{0,\gamma^-}, \eta_{0,\gamma^+}\}_{\gamma \in \Gamma}$
and thet of $\breve \rho$ is equal to
$\breve \Sigma:=\Sigma\cup \{\eta_{0,\gamma^-}, \eta_{0,\gamma^+}\}_{\gamma \in \Gamma}\cup\{\eta_{0,0^+}, \eta_{0,\infty^-}\}$. 

There is another way to describe $\widehat \rho$: for every 
$x\in \widehat {\P^1}$ there is a unique (up to definable homeomorphism)
continuous injective path $p\colon I\to \widehat {\P^1}$ where $I=[o,e]$ is a generalized
interval in the sense of \cite[3.9]{hrushovski-l2016} such that $p(o)=x, p(e)\in S$, and $p(t)\notin \Sigma$
for $t\neq e$; then $\widehat\rho(x)=e$. 

Now let $\rr$ be the definable map from ${\P^1}$ to $\widetilde {\P^1}$
that maps $a$ to $h(a,1^-)$. 
Then $\rr(0)=0$ and $\rr(a)=\eta_{a,\abs a^-}$ for every $a\in \gm$. 
The image $\rr(\gm)=\rr({\P^1}\setminus \Sigma)$
is thus equal to $\Sigma_\RV$; 
observe that modulo the aforementionned 
bijection $\Sigma_\RV\simeq \RV^\times$ the map $\rr$ 
maps any $a\in \gm$
to $\rv(a)$. 
If $s\in \Sigma_\RV\setminus \{0,\infty\}$ the fiber $\rr^{-1}(s)$
is an open disc; more precisely if $s=\eta_{a,\abs a^-}$ then $\rr^{-1}(s)$
is the open disc with center $a$ and radius
$\abs a$; this is also $\rv^{-1}(\rv(a))$. 

By applying $\rr$ to realizations of types, 
we get maps 
\[\widehat \rr\colon \widehat {\P^1}\to \widetilde {\P^1}, \widetilde \rr\colon \widetilde
{\P^1}\to \widetilde {\P^1},\;\text{and}\;\breve \rr \colon \breve {\P^1}\to \breve {\P^1}.\]
Note that 
 $\wrr(\widehat {\P^1}\setminus \Sigma)$
 is  still 
equal to $\Sigma_\RV$, but $\wrr(\widehat {\P^1})$ 
is equal to $\Sigma_\RV\cup \Sigma$, while $\widetilde \rr(\widetilde {\P^1})$
and $\breve \rr(\breve {\P^1})$ are respectively equal to
$\Sigma_\RV\cup \widetilde \Sigma$ and $\Sigma_\RV\cup \breve \Sigma$.

For defining $\rr$ we have used an explicit formula involving a specific
deformation retraction from $\widehat {\P^1}$ to $\Sigma$, 
but it could have been done in a more intrinsic way
(though less convenient for seeing that $\rr$ is definable)
analogous
to the direct definition of $\rho$ given above. Indeed let
$x\in {\P^1}$ and  let $p\colon [o,e]\to \widehat {\P^1}$ be the unique injective path (up to re-parametrization)
such that $p(o)=x, p(e)\in \Sigma$, and $p(t)\notin \Sigma$
for $t\neq e$. Then $\rr(x)=x$ if $o=e$, \ie , if $x\in \Sigma$, and $\rr(x)=p(e^-)$ otherwise.

\begin{defi}\label{defi-admissible-skeleton}
Let $X$ be a $K$-algebraic curve 
(\ie, a $K$-algebraic variety of pure dimension 1)
and let $U$ be a $K$-definable subset of $X$. 
A \emph{skeleton} of $\widehat U$ is
 a $\Gamma_0$-internal 
subset of $\widehat U$; a $K$-skeleton
is a $K$-definable skeleton. 
A $K$-skeleton $\Sigma$ of $\widehat U$ is called
\emph{admissible}
if there exists a $K$-definable
strong deformation retraction $\widehat h\colon \widehat U\times I\to \widehat U$
with $\widehat h(\widehat U,e)=\Sigma$ for $e$ the right endpoint of $I$, induced (through the realization of types)
by its restriction 
$h\colon U\times I\to \widehat U$ to simple points
(we will call for short such an $\widehat h$ an admissible $K$-homotopy with final image $\Sigma$). 
\end{defi}

\begin{rema}
Admissibility is a property of $K$-skeletons, and its definition involves $K$-definability. 
So if $L$ is a valued extension of $K$, the property for a $K$-skeleton of
being admissible
as an $L$-skeleton is a priori weaker than being admissible as an $K$-skeleton; but it turns out
that both are in fact equivalent, see Remark \ref{admissibility-K-L} below. 
\end{rema}

\subsection{Existence of admissible skeletons and consequences}
Let $X$
and $U$ be as in the above definition. The following facts 
follow from \cite[Chapter 7]{hrushovski-l2016}
(they are extended to higher dimensional varieties in \opcit, Theorem 11.1.1): 
every $K$-skeleton $\Tau$ of $\widehat U$ is contained in some admissible
$K$-skeleton
$\Sigma$; one can morever require that $\Sigma$ be purely $1$-dimensional, 
that it be the target of a \emph{topologically proper} admissible
$K$-homotopy $h$, 
and that finitely many given $K$-definable maps from $U$ to $\Gamma_0$ be constant along
the trajectories of $h$.

\subsubsection{}\label{basic-topology-uhat}
The existence of admissible $K$-skeletons ensures that $\widehat U$ is locally path-connected, and
even locally contractible. It also ensures that $\widehat U$ has finitely many
(path-)connected components, each of which is
the stable completion of some
$\overline K$-definable subset of $U$; 
more precisely, 
if $\widehat h\colon \widehat U\times I\to \widehat U$ is an admissible $K$-homotopy 
with target $\Sigma$, and if $e$ is the right endpoint of $I$, the connected components of $\widehat U$
are exactly the sets fo the form $\widehat h(\cdot, e)^{-1}(\Tau)$ for $\Tau$ a connected
component of $\Sigma$, and we conclude by observing that 
$\widehat h(\cdot, e)^{-1}(\Tau)$ is the stable completion of $h(\cdot, e)^{-1}(\Tau)$
(and as $\pi_0(\Sigma)$ is finite, each connected component of $\Sigma$ is $\overline K$-definable). 

\subsubsection{}\label{admissible-loops}
Every skeleton on $\widehat U$ is of dimension $\leq 1$; this prevents any (definable)
loop on $\widehat U$ from moving along any homotopy, so that an admissible $K$-skeleton necessarily contains all
loops of $\widehat U$ (there are therefore only finitey many such loops). 
Also note that any admissible $K$-skeleton of $\widehat U$ is closed in $\widehat U$ (as the set
of points that are fixed at every time under a 
$K$-admissible retraction) and that its intersection with every connected component of
$\widehat U$
is non-empty and connected.

Conversely, let $\Sigma$ be
a closed $K$-skeleton of $\widehat U$ 
containing all loops on $\widehat U$ and whose intersection with every connected component of $\widehat U$ 
is non-empty and connected. Choose  an
admissible $K$-skeleton $\Tau$ of $\widehat U$
containing $\Sigma$. Then $\Sigma$ is closed
in $\Tau$, it contains all the loops of $\Tau$, 
and the intersection of $\Sigma$ with every connected component of $\Tau$ is non-empty and connected. 
This implies that there exists a $K$-definable strong deformation retraction from $\Tau$ to $\Sigma$
(first build a $\overline K$-definable such retraction $h(\cdot, \cdot)$ and then 
define $h'(x,t)$ as the the closest point to $x$ among all $\sigma^{-1}(h(\sigma(x),t))$ for
$\sigma\in \mathrm{Aut}(\overline K/K)$); it follows that 
$\Sigma$ itself is admissible (just retract $\widehat X$ to $\Tau$, and then $\Tau$ to $\Sigma$). 

It follows that if $\Sigma$ is an admissible $K$-skeleton of $\widehat U$, a $K$-definable skeleton
$\Tau$ of $\widehat U$ 
containing $\Sigma$ is admissible if and only if the natural map form $\pi_0(\Sigma)$ to $\pi_0(\Tau)$
is bijective. As $\Sigma$ contains all loops of $\widehat U$ this amounts to requiring that $\Tau$ be built
from $\Sigma$ by successive applications of the operation consisting in adding to a $\overline K$-definable
skeleton $\Upsilon$ a $\overline K$-definable interval $[x,y]$ with $[x,y]\cap \Upsilon=\{y\}$ (we
insist that even if $\Sigma$
and $\Tau$ are $K$-definable, the successive edges one adds might be merely $\overline K$-definable).

\subsubsection{}
An admissible $K$-skeleton $\Sigma$ of $\widehat U$ has a strong convexity
property:  for all $(x,y)\in \Sigma^2$ and all compact intervals $I$ on $\widehat U$
with endpoints $x$ and $y$ one has $I\subset \Sigma$. Indeed, if such an interval exists then $x$ and $y$ are on the 
same connected component of $\widehat U$, and therefore of $\Sigma$, so there is at least one such interval entirely
contained in $\Sigma$; but then all of them are conained in $\Sigma$, since $\Sigma$ contain all the loops drawn on $\widehat U$. 

\subsubsection{}
Let us also mention that some GAGA principle holds here: if $X$ is geometrically
connected as an algebraic $K$-curve, 
$\widehat X$ is connected (\cite{hrushovski-l2016}, Theorem 10/4.2; this is stated and shown for an
arbitrary $X$, but the proof almost immediately reduces through Bertini's theorem to the case of a curve). 

\begin{rema}\label{admissibility-K-L}
Let $\Sigma$ be a $K$-skeleton and let $L$ be an arbitrary valued extension of $K$. It follows from the characterization 
of admissibility given in \ref{admissible-loops}
that $\Sigma$ is admissible as a $K$-skeleton if and only if it is admissible as an $L$-skeleton (though the latter
condition seems weaker, for it involves merely an $L$-definable retraction). 
\end{rema}

\subsection{The canonical map to a $K$-admissible skeleton}
\label{rho}
Let $\Sigma$ be an admissible $K$-skeleton on $\widehat U$.

\subsubsection{}
Let $x\in \widehat U$. 
The existence of an admissible $K$-homotopy with final image $\Sigma$ ensures the existence of 
a generalized interval $[x,y]$ on $\widehat U$ 
with $[x,y]\cap \Sigma=\{y\}$, and this interval is unique because the intersection of $\Sigma$
with the connected component of $x$ is connected and because all loops of $\widehat U$ are contained
in $\Sigma$.

The latter property in fact characterizes the admissibility.
Indeed, let $\Tau$ be a closed
$K$-skeleton of $\widehat U$ such that for every $x\in \widehat U$ there exists 
a unique interval $[x,y]$ on $\widehat U$ with $[x,y]\cap \Tau=\{y\}$ it is easily seen that
the intersection of $\Tau$ with every connected component of $\widehat U$ is connected and non-empty,
and that $\Tau$ contains 
all loops drawn on $\widehat U$, so $\Tau$ is admissible.

\subsubsection{}\label{rho-open-fibers}
By the above there exists for
every $x\in \widehat U$ 
a unique segment $[x,\widehat\rho_{U,\Sigma}(x)]$ on $\widehat U$ such that 
$[x,\widehat \rho_{U,\Sigma}(x)]\cap \Sigma=\{\widehat \rho_{U,\Sigma}(x)\}$. 
The map $\widehat \rho_{U,\Sigma}$ is $K$-definable: indeed, 
this is the final map associated to any admissible $K$-homotopy with
final image $\Sigma$; this also implies that 
$\widehat \rho_{U,\Sigma}$ is induced by
its restriction $\rho_{U,\Sigma}$ to $U$. As a consequence, for every $K$-definable subset
$D$ of $\Sigma$, the pre-image
$\widehat \rho^{-1}_{U,\Sigma}(\Sigma)$ is the stable completion of the 
$K$-definable set $\rho_{U,\Sigma}^{-1}(D)$. 

\paragraph{}\label{puncutred-preimage-open}
Let $s\in \Sigma$ and let $x$ be a point of $\widehat \rho_{U,\Sigma}^{-1}(s)
\setminus \Sigma=\widehat \rho_{U,\Sigma}^{-1}(s)
\setminus \{s\}$. Let $\Omega$ be a path-connected neighborhood of $x$ in 
$\widehat U$ that does not meet $\Sigma$. Let $y\in \Omega$. There is an interval
$[x,y]$ on $\Omega$ (which is unique since $\Sigma$ contains all loops on $\widehat U$). 
Then $[x,y]\cup[x,s]$ is connected and meets $\Sigma$ only on $s$. It follows that 
$\widehat \rho_{U,\Sigma}(y)=s$; hence
$\widehat \rho_{U,\Sigma}^{-1}(s)\setminus \{s\}$ is open in 
$\widehat U$. 

\paragraph{}
Let $\Tau$ be an admissible $K$-skeleton of
$\widehat U$ containing $\Sigma$ and let $x\in\widehat U$. It follows from the 
construction that $\widehat \rho_{U,\Tau}(x)\in [x,\widehat \rho_{U,\Sigma}(x)]$, and 
$[\widehat \rho_{U,\Tau}(x),\rho_{U,\Sigma}(x)]$
is contained in $\Tau$ since $\Tau$ is admissible. So
$\widehat \rho_{U,\Sigma}=\widehat \rho_{U,\Sigma}|_\Tau\circ \widehat \rho_{U,\Tau}$; note that 
$\widehat \rho_{U,\Sigma}|_\Tau$ can be described purely in terms of $\Tau$ and $\Sigma$, 
without any reference to $U$: for each $t\in \Tau$
the interval $[t,\widehat \rho_{U,\Sigma}(t)]$ is contained in $\Tau$,  its intersection
with $\Sigma$ is equal to $\{\widehat \rho_{U,\Sigma}(t)\}$, and this determines uniquely 
 $\widehat \rho_{U,\Sigma}(t)$.

\begin{defi}\label{skeleton-adapted}
Let $V$ be a $K$-definable
subset of $U$. 
An admissible $K$-skeleton $\Sigma$
of $\widehat U$ is 
\emph{adapted to $V$} if there exists a $K$-definable subset $D$
of 
the skeleton $\Sigma$ such that $V=\rho_{U,\Sigma}^{-1}(D)$, which amounts
to saying that $\widehat V=\widehat \rho_{U,\Sigma}^{-1}(D)$. 
\end{defi}

\begin{rema}
Assume
that $\Sigma$
is adapted to $V$, and let $D$ be as in the definition. 
The set $D$ is uniquely determined: it is necessarily equal
to $\rho_{U,\Sigma}(V)=\widehat\rho_{U,\Sigma}(\widehat V)$, and 
$\widehat V$ is stable under any admissible $K$-homotopy of 
$\widehat U$ with final image $\Sigma$, so $D$ is an admissible 
$K$-skeleton of $\widehat V$. 

Note also that any 
admissible $K$-skeleton $\Tau$ of $\widehat U$ containing
$\Sigma$ is still adapted to $V$: indeed, for such a $\Tau$ one has
$\rho_{U,\Sigma}=\widehat \rho_{U,\Sigma}|_\Tau\circ \rho_{U,\Tau}$ 
(by \ref{puncutred-preimage-open}), so if $V=\rho_{U,\Sigma}^{-1}(D)$
then $V=\rho_{U,\Tau}^{-1}(\widehat \rho_{U,\Sigma}^{-1}(D)\cap \Tau)$. 
\end{rema}

\begin{lemm}\label{exist-adapted}
Let $(V_i)$ be a finite family of $K$-definable subsets of $U$ and
let
$\Upsilon$ be $K$-skeleton of $\widehat U$. 
There exists 
an admissible $K$-skeleton $\Sigma$ of $\widehat U$ containing $\Upsilon$
and adapted to each $V_i$.
\end{lemm}

\begin{proof}
For each $i$ denote by $\chi_i$ the characteristic
function of $V_i$ in $U$, viewed as a map from $U$ to $\Gamma_0$. 
 Then there exists an admissible $K$-skeleton $\Sigma$
of $\widehat U$ containing 
$\Upsilon$ and an admissible $K$-homotpy $h$
from $\widehat U$ to $\Sigma$
such that each $\chi_i$ is constant along the trajectories of $h$; then 
$\Sigma$ is adapted by design to every $V_i$. 
\end{proof}

\begin{coro}\label{coro-boundary-uhat}
Let $V$ be a $K$-definable
subset of $U$. 
The topological boundary of $\widehat V$ in
$\widehat U$ is finite; it coincides with the topological 
boundary of $\rho_{U,\Sigma}(V)$ 
in $\Sigma$ for any
admissible $K$-skeleton $\Sigma$ of $\widehat U$ adapted to $V$.
\end{coro}

\begin{proof}
By the lemma above there exists an admissible
$K$-skeleton of $\widehat U$ adapted to $V$.
Set 
$D=\rho_{U,\Sigma}(V)$. 
We then have $V=\rho_{U,\Sigma}^{-1}(D)$ 
and $\widehat V=\widehat \rho_{U,\Sigma}^{-1}(D)$ (so in particular
$D=\widehat V\cap \Sigma$). 

Let $\partial D$ and $\mathring D$ denote the topological boundary of $D$ 
and the topological interior of $D$ in $\Sigma$; note
that $\partial D$ is finite. 
It is tautological 
that $\partial D$ is contained in the topological boundary $\partial 
\widehat V$ of $\widehat V$
in $\widehat U$; 
moreover $\widehat \rho_{U,\Sigma}^{-1}(\mathring D)$ is contained
in the topological interior of $\widehat U$ in $\widehat V$, and so is
$\widehat \rho_{U,\Sigma}^{-1}(d)\setminus \{d\}$ for every $d\in D$
by \ref{puncutred-preimage-open} above, so
$\partial \widehat V=\partial D$.\end{proof}

\section{Topology of curves: the complement
of an admissible skeleton}\label{topology-curve-complement}
%

\subsection{Germs of paths and residue curves}
Let $X$ be a $K$-algebraic curve and let $U$
be a $K$-definable open subset of $X$.

\subsubsection{}\label{type-of-types}
Let $M$ be a model of $\acvf$ containing $K$. 
Let $f\colon (a,b)\to \widehat U$ be a
continuous $M$-definable map 
where $a<b$ are two different elements of 
$\abs M$ with $f(a)=x$. 
Let $a^+$ be the $M$-definable type on $(a,b)$ that lies
on $(a,c)$ for all $c\in \abs M$ with $a<c<b$. 
Then $f(a^+)$ is an element of $\breve U(M)$ that only depends
on the ``left germ" of the parametrized path $f$, by which we refer
to the relation identifying $f\colon (a,b)\to \widehat U$ and $f'\colon (a',b')
\to \widehat U$ if  there are $a<c\leq b$ and $a'<c'\leq b'$ such that $f|_{(a,c)}$
and $f'|_{(a',c')}$ either are
constant with the same value or agree up to $M$-definable homeomorphic reparametrization. 
It follows from \cite[Theorem 2.1.15]{hrushovski-l2016}
that every element of $\breve U(M)$ arises in this way. 

\subsubsection{}\label{interval-type}
If $M$ and $f$ are as above, one is necessarily in one the
following three cases. 

\paragraph{The constant case} If the left-germ of $f$ is constant (\ie, it admits
a constant representative, allowing to assume that $f$ itself is constant), 
then $f(a^+)$ belongs to $\widehat U(M)$; it is equal to the constant value 
of $f$. 

\paragraph{The unbounded case}
If the left-germ of $f$ is not constant (allowing ourselves to assume that
$f$ is injective) and $a=0$ then $f(a^+)$ is unbounded, \ie, it belongs
to $\breve U(M)\setminus \widetilde U(M)$; if $\overline X$ denotes a projective
compactification of $X$ then $f(a^+)$ is the composition of the discrete valuation associated
to an $M$-point $\xi$ of the normalizaiton of $\overline X$ with the valuation 
of $M$. Let $x$ be the image of
$\xi$ in $\overline X(M)$; the type $f(a^+)$ has limit
$x$ in $\overline X(M)$; otherwise said, $f$ has a continuous extension 
to $[a,b)$ with $f(a)=x$. 

\paragraph{The bounded case}\label{interval-type-bounded}
If the  left-germ of $f$ is  not constant and $a\neq 0$ then 
$f(a^+)$ is an element of $\widetilde U$; if we set $x=\varpi(f(a^+))$ 
(the map $\varpi$ has been defined in \ref{defsigma}) the point $x$ 
is not simple and 
$f(a^+)$ has limit $x$; otherwise said, $f$ has a continuous extension to 
$[a,b)$ with $f(a)=x$. 

\subsubsection{The residue curve: the classical and the generalized viewpoints}
We still denote by $M$ a model of $\acvf$ containing
$K$. Let $x$ be a non-simple point of $\widehat X(M)$. 
The type $x$ is induced by a valuation $v$ (extending that of $M$) 
on the function field of some (reduced) irreducible component of 
the scheme $X_M$. 

\paragraph{}
The value group of $v$ is $\abs{M^\times}$
and it residue field is 
is $\k(M)(C_x)$ for some canonical smooth projective $C_x$ defined over
$\k(M)$, which we call the
\textit{residue curve} at $x$.

\paragraph{}
As the value group of $v$ is equal to that of $M$, the generalized residue field
of $v$ is then 
merely $\RV(M)\otimes_{\k(M)}\k(M)(C_x)$. We can 
also write it $\RV(M)(C_x)$, 
if we 
to implicitly interpret
$C_x$ as a generalized variety over $\RV(M)$
(\ref{generalized-variety}). This is harmless since this generalized
variety coincides
with $C_x$ as a geometric object, and this is often more natural
(see already Example \ref{generalized-residue-curve}
just below), especially if one wants to deal with families $(C_x)_x$ of residue curves
when $x$ runs through a skeleton. This viewpoint will also be convenient 
for the followig reason: if $f$ is an $M$-regular function invertible at $x$ then
$\abs{f(x)}$ belongs to $\abs{M^\times}$, so we can choose $a\in M^\times$ such that 
$\abs{f(x)}=\abs a$ and then consider the element 
$\res(f(x)/a)$ of $\k(M)(C_x)^\times$, interpreting it as a rational function on the 
classical algebraic curve $C_x$. But it depends on the choice of $a$, so this is more natural 
to deal with $\rv(f(x))\in \RV(M)(C_x)$ and interpret it as a generalized rational function 
on $C_x$ viewed as a generalized algebraic curve.

\begin{exem}\label{generalized-residue-curve}
Let $r\in \abs {M^\times}$. 
If $a$ denotes any element of $M$ with $\abs a=r$, the classical residue field
at $\eta_{0,r}$ (over $M$) can be identified to the rational function field
$\k(M)(\theta)$ with $\theta=\res((T/a)(\eta_{0,r}))$, so 
the residue curve $C$ at $\eta_{0,r}$  is isomorphic to
the projective
line over $\k(M)$ and one has 
\[C(M)=C(\k(M))\simeq \{0\}\cup \{\res(\lambda)\}_{\lambda \in M, \abs \lambda=1}\cup\{\infty\}\;;\]
but this isomorphism depends on the choice of the rescaling factor $a$.  
It is therefore more natural to deal with the generalized residue field at $\eta_{0,r}$ which 
is equal to $\RV(M)(\tau)$ with $\tau=\rv(T(\eta_{0,r}))$, so
the residue curve $C$ at $\eta_{0,r}$  seen as a generalized variety 
is isomorphic to
the generalized projective
line $\P_\RV^r$ (see \ref{generalized-variety}
for a definition) and one has
\[C(M)=C(\RV(M)))\simeq \{0\}\cup \{\rv(\lambda)\}_{\lambda \in M, \abs \lambda=r}\cup\{\infty\}\;;\]
of course one goes from one description of $C(M)$
to the other one by using $\mu \mapsto \mu/\rv(a)$. 
\end{exem}

\subsection{}
We fix for the remaining part of 
section \ref{topology-curve-complement}
a $K$-algebraic curve and 
a $K$-definable open subset $U$
of $X$.

\subsubsection{}
Let $x$ be a non-simple point of 
$\widehat  U$. By the above there is a natural bijection
between the set $\mathsf P_{x,\widehat X}(M)$ of germs of $M$-definable
paths on $\widehat X$ emanating from $x$ and 
$\varpi^{-1}(x)$. The latter is itself in 
bijection with the set of $M$-refinements
of $v$. Such a refinement is given by a valuation on the classical residue field
$\k(M)(C_x)$ trivial on $\k(M)$, and there are two kind of such valuations: the
trivial one (giving rise to the trivial refinement, \ie, $v$ 
itself) and the discrete valuations associated to 
the $\k(M)$-point of $C_x$ (viewed as a classical algebraic curve), that corresponds
to the proper refinements of $v$. Of course for handling $C_x$ as a generalized curve 
we will rather write $C_x(M)$ than $C_x(\k(M))$.

We thus eventually get a bijection between 
$\mathsf P_{x,\widehat X}(M)$ and
a set obtained by adjoining to $C_x(M)$
a single element $\eta$; the most natural way to see 
$C_x(M\cup\{\eta\}$ in this setting is to identify it with $\widehat C_x(M)$ where
$\widehat C_x$ is the set of types on $C_x$ over the 
generalized residue field. Under this bijection, the constant path is mapped to the generic type $\eta$,  
and every (germ of) injective path to a simple point belonging to $C_x(M)$. 

\subsubsection{}
By unraveling the constructions, we see that the bijection can be described as follows. 
Let $p\colon [a,b)\to \widehat X$ be an
element of $\mathsf P_{x,\widehat X}(M)$ (with $p(a)=x$)
and let $f$ be an $M$-regular function defined on an
$M$-Zariski neighborhood of $x$ and such that $f(x)\neq 0$. Let $\xi \in \widehat {C_x}(M)$ the
point corresponding to $I$. Then $\abs f\circ p$ is $<\abs {f(x)}$, 
\resp equal to $\abs{f(x)}$, \resp $>\abs{f(x)}$
on some sub-interval $(a,c)$ of $(a,b)$ 
if and only if 
the generalized rational function 
$\rv(f(x))$ on $C_x$
vanishes at $\xi$, \resp is well-defined and invertible at $\xi$, \resp has a pole at $\xi$. 

\subsubsection{}
Now assume that the point $x$ above is contained in $\widehat U$; choose an admissible
$K$-skeleton $\Sigma$ of $\widehat X$ adapted to $U$ and containing $x$ (Definition \ref{skeleton-adapted}), 
Lemma \ref{exist-adapted}). 
Set $D=\rho_{X,\Sigma}(U)$; this is an 
$M$-definable subset of $\Sigma$ and $\widehat U=\widehat \rho_{U,\Sigma}^{-1}(D)$. 
Let $F$ be the finite set of $M$-points of $C_x$ corresponding to the 
elements of $\mathsf P_{x,\widehat X}$ which ``are" (\ie, have a representative which is)
contained in $\Sigma$ and not contained in $D$. 
Let $I$ be any non-constant element of
$\mathsf P_{x,\widehat X}(M)$  and let $\zeta$ be the corresponding
$M$-point of $C_x$. 
Suppose that $I$ is contained in $\widehat U$. Then if $I$ is contained in  $\Sigma$ it is contained in $D$, since
$\widehat U\cap \Sigma=D$ by construction; as a consequence, $\zeta\notin F$.  
Conversely, assume that $\zeta\notin F$. Then if $I$ is contained in $\Sigma$, it is contained in $D$, hence
in $\widehat U$. And if $I$ is not contained in $\Sigma$ there exists some $y\in \widehat X\setminus \Sigma$ such that 
$I=(y,x)$ (as a germ) and $(y,x)\cap \Sigma=\emptyset$. Then $x=\widehat \rho{(U,\Sigma)}(z)$
for every 
$z\in (y,x)$, so $(y,x)\in 
\widehat \rho{(U,\Sigma)}^{-1}(x)$; and since $x$ belongs to $\widehat U\cap \Sigma$ it belongs to $D$, so
$(y,x)\in \widehat U$. 
Otherwise said, if we set $C_{x,\widehat U}=C_x\setminus F$, the natural bijection between 
$\mathsf P_{x,\widehat X}$ and $\widehat C_x$ induces a bijection between 
$\mathsf P_{x,\widehat U}$ and $\widehat C_{x,\widehat U}$. Note that the latter property shows that 
the dense $M$-Zariski open subset $C_{x,\widehat U}$ of $C_x$ actually only depends on $U$, and not on the choice
of a particular skeleton adapted to $U$.  It follows from 
Corollary
\ref{coro-boundary-uhat} that 
the curve $C_{x,\widehat U}$ is equal to the whole of $C_x$ if and only if $x$
belongs to the topological interior of $x$ in $\widehat U$. More generally, if $V$ is a $K$-definable
subset of $U$ with $x\in \widehat V$, then $C_{x,\widehat V}\subset C_{x,\widehat U}$ with equality if and only if
$x$ belongs to the topological interior of $\widehat V$ in $\widehat U$:
one sees this by considering an 
admissible $K$-skeleton adapted to both $U$ and $V$.

\begin{enonce}[remark]{Warning}\label{field-definition-Cx}
If the point $x$ is $K$-definable so is the curve $C_x$, 
thus $C_x$ is defined over the perfect closure of $\RV(K)$
as a generalized algebraic variety, but not necessarily over
$\RV(K)$ itself. This is related to the fact that the 
formation of the generalized residue field
does not commute in general with ground field extension 
(but it does as soon as the ground field is algebraically closed). 

\end{enonce}

\begin{defi}\label{defi-allbranches}
Let
$\Sigma$ be an admissible $K$-skeleton of $\widehat U$. 
We shall say that $\Sigma$ \emph{contains all $U$-branches at its
simple points} if for every simple point $x$ of $\Sigma$
the two following equivalent condition are fulfilled: 
\begin{enumerate}[i]
\item every germ of interval  on $\widehat U$ emanating from $x$ is contained in $\Sigma$; 
\item $\widehat \rho_{U,\Sigma}^{-1}(x)=\{x\}$. 
\end{enumerate}
\end{defi}

\begin{rema}\label{normal-unique-branch}
Let $x$ be a point of $X$. 
it follows from \ref{type-of-types} that the set of germs of intervals
on $\widehat X$ emanating from $x$ is in bijection with the set of elements 
of $\breve X$ that belong to $\breve U$ for every definable subset $U$ of $X$ 
such that $\widehat U$ is an open neighborhood of $x$ in $\widehat X$, which is itself 
in bijection with the set of pre-images of $x$ on the (geometric) normalization of $X$
(to such a pre-image, associate the type obtained by composing the corresponding
discrete valuation with the valuation of the ground field). 
In particular this set is always finite, and is a singleton as soon as $x$ is a smooth point. 
\end{rema}

  \begin{lemm}\label{lemm-u-branches}
  Let $\Sigma$ be a $K$-skeleton 
  of $\widehat U$. 
  There exists an admissible $K$-skeleton of $\widehat U$ containing
  all $U$-branches at its simple points and containing $\Sigma$.
  \end{lemm}
  
 \begin{proof} We can replace $K$ by its perfection, and $X$ by the associated reduced (hence generically smooth)  curve. 
 Let $X'$ be the normalization of $X$,  let $U'$ be the pre-image of $U$ on $X'$ and $\Sigma'$ be the pre-image of
 $\Sigma$ on $\widehat {U'}$. Let $E$ be the set of singular points of $X$ and let $E'$ be its pre-image on $X'$. 
 
 Let $\Tau$ be an admissible $K$-skeleton of $\widehat {U'}$ containing $\Sigma'$
 and $E'\cap U'$. The image $\Upsilon$ of $\Tau$ on $\widehat U$ 
 is a $K$-skeleton. Let $x$ be a simple point of $\Upsilon$. 
 Let $I$ be a compact interval on $\widehat U$ having $x$ as an endpoint and which does not contain
 any simple point other than $x$. Let $I'$ be the pre-image of $I$ on $\widehat {U'}$. Then $I'$ consists
 of a connected
 component $J$ homeomorphic to $I$, containing a unique pre-image $x'$ of $x$, and 
of a finite set (possibly empty) of isolated pre-images of $x$. The existence of $J$ ensures that 
$x'$ is not isolated in $\widehat{U'}$. Moreover $x'$ belongs to $\Tau$ : indeed this is clear if $x$ is a smooth point, 
for $x'$ is then the unique pre-image of $x$ in $\widehat{X'}$, and $x$ has a pre-image in $\Tau$ by definition of $\Upsilon$; and
if $x$ is not smooth then $x'\in E'$, so $x'\in \Tau$. 
Now as $x'$ is not an isolated point of $\widehat{U'}$ and as $\Tau$ is admissible, $x'$ is not isolated in $\Tau$, so up to shrinking $I$
we can assume that $J\subset \Tau$ (as $X'$ is normal, there is a unique germ of interval
on $\widehat {X'}$ emanating from $x'$, see Remark \ref{normal-unique-branch}). Then $I\subset \Upsilon$. 

It remains to show that $\Upsilon$ (which obviously contain $\Sigma$)
 is admissible. But $\Tau$ is admisssible,
 so there exists an admissible $K$-homotopy
on $\widehat {U'}$ with target $\Tau$. As $\Tau$ contains $E'\cap U'$ and as $X'\to X$ is an isomorphism 
 outside $E$, this homotopy
  descends to $\widehat U$  on which
 it induces an admissible $K$-homotopy
 with target $\Upsilon$. 
\end{proof}
 
\subsection{The map $\wrr$}\label{construction-retraction-huber}
From now on we denote by $\Sigma$ an
admissible skeleton 
of $\widehat U$ containing all $U$-branches at its simple points. 

\subsubsection{Construction of $\wrr$}
Let $u$ be a point of $\widehat U$. 

\paragraph{}
If $u\in \Sigma$ we set $\wrr_{U,\Sigma}(u)=u$. 

\paragraph{}\label{wrr-not-sigma}
If $u\notin \Sigma$
then $[u,\widehat \rho_{U,\Sigma}(u))$ defines a germ of injective path emanating from
$\widehat \rho_{U,\Sigma}(u)$
and not contained in $\Sigma$, and in turn
defines a point $\wrr_{U,\Sigma}(u)$ of $\breve U$ by \ref{interval-type}. 
Since $\Sigma$ contains all $U$-branches at its simple points,
the point $\widehat \rho_{U,\Sigma}(u)$ is not simple, 
so one is in the bounded case described in 
\ref{interval-type-bounded}: the point
$\wrr_{U,\Sigma}(u)$ actually belongs to $\widetilde U$, and 
$\varpi(\wrr_{U,\Sigma}(u))=\widehat\rho_{U,\Sigma}(u)$
(we remind that $\varpi$ denotes the natural map from $\widetilde X$ to $\widehat X$, see 
\ref{defsigma}).

\subsubsection{}
We denote by $\rr_{U,\Sigma}$ the restriction of $\widehat \rr_{U,\Sigma}$ to $U$. The map $\widehat \rr_{U,\Sigma}$ can be
reconstructed from
$\rr_{U,\Sigma}$ by realizing types. 
The maps $\rr_{U,\Sigma}$
and $\widehat \rr_{U,\Sigma}$ are clearly
pro-$K$-definable. As $\widetilde U$ is strictly pro-$K$-definable, 
it follows that the images $\rr_{U,\Sigma}(U)$
and $\wrr_{U,\Sigma}(\widehat U)$ are $K$-definable. 

We set $\Sigma_{\RV,U}=\wrr_{U,\Sigma}(\widehat U\setminus \Sigma)$. 
It follows from \ref{wrr-not-sigma} that $\varpi(\Sigma_{\RV,U})$ is
contained in $\Sigma$ and only consists of non-simple points, 
and that $\widehat \rho_{U,\Sigma}=\varpi\circ \wrr_{U,\Sigma}$ on $\widehat U
\setminus \Sigma$.

\begin{lemm}\label{lemm-mathfrakr-fibers}
The fibers of $\wrr$ are exactly the connected components of
$\widehat U\setminus \Sigma$. 
\end{lemm}

\begin{proof}
Let $x$ and $y$ be two points of $\widehat U\setminus \Sigma$.

If $\wrr_{U,\Sigma}(x)=\wrr_{U,\Sigma}(y)$ then 
$x$ and $y$ have the same image, say $z$, under the map $\widehat \rho_{U,\Sigma}$;
moreover $[x,z]$ and $[y,z]$ define the same germ of interval emanating from $z$, so there is some $t\in [x,z)\cap[y,z)$
such that $[x,z]\cap [y,z]=[t,z]$. Then $[x,z)\cup [y,z)$ is connected, so there is an interval $[x,y]$ on $\widehat U
\setminus \Sigma$ (which is unique since $\Sigma$ contains all loops on $\widehat U$), so $x$ and $y$ lie in the same
connected component of $\widehat U\setminus \Sigma$. 

Conversely, assume that  $x$ and $y$ lie in the same
connected component of $\widehat U\setminus \Sigma$. 
Since $\widehat U$ is locally path-connected, there is some interval $[x,y]$ on $\widehat U\setminus \Sigma$. 
Then the union $[x,y]\cup[y,\widehat \rho_{U,\Sigma}(y)]$ is connected, and its intersection with $\Sigma$
is equal to $\{\widehat \rho_{U,\Sigma}(y)\}$. It follows that  $[x,y]\cup[y,\widehat \rho_{U,\Sigma}(y)]$
contains an interval $[x,\widehat \rho_{U,\Sigma}(y)]$ whose intersection with $\Sigma$ is equal
to  $\{\widehat \rho_{U,\Sigma}(y)\}$, and which coincides with
$[y,\widehat \rho_{U,\Sigma}(y)]$ near $\widehat \rho_{U,\Sigma}(y)$.  As a consequence, 
$\widehat \rho_{U,\Sigma}(x)=\widehat \rho_{U,\Sigma}(y)$
and $\wrr_{U,\Sigma}(x)=\wrr_{U,\Sigma}(y)$. 
\end{proof}

\begin{coro} One has $\Sigma_{\RV,U}=\rr(U\setminus \Sigma)$, so $\Sigma_{\RV,U}=\rr(U)$
if $\Sigma$ has no simple points. 
\end{coro}

\begin{proof} The inclusion $\rr(U\setminus \Sigma)
\subset \Sigma_{\RV,U}$ is obvious, so it remains to show the converse inclusion. 
Let $x\in\widehat U\setminus \Sigma$, say $x\in \widehat U(M)$ for some model $M$ of $\acvf$.
By \ref{basic-topology-uhat}, there exists some $M$-definable
subset $V$ of $U$ such that $\widehat V$ is a connected open 
neighborhood of $x$ in $\widehat U$. By Lemma \ref{lemm-mathfrakr-fibers}
above 
one has then $\wrr_{U,\Sigma}(y)=\wrr_{U,\Sigma}(x)$ for every
$y\in \widehat V$ ; since $V(M)\neq \emptyset$ (as $\widehat V$ contains $x$, 
so $V$ is non-empty) we get the existence of 
$y\in V(M)$ with $\wrr_{U,\Sigma}(y)=\wrr_{U,\Sigma}(x)$. 
\end{proof}

\begin{rema}
Lemma \ref{lemm-mathfrakr-fibers}
shows that lying in the same connected
component of $\widehat U\setminus \Sigma$ is a
$K$-definable equivalence relation (with quotient $\Sigma_{\RV,U}$), 
which did not seem a priori obvious. 
\end{rema}

\subsection{The fibers of $\varpi \colon \Sigma_{\RV,U}\to \Sigma$}
Our purpose is now to investigate the fibers of the map $\varpi 
\colon \Sigma_{\RV,U}\to \Sigma$. 
In fact one has $\varpi(s)=s$ for every $s\in \Sigma$ (even for every $s$
in $\widehat U$) and it will be to some extent more natural to study the fibers of
$\varpi\colon  \Sigma_{\RV,U}\cup \Sigma\to \Sigma$.

\subsubsection{}
Let $s$ be a point of $\Sigma$. 

\paragraph{}
If $s$ is simple, its only pre-image
in $\Sigma_{\RV,U}\cup \Sigma$ under $\varpi$ is $s$ itself. 

\paragraph{}
Assume that $s$ is not simple. 
By the very construction of $\wrr_{U,\Sigma}$, the set of points
$\zeta$ of $\Sigma_{\RV,U}$ such that $\varpi(\zeta)=s$ can be identified
with the set of germs of injective paths emanating from $s$, contained
in $\widehat U$ and not contained in $\Sigma$; hence $\varpi^{-1}(s)
\cap \Sigma_{\RV,U}$ is in bijection with 
the dense Zariski-open subset $C_{s, \widehat U\setminus \Sigma}$
of $C_{s, \widehat U}$ obtained by removing the (finitely many) points corresponding 
to germs of edges of $\Sigma$ emanating from $s$. 
It follows that $\varpi^{-1}(s)
\cap (\Sigma_{\RV,U}\cup \Sigma)$ is in natural bijection with
$\widehat C_{s,\widehat U\setminus \Sigma}$ (with $s$ corresponding to the generic type).

\subsubsection{}\label{components-are-stable-completions}
Let $s$ be a non-simple point of $\Sigma$ and 
let $D$ be a Zariski-constructible subset of 
$ C_{s,\widehat U\setminus \Sigma}$. We observe that
the stable completion 
$\widehat D$ of $D$ is then equal to $D$ if $D$ is finite, and to $D\cup\{\eta\}$ otherwise
(with $\eta$ the generic type). 
Let us view $\widehat
C_{s,\widehat U\setminus \Sigma}$ as contained in $(\Sigma_{\RV,U}\cup \Sigma)\subset 
\widetilde U$. Since
$\wrr_{U,\Sigma}$ is induced by $\rr_{U,\Sigma}$ through realization of types, the pre-image
$\wrr_{U,\Sigma}^{-1}(\widehat D)$ is the stable completion of $\rr_{U,\Sigma}^{-1}(D)$. 
In particular, if $\zeta$ is a point of $C_{s,\widehat U\setminus \Sigma}$
then $\wrr_{U,\Sigma}^{-1}(\zeta)$, which is a connected component of 
$\widehat U\setminus \Sigma$ (Lemma \ref{lemm-mathfrakr-fibers}),
is equal to the
stable completion  of  $\rr_{U,\Sigma}^{-1}(\zeta)$.

\begin{exem}\label{sigmarv-projective-line}
A particular case of the construction above had already been carried out directly 
in \ref{projective-line-explicit} for $\P^1$ and 
its standard $K$-skeleton $\Sigma=[0,\infty]$
(what we denoted there by $\Sigma_\RV$ and $\rr$ should be denoted by $\Sigma_{\RV,\P^1}$ and $\rr_{\P^1,\Sigma}$
according to our
current notation, but we will keep the simpler notation of  \ref{projective-line-explicit} in what folllows). 

We saw in \ref{projective-line-explicit} that there is a $K$-definable isomorphism 
between $\Sigma_\RV$ and $\RV$, modulo which $\rr \colon (\P^1\setminus \Sigma)=\gm
\to \Sigma_\RV$ is simply $a\mapsto \rv(a)$, and modulo which the map $\varpi \colon 
\Sigma_\RV\to \Sigma$
is simply $\rv(a)\mapsto \eta_{0,\abs a}$. If $\zeta$ is an element of $\Sigma_\RV$
identified to $\rv(a)$ for some $a\neq 0$ then $\rr^{-1}(\zeta)$ 
is the open disc 
$\rv^{-1}(\rv(a))$ of center $a$ and radius $\abs a$ and the connected component $\wrr^{-1}(\zeta)$ 
of $\widehat {\P^1}\setminus \Sigma$ is the stable completion of this open disc. 

For any $r\in \Gamma$ the fiber $\varpi^{-1}(\eta_{0,r})\cap \Sigma_\RV$ is equal, modulo the bijection $\Sigma_\RV \simeq \RV$, 
to $\{\rv(a)\}_{\abs a=r}$, so this is precisely $C_{\eta_{0,r},\widehat {\P^1}\setminus \Sigma}=\P^r_\RV
\setminus \{0,\infty\}$ \emph{when 
$C_{\eta_{0,r}}$ is seen as the generalized curve $\P^r_\RV$}, see Example \ref{generalized-residue-curve}
(the two germs of intervals emanating from $\eta_{0,r}$ and contained in $\Sigma$ correspond precisely to the points
$0$ and $\infty$ of $\P^r_\RV$). 
It is thus definitely more natural in this setting to deal with generalized residue curves. 
\end{exem}

\subsection{A variant: the map $\trr$}
The map $\wrr_{U,\Sigma}$ can be deduced from $\rr_{U,\Sigma}$ through the realization of types. 
By the same procedure we get a map  $\trr_{U,\Sigma} \colon \widetilde U\to \widetilde U$.
But any point of $\widetilde U$ can be seen as a definable type on $\widehat U$ by the procedure
described in \ref{type-of-types}, and with this viewpoint $\trr_{U,\Sigma}$ is the extension of 
$\wrr_{U,\Sigma}$ through the
realization of types. 

\subsubsection{}
Let $M$ be a model of
$\acvf$ containing $K$ and let $u\in \widetilde U(M)$. Let us describe it like in \ref{type-of-types}
as $f(a^+)$ for some germ of path $f\colon (a,b)\to \widehat U$ with $a$ and $b$ in $\abs M.$ We can in 
fact always assume that $a\neq 0$: indeed, 
since $u$ belongs to $\widetilde U(M)$ it is bounded with respect to $M$, so if $a=0$ the germ
$f$ is constant near $a$ 
and thus can be replaced by any other constant germ of path with the same value. 
We can also assume that the image $I$ of $f$ either avoids $\Sigma$ or is entirely contained in $\Sigma$. 
We choose a model $N$ of $\acvf$ containing $M$ and such that the type $a^+$ is realized over $M$
by some $\alpha\in \abs N$. Set $v=f(\alpha)\in \widehat U(N)$, and we distinguish two cases.

\paragraph{}\label{rtilde-component}
Assume first that 
$I$ avoids $\Sigma$. Then $I$ is contained in some 
component of $\widehat U\setminus \Sigma$, so $\wrr_{U,\Sigma}$ is constant on $I$ and its value is an 
element of $\Sigma_{\RV,U}(M)$, \ie, it can be identified
to some $\zeta\in 
C_{s,\widehat U\setminus \Sigma}(M)$ for some non-simple point $s\in \Sigma(M)$. 
Then $\wrr_{U,\Sigma}(v)=\zeta$ as well, so $\trr_{U,\Sigma}(u)=\zeta$.

\paragraph{}
Assume now that $I\subset \Sigma$. This is the case where there is some
$s\in \Sigma(M)$ such that $u$ is either equal to $s$ or can be
identified to $\zeta$ for some
$\zeta\in C_{s,\widehat U}$ corresponding to an edge of $\Sigma$ emanating from $s$ (the one 
provided by $I$). Then as $f(\alpha)\in \Sigma$ we have $\wrr_{U,\Sigma}(v)=v$, so $\trr_{U,\Sigma}(u)=u$.

\subsubsection{}
Let $s$ be a non-simple point of  $\Sigma(M)$ and let $\zeta$ be a 
point of $C_{s,\widehat U\setminus
\Sigma}$. We can see $\zeta$ as a point of $\widetilde U$
satisfying $\varpi(\zeta)=s$. Then by considering a germ of interval
emanating from $s$ and defining $\zeta$ and applying \ref{rtilde-component}
we see that $\trr_{U,\Sigma}(\zeta)=\zeta$.

\subsubsection{Recapitulation}
The image $\trr_{U,\Sigma}(\widetilde U)$ is fibered through $\varpi$ over $\Sigma$; its fiber over a
simple point $s$ is $\{s\}$, and its fiber over a non-simple point $s$ 
can be identified to  $\widehat C_{s,\widehat U}$ (with $s$ corresponding to the generic type). 
Let us use this identification, and choose
$\zeta\in \widehat C_{s,\widehat U}$. 
Then $\trr_{U,\Sigma}^{-1}(\zeta)$ contains $\zeta$; 
it is reduced to $\{\zeta\}$ if $\zeta$ is the generic type or a point corresponding to a germ of interval
emanating from $s$ and contained in $\Sigma$, and
it is equal to
$\widetilde{\rr_{U,\Sigma}^{-1}(\zeta)}$
otherwise. 

\subsection{Sub-skeletons}\label{subskeletons}
Let $\Tau$ be a $K$-definable subset of $\Sigma$ and set
$V=\rho_{U,\Sigma}^{-1}(\Tau)$. Then $\widehat V=
\widehat\rho_{U,\Sigma}^{-1}(\Tau)$ and $\Tau$ is an 
admissible $K$-skeleton of $\widehat V$. It follows from
the construction that $\Tau$ 
contains all $V$-branches at its simple points, that for every
non-simple point $s$ of $\Tau$ one has
$C_{s,\widehat V\setminus \Tau}=C_{s,\widehat U\setminus \Sigma}$, 
that $\Tau_{\RV,V}=\Sigma_{\RV,U}\times_\Sigma \Tau$, and that
$\rr_{V,\Tau}$, \resp $\wrr_{V,\Tau}$, \resp  $\trr_{V,\Tau}$
is the restriction of $\rr_{U,\Sigma}$, \resp $\wrr_{U,\Sigma}$, \resp
$\trr_{U,\Sigma}$ to $V$, \resp $\widehat V$, \resp $\widetilde V$.

\subsection{More topology : branches at infinity, trees}
Let $U$ be a $K$-definable subset of a $K$-algebraic curve $X$. 

\subsubsection{}
A \textit{branch at infinity} of $\widehat U$ 
is an equivalence class of closed definable subsets of $\widehat U$
which are (definably) homeomorphic to a semi-open interval, for the relation 
identifying two such subsets if their intersection is still a semi-open interval. 

The definable set $\widehat U$ has only finitely many branches at infinity: if $\Sigma$
is an admissible $K$-skeleton of $\widehat U$ such that
$\rho_{U,\Sigma}$ is topologically proper, the branches at infinity of $\widehat U$
are exactly those of $\Sigma$.

\subsubsection{}
We shall say that
$\widehat U$ is a \textit{tree} if for all $x,y$ in $\widehat U$ there is a unique 
$K$-definable interval on $\widehat U$ with endpoints $x$ and $y$. If $U$ is non-empty then 
$\widehat U$ is a tree if and only if it is ($K$-definably) contractible, and also if and
only if it is connected and contains no loop. 
If $U$ is a Zariski open subset of $\P^1$, \resp an open or a closed disc, \resp an open, closed
or semi-open annulus, $\widehat U$ is a tree. 

If $\widehat U$ is a tree and if $b_1,\ldots, b_n$ are its branches at infinity, $\widehat U$
admits a canonical definable compactification, which is still a tree, and is obtained by taking
for each $j$ a representative $I_j$ of $b_j$ and by adding a point at the open end of the interval
$I_j$. 

\subsubsection{}\label{complement-of-skeleton}
Let $\Sigma$ be an admissible $K$-skeleton of $\widehat U$
containing all $U$-branches at its simple points. 
Let $\Omega$ 
be a connected component
of $\widehat U\setminus \Sigma$. It is equal to $\wrr_{U,\Sigma}^{-1}(\zeta)$ for some point 
$\zeta\in \Sigma_{\RV,U}$
and we have 
$\Omega=\widehat V$ where $V$ is the definable subset 
 $\rr_{U,\Sigma}^{-1}(\zeta)$ of $U$
 (Lemma \ref{lemm-mathfrakr-fibers}, \ref{components-are-stable-completions}). 
Set $s=\varpi(\zeta)$.
For every $\omega\in \Omega$
there is a unique interval of the form $[\omega,t]$ with $[\omega,t]\cap \Sigma=\{t\}$, 
and one has necessarily $t=s=\rho_{U,\Sigma}(\omega)$ ; moreover, the germ of the interval $[\omega,s]$
at $s$ does not depend on $\omega$. It follows that $\Omega$ is a tree, and 
that there exists a unique branch at infinity on $\Omega$
having a limit in $\widehat U$
(this branch is the class of $[\omega,s)$ for any $\omega\in \Omega$, and its limit is $s$); 
the topological boundary of $\Omega$ in $\widehat U$ is equal to $\{s\}$. 

The map $\rho_{U,\Sigma}$ is topologically proper if and only every connected component
$\Omega$ of $\widehat U \setminus \Sigma$ has only one branch at infinity; this amounts to
require that every such component have compact closure in $\widehat U$. 

\begin{prop}{}\label{prop-connected-comp-definable}
Let $\Sigma$ be any closed $K$-skeleton of $\widehat U$. 
Let $\Tau$ be an admissible $K$-skeleton that
contains all
$U$-branches at its simple points and  that contains $\Sigma$
(note that such a $\Tau$ exists by Lemma \ref{lemm-u-branches}). 
Let $\Pi$
be the finite set of connected components
of $\Tau\setminus \Sigma$. 

\begin{enumerate}[1]
\item The family of connected component of $\widehat U\setminus \Sigma$ 
is
\[(\widehat \rho_{U,\Tau}^{-1}(\Theta))_{\Theta \in \Pi}\coprod (\wrr_{U,\Tau}^{-1}(\zeta))_{s\in 
\Sigma\setminus U, \zeta \in C_{s,\widehat U\setminus \Tau}}.\]
\item Lying on the same connected component of $\widehat U\setminus \Sigma$ is a $K$-definable equivalence relation
with quotient set $\Pi\coprod ( \varpi^{-1}(\Sigma)\cap \Tau_{\RV,U})$.
\item Each connected component of $\widehat U\setminus \Sigma$ 
is the stable completion of a definable subset of $U$.
\end{enumerate}
\end{prop}

\begin{proof} 
The set $\widehat U\setminus \Sigma$
is the disjoint union of 
its intersection
with $\widehat \rho^{-1}_{U,\Tau}(\Tau\setminus \Sigma)$
and 
$\widehat \rho_{U,\Tau}^{-1}(\Sigma)$, 
which can be respectively written 
$\coprod_{\Theta \in \Pi}\widehat \rho_{U,\Tau}^{-1}(\Theta)$
and $\coprod_{t\in \Sigma\setminus U, \zeta \in C_{t, \widehat U
\setminus \Tau}}\wrr_{U,\Tau}^{-1}(\zeta)
$.

Now for each $\Theta\in \Pi$ the summand
$\widehat \rho_{U,\Tau}^{-1}(\Theta)$ is the stable completion of
the $K$-definable subset $\rho_{U,\Tau}^{-1}(\Theta)$ of $U$, and it is open
in $\widehat U$, connected
and  non-empty; for each $t\in \Sigma \setminus U$ and each $\zeta \in 
C_{t,\widehat U\setminus \Tau}$, 
the pre-image $\wrr_{U,\Tau}^{-1}(\zeta)$ is the stable completion of the 
$K\zeta$-definable subset  $\rr_{U,\Tau}^{-1}(\zeta)$ of $U$, and it is open
in $\widehat U$
(and even in $\widehat X$), connected
and non-empty. This ends the proof of (1) and (3), and (2) is an obvious consequence of (1). 
\end{proof}


\begin{lemm}\label{skeleton-disc}
Let $X$ be an algebraic $K$-curve, let $U$
be a $K$-definable subset of $X$
and let 
$\Sigma$ be an admissible $K$-skeleton of $\widehat U$. 
Let $V$ be a $K$-definable subset of $U$ such that
$\widehat V$
is a non-empty connected open subset
of $\widehat U$ and such that $\Sigma$ contains
the topological boundary $\partial \widehat V$ of $\widehat V$ in
$ \widehat U$. Then the following holds:

\begin{enumerate}[1]
\item If $\widehat V\cap \Sigma=\emptyset$
then $\widehat V$ is a connected component of
$\widehat U\setminus \Sigma$. 
\item If $\Sigma \cap \widehat V\neq \emptyset$
then: 
\begin{itemize}[label=$\diamond$]
\item 
$\Sigma \cap \widehat V$ is 
an admissible $K$-skeleton of $\widehat V$;
\item $\widehat \rho_{V,\Sigma\cap \widehat V}=\widehat \rho_{U,\Sigma}|_{\widehat V}$;
\item $\widehat V=\widehat \rho_{U,\Sigma}^{-1}(\widehat V\cap \Sigma)$. 
\end{itemize}
\end{enumerate}
\end{lemm}

\begin{proof}
Assume that $\widehat V$ does not intersect $\Sigma$. Then $\widehat V$ is an open, non-empty
and connected subset of $\widehat U\setminus \Sigma$, which is moreover closed in
$\widehat U\setminus \Sigma$, for the boundary of $\widehat V$ in $\widehat U$ is contained
in $\Sigma$. So $\widehat V$ is a connected component of $\widehat U\setminus \Sigma$.

Assume now that $\Sigma\cap \widehat V$ is non-empty. Choose a point $v\in \Sigma
\cap \widehat V$. As $\widehat V$ is connected, its closure is contained in a connected
component $\Omega$ of $\widehat U$, and $\Sigma\cap \Omega$ is 
connected since $\Sigma$
is $K$-admissible. Let $w\in \partial \widehat V$; it belongs to $\Sigma$ by assumption, so
by connectedness of $\Sigma\cap \Omega$ 
there is a path joining $w$ to $v$ on $\Sigma$; and as $\Sigma$ contains all loops drawn on
$\widehat U$, every interval joining $w$ to $v$ is contained in $\Sigma$. By connectedness 
of $\widehat V$, every branch at infinity of $\widehat V$ with limit $w$ is contained in an 
interval joining $w$ to $v$ all of whose
points but $w$ belong to $\widehat V$. 
Therefore $\Sigma\cap \widehat V$ contains all branches at infinity of $\widehat V$ 
with limit $w$; as $w$ is arbitrary, $\Sigma\cap \widehat V$ 
contains all branches at infinity of $\widehat V$ 
having a limit in $\widehat U$. 

Let $x\in \widehat V$. In order to prove that $\Sigma \cap \widehat V$ is 
an admissible $K$-skeleton of $\widehat V$
and
$\widehat \rho_{V,\Sigma\cap \widehat V}=\widehat \rho_{U,\Sigma}|_{\widehat V}$, 
it suffices to show that $[x,\widehat \rho_{U,\Sigma}(x)]\subset  \widehat V$. 
We argue by contradiction, so we assume that this is not the case. 
Then $[x,\widehat 
\rho_{U,\Sigma}(x)]\cap \widehat V$ is a proper open subset of
$[x,\widehat \rho_{U,\Sigma}(x)]$, and the connected component of $x$ in the latter open subset
if of the form $[x,y)$ with $y\in \partial \widehat V$. But we have seen that any 
branch at infinity of $\widehat V$
having a limit in $\widehat U$
is contained in $\Sigma$, so there exists
$z\in [x,y)$ with $[z,y)\subset \Sigma$, contradicting the fact that $[x,
\widehat \rho_{U,\Sigma}
(x))$ does not intersect $\Sigma$. 

It remains to show that $\widehat V=\widehat \rho_{U,\Sigma}^{-1}(\widehat V\cap \Sigma)$. 
It follows from the above that $\widehat V$ is contained in
$\widehat \rho_{U,\Sigma}^{-1}(\widehat V\cap \Sigma)$. Conversely, let 
$x$ be a point of $\widehat U$ such that $\widehat \rho_{U,\Sigma}(x)
\in \widehat V$. If $x\notin \widehat V$ then there exists
$y\in [x,\widehat\rho_{U,\Sigma}(x))$
with $y\in \partial \widehat V$; but then $y\in \Sigma$, contradicting
the definition of $\widehat 
\rho_{U,\Sigma}$. 
\end{proof}

\section{Generalities on Nash functions}\label{section-nash}
We start this section 
 by proving that a flat map induces
 an open map at the level of stable completions. 
We shall need it only for étale maps, for which it could be deduced (through
base-change by the normalization of $X$) from  \cite[Cor. 9.7.2]{hrushovski-l2016}. 
But as its proof is quite short, it seems to us to be of independent interest to
handle the flat case.

\begin{prop}\label{flat-open}
Let $f\colon Y\to X$ be a flat map
between algebraic $K$-varieties. The induced map 
$\widehat Y\to \widehat X$ is open. 
\end{prop}

\begin{proof}
We can assume that $K$ is a model
of $\acvf$. Let $V$ be a $K$-definable subset of $Y$ such that 
$\widehat V$ is open in $\widehat Y$; set $U=f(V)$. We are going to prove
that $\widehat U$ is open in $\widehat X$. 
Let $u$ be a point of 
$\widehat U(K)$ and let
$p$ be a type over $K$ on $X$ such that $u$ belongs to the $K$-Zariski closure of $p$
and for every pair $(f,g)$ of $K$-regular functions defined on a $K$-Zariski neighborhood
of $u$ such that $\abs{f(u)}<\abs{g(u)}$ one has $\abs{f(p)}<\abs{g(p)}$; we are going to show
that $p$ lies on $U$, which will end the proof.

Since $f(\widehat V)=\widehat U$ there exists some $v\in \widehat V(K)$ lying above $u$. 
From now
on we will see $X$ and $Y$ as $K$-schemes, and use
the corresponding language, which
is more convenient for exploiting flatness; so  \emph{points} will be
\emph{scheme-theoretic points} and we will use contructible
compactness instead of logical compactness. 
We see $u$ and $p$ as elements of the valuation spectrum of $X$; \ie, pairs consisting
of a point of $X$ and a valuation on its residue field extending that of $K$. 
Let $\theta$, \resp $\xi$, be the points
of $X$ supporting $u$, \resp $p$. 
As the type $u$ belongs to
$\widehat U(K)$, it is given by a valuation 
$\abs \cdot$ on $\kappa(\theta)$, extending that of $K$ and 
with value group $\abs{K^\times}$. 
The assumption on the pair $(u,p)$ implies that 
$\xi$ specializes to $\theta$
(since $\abs{f(u)}>0\Rightarrow \abs{f(p)}>0$)
and that $p$ is of the form 
$\abs \cdot '\circ \nu$ where $\nu$ is a valuation on $\kappa(\xi)$ trivial on $K$
and centered at $\theta$
and where $\abs\cdot'$ is a valuation on the residue field
of $\nu$ whose restriction to $\kappa(\theta)$ refines $\abs \cdot$
(as $\abs{f(u)}<1\Rightarrow \abs{f(p)}<1$); but
note that $\abs \cdot '$ and $\abs \cdot$ coincide on $K$. 

For proving the result, we can replace the scheme $X$ with the reduced closure of $\xi$
(and all other data by the corresponding fiber products) so we
can assume that $X$ is integral and $\xi$ is its generic point.
We denote by $\omega$ the point of $Y$
supporting $v$. 

Our purpose is now to show the existence of some $q$ on the valuation spectrum of $Y$
that lies above $p$ and has the property that its support generalizes $\omega$ and that 
$\abs{f(q)}<\abs{g(q)}$ for all pairs $(f,g)$ of $K$-regular functions on $Y$ around
$\omega$ such that  $\abs{f(v)}<\abs{g(v)}$. By openness of $\widehat V$ in $\widehat Y$ this will
ensure that $q$ lies on $V$, so that $p=f(q)$ will lie on $U$ -- we will then be done.

Let $Z$ be any blow-up of $X$ along a proper closed subscheme $C$.  
By flatness $Y\times_X Z$ is the blow-up of $Y$
along $Y\times_X C$. Let $\zeta$ be the center of $\mu$ on $Z$ and let $E_Z$ be the 
set of points of $Y\times_X Z$ lying above $\omega$ on $Y$
and $\zeta$ on $Z$; it is pro-constructible (in the sense of scheme theory) and
non-empty, as $\omega$ and $\zeta$ both lie over $\theta$.
The collection $(E_Z)$ for $Z$ running through all
possible blow-ups of $X$ along a proper closed subscheme is a filtering system
of non-empty compact sets (for the constructible topology), so it has a non-empty limit. 
Let $\nu$ be an element of this limit. For every $(Z,C)$ as above, $Y\times_ZC$
is nowhere dense in $Y$ by flatness of $Y\to Z$, so the set of irreducible
components of $Y\times_X Z$ is in natural bijection with the
set $I$ of irreducible components of $Y$
(to an irreducible component $T$ of $Y\times_XZ$ one associates its image $S$ on $Y$; 
one recovers $T$ from $S$ by taking its strict
transform). Then $\nu$ defines a projective system of 
non-empty subsets of $I$, which necessarily stabilizes. This implies that
there exists an irreducible component $S$ of $Y$ such that $\nu$ induces an element
of the limit of a filtering family of blow-ups of $S$ along proper centers, which 
can be lifted (again by constructible compactness) to an element of the limit
of \emph{all} blow-ups of $S$ along proper centers; we still call
this element $\nu$. 
By construction, $\nu$ induces 
an element of the valuative spectrum $Y$, 
supported
at the generic point of $S$ and centered at $\omega$, which we still
denote by $\nu$. 
The valuation $\abs \cdot$ on $\kappa(\omega)$ defined by $v$ restricts to the valuation 
$\abs \cdot$ of $\kappa(\theta)$ defined by $u$; then the refinement $\abs \cdot '$ of
$\abs \cdot$ on $\kappa(\theta)$ extends to a refinement $\abs \cdot '$ of $\abs \cdot$
on $\kappa(\omega)$, and in turn to a valuation $\abs \cdot '$ on the residue field
of $\nu$. Now the valuation $q=\abs \cdot '\circ \nu$
satisfies the required properties. 
\end{proof}
\begin{defi}
Let $K$ be a valued field, 
let $X$ be a $K$-variety
and let $U$ be
a $K$-definable subset of $X$.
A \emph{nice $K$-definable exhaustion} of $U$ is 
a non-decreasing
$K$-definable family $(U_t)_{0\leq t<a}$ of subsets of $U$, where $a$ is an element of $\abs{K^\times}^\Q$, 
such that $U$ is the union of the $U_t$ and each 
$\widehat{U_t}$ is compact. 
\end{defi}

\begin{exem}
Let $D\subset \A^n$ be an open polydisc of polyradius $r$. For every $t\in [0,1)$, let $D_t$
be the closed polydisc of polyradius $tr$. Then $(D_t)$ is a nice
$K$-definable exhaustion of $D$. 
\end{exem}


\begin{rema}\label{exhaustion-scalar-extension}
Let $X$ be a $K$-algebraic variety and let 
$U$
be a $K$-definable subset of $X$.
Assume that there exists a finite extension $L$ of $K$ such that $U$ has a nice
$L$-definable exhaustion $(U_t)_t$. 
Then $U$ has nice $K$-definable exhaustion.

Indeed, for seeing this we can replace $K$ by $K\h$
and $L$ by $L\h$, so
we can assume that $K$ is henselian. 
Then we can enlarge $L$ and assume it is normal
over $K$, and now
$(\bigcup_{\sigma \in \mathrm{Aut}(L/K)}\sigma(U_t))$ is a nice
$K$-definable exhaustion of $U$. 

More generally, assume that there exists a valued extension $M$ 
of $K$ such that $U$ admits a nice $M$-definable
exhaustion \emph{and that $K$ is 
non-trivially valued}; then $U$ admits a nice $K$-definable exhaustion. 
Indeed, we can enlarge $M$ and assume that it is a model of $\acvf$. 
The algebraic closure $\overline K$ of $K$ in $M$ is then a model of $\acvf$ (as
$K$ is non-trivially valued), so $U$ has a nice $\overline K$-definable exhaustion, 
necessarily defined over some finite extension of $K$. Then by the above $U$ has 
a nice $K$-definable exhaustion. Note that one cannot drop the assumption 
that $K$ is non-trivially valued: a positive dimensional open unit disc over a trivially
valued fields does not admit any nice $K$-definable exhaustion, but of course it admits
one over any non-trivially valued extension of $K$.

\end{rema}

\begin{lemm}\label{definable-distance}
Let $K$ be a valued field, let $X$ be a $K$-variety and let $Y$ be a Zariski-closed subset of $X$. 
There exists a $K$-definable function $\delta \colon X\to \Gamma_0$
inducing a continuous map $\widehat X\to \Gamma_0$ and whose vanishing locus is $Y$.
\end{lemm}

\begin{proof}
Let $\mathscr X$ be any proper model of $X$ over the ring of integers $\mathscr O_K$ 
of
$K$ (one can build such a model  by applying Nagata's compactification over some
finitely generated $\Z$-subalgbera of $\mathscr O_K$ over which $X$ is defined). 
Let $\mathscr Y$ be the schematic closure of $Y$ in $\mathscr X$, let $\mathscr X'$ be the blow-up of
$\mathscr X$ along $\mathscr Y$, and let $\mathscr Y'$ be the exceptional divisor of the blow-up; let
$Y'$ and $X'$ denote the generic fibers of $\mathscr X'$ and $\mathscr Y'$ respectively. As $X'\to X$ is an isomorphism 
outside of $Z$ and as $\widehat{X'}\to \widehat X$ is topologically proper, we can replace $(X,Y)$ by $(X',Y')$;
hence we can assume that $X$ has a proper model 
$\mathscr X$ over $\mathscr O_K$ and that $Y$ is the generic fiber of an effective Cartier divisor $\mathscr Y$
of $\mathscr X$. Let $\mathscr L$ be the line bundle on $\mathscr X$ associated to $\mathscr Y$; it comes equipped with a 
canonical section $s$ having $\mathscr Y$ as its zero-locus. Let $L$ be the generic fiber of $\mathscr L$; it
is endowed with a natural model metric $\abs \cdot $ induced by $\mathscr L$, and $\delta:=x\mapsto \abs 
{s(x)}$ is a $K$-definable map taking values in $\Gamma_0$ that vanishes exactly on 
$Z$; being locally given by norms of regular functions, it induces 
a continuous definable map on $\widehat X$. 
\end{proof}

\begin{lemm}\label{lemm-exhaustion}
Let $K$ be a valued field, 
let $X$ be a $K$-variety
and let $U$ be
a $K$-definable subset of $X$
admitting a nice 
$K$-definable exhaustion.
Then $\widehat U$ has compact closure in $\widehat X$. 
\end{lemm}

\begin{proof}
Fix a nice exhaustion 
$(U_t)_{0\leq t<a}$ of 
$U$. As $X$ is separated (by our very definition of a variety)
there exists a
proper $K$-variety $Y$ equipped with a dense open immersion 
$i\colon X\hookrightarrow Y$ (defined over $K$). 
We use $i$ to identify $X$ with an open subset of $Y$.

Let $Z$ be the complement of $X$ in $Y$.
By the preceding lemma there exists a $K$-definable
function $\delta\colon Y\to \Gamma_0$
that vanishes exactly on 
$Z$ and extends to a 
continuous definable from $\widehat Y$
to $\Gamma_0$. 

Our purpose is to prove that $\delta$ does not vanish on the closure of $\widehat U$ in $\widehat Y$; this will ensure that the latter 
closure, which is compact by properness of $Y$, 
is contained in $\widehat X$, and therefore that it coincides with the closure of $\widehat U$ in $\widehat X$, which will end the proof. 

Let $\psi \colon [0,a)\to \Gamma_0$ be the map that sends $t$ to the infimum of $\delta$ on $U_t$. 
It is $K$-definable, non-increasing, and it does not vanish: indeed if one had $\psi(t)=0$
from some $t$ then $\delta$ would vanish on $U$, contradicting the fact that $U\cap Z=\emptyset$
since $U\subset X$. Then by o-minimality there exists some $\epsilon< 1$, some 
$b\in \abs{K^\times}^\Q$ and some non-positive rational number $r$  such that $\psi(t)=bt^r$ for all $t\in (\epsilon, 1)$. 
But then $\psi \geq ba^r$ on the whole of $U$, which implies  that $\delta \geq ba^r$ on the whole of $U$, and this also holds
on the whole of $\widehat U$ (just by realizing types); by continuity of $\delta$ one still has 
$\delta \geq ba^r$ on the closure of $\widehat U$ in 
$\widehat X$, and we are done. 
\end{proof}

\begin{lemm}\label{lemm-same-henselization}
Let $f\colon Y\to  X$ be a quasi-finite
morphism of $K$-varieties. Assume that
for every $y\in Y$, the field $K(y)$ is separable over 
$K(f(y))$ (\eg, $f$ is étale, or merely unramified). Let $D$ be a $K$ definable subset of $Y$ such that 
$f|_D$ is injective. One has then $K(y)\h=K(f(y))\h$ for every $y\in D$.
\end{lemm}

\begin{proof}
Set $E=f(D)$. This is a $K$-definable subset of $X$, and $f$ induces a $K$-definable bijection $D\simeq E$. 
Let $\sigma$ be the inverse bijection. This is a $K$-definable map. By construction one has $y=\sigma(f(y))$ for every
$y\in D$, so that $y$ is definable over $K(f(y))$. As $K(y)$ is a finite separable extension of $K(f(y))$ by assumption, this
means that $K(y)\subset K(f(y))\h$, so $K(y)\h=K(f(y))\h$. 
\end{proof}

\begin{defi}\label{defi-nash-iso}
Let $f\colon Y\to X$ be a morphism between algebraic varieties over some valued 
field $K$, let $V$ be a $K$-definable subset of $Y$ and let $U$
be a $K$-definable subset of $X$. We shall say that $f$ induces a 
($K$-definable) \emph{Nash isomorphism} between $V$ and $U$ if $V$
is contained in the étale locus of $f$
and $f$ induces
a $K$-definable isomorphism between $V$ and $U$ inducing in turn
a \emph{homeomorphism} between $\widehat V$ and $\widehat U$. 
\end{defi}

\begin{exem}\label{example-nash-iso}
Assume that $f$ is a map from 
$Y$ to a Zariski-open subset $\Omega$ of 
$X$ containing $U$, and that $V$
is contained in the étale locus of $f$. 
Moreover we assume to be given a $K$-regular function 
$\phi$ on $\Omega$ and $K$-regular function $\psi$ on $X$
such that the following hold: 

\begin{itemize}[label=$\diamond$] 
\item $\abs \phi<1$ on $\Omega$ ; 
\item for every $z\in U$ there exists a unique 
element $\sigma(z)$ in the fiber
$f^{-1}(z)$ with $\abs{\psi(\sigma(z))}=\abs{\phi(z)}$ ; 
\item for every $z\in U$ and every $t\neq \sigma(z)$ in 
in $f^{-1}(z)$ one has $\abs{\psi(t)}=1$ ; 
\item $\sigma(U)=V$. 
\end{itemize}

Then $f$ induces a
Nash isomorphism between 
$V$ and $U$. Indeed, 
by replacing $Y$ with the étale locus
of $f$ (which 
contains $V$)
we can assume that $f$ is étale.

By construction, 
$f$ induces a definable bijection $V\simeq U$
(with inverse bijection $\sigma$), and in turn a 
continuous bijection $\widehat V\to \widehat U$. 

By proposition \ref{flat-open}, $f$ induces an open map 
from $\widehat {Y'}$ to $\widehat X$; by base-change, 
$f^{-1}(\widehat U)\to \widehat V$ is open. As 
$V$ is by definition the subset of $f^{-1}(U)$
defined by the condition $\abs \psi <1$, its stable 
completion $\widehat V$ is open in $f^{-1}(\widehat U)$, 
so $f$ induces an open map from $\widehat V$ to $\widehat U$, hence
a homeomorphism $\widehat V\simeq \widehat U$. 
%
\end{exem}

%

\begin{lemm}\label{lemm-nash-topology}
Let $f\colon Y\to X$ a morphism between $K$-algebraic varieties, 
and let $U$ and $V$ be $K$-definable subsets of $X$ and $Y$ respectively, 
such that $f$ induces a Nash isomorphism between 
$V$ and $U$. 

\begin{enumerate}[1]
\item The set $U$ admits a nice $K$-definable exhaustion
if and only if so does $V$. 

\item The set $\widehat V$ is open and closed in 
$f^{-1}(\widehat U)$. 

\item The set $\widehat U$ is open in $\widehat X$
if and only if $\widehat V$ is open in $\widehat Y$. 
\end{enumerate}
\end{lemm}

\begin{proof}
By definition $f$ induces a homeomorphism between 
$\widehat V$ and $\widehat U$. 
So if $(U_t)_t$ is a nice $K$-definable exhaustion of $U$
then $(f^{-1}(U_t)\cap V)_t$ is a nice $K$-deifnable
exhaustion of $\widehat V$, and if $(V_t)_t$ is a nice
$K$-definable exhaustion of $V$ then $(f(V_t))_t$
is a nice $K$-definable exhaustion of $U$, whence (1).

Let $\sigma \colon U\to V$
be the converse bijection of $f|_V\colon V\to U$. As our varieties are separated by definition,
their  stable completions are Hausdorff, so
the continuous section 
$\sigma$ of $f^{-1}(\widehat U)\to \widehat U$
has closed image; this means that $\widehat V$ is closed
in $f^{-1}(\widehat U)$. 

It remains to prove that $\widehat V$
is open in $f^{-1}(\widehat U)$ as well as (3).
For both properties we can replace 
$Y$ by the étale locus of $f$ (which contains $V$)
and assume that $f$ is étale. Then by the local
structure of an étale map there exists
a finite cover $(Y_i)$
of $Y$ by $K$-Zariski open affine subsets and, for each $i$, 
an affine $K$-Zariski open subset $X_i$ of $X$ such that $f$
induces an elementary étale map $Y_i\to X_i$; \ie, 
$X_i=\spec A_i$ and $Y_i$ is the open subset $D(P_i(T))$ of $\spec A_i[T]/P_i$ for some
monic polynomial $P_i\in A_i[T]$. 

Set $U_i=U\cap X_i\cap \sigma^{-1}(Y_i)$ and $V_i=\sigma(U_i)=V\cap Y_i$. 
It suffices to prove that 
$\widehat{V_i}$ is open in $f^{-1}(\widehat{U_i})\cap \widehat{Y_i}$
for all $i$, 
and that each $\widehat{V_i}$ is open in $\widehat{Y_i}$ if and only if 
$\widehat{U_i}$ is open in $\widehat{X_i}$. So we can
assume that $X=\spec A$ and $Y$ is the open subset $D(P'(T))$ of
$\overline Y:=\spec A[T]/P$ for some monic polynomial $P=T^n+\sum_{1\leq i\leq n}
 a_i T^{n-i}\in A[T]$; we still denote by $f$ the structure map from
 $\overline Y$ to $X$ (that extends the original $f$). 

Let $\Phi\colon X\to \Gamma$ the $K$-definable map $\max \abs{a_i}^{1/i}$
from $X$ to $\Gamma$, and let $\Psi$ be the map
$\abs{T-T\circ \sigma\circ f}$ from $\overline Y$ to $\Gamma$. 
Denote by $\Omega$ the $K$-definable subset
of $Y$ defined by the inequality 
$\Psi< \abs{P'(T)}\cdot \Phi^{1-n}$. 
By a standard computation, for all models
$M$ of $\acvf$ containing $K$ and all 
$x\in U(M)$ such that $\sigma(x)\in \Omega$, there are no roots
in $M$ of the polynomial
$T^n+\sum_i a_i(x)T^{n-i}$ different from $T(\sigma(x))$ and
whose distance to the latter is 
strictly bounded by $P'(T(\sigma(x))(\max_i \abs{a_i(x)}^{1/i})^{1-n}$. 
Otherwise said $\sigma(x)$ is the unique point
of $f^{-1}(x)$ inside $\Omega$. 
As a consequence one has
$\Omega\cap f^{-1}(U)=V$, so 
$\widehat \Omega \cap f^{-1}(\widehat U)=\widehat V$, hence
$\widehat V$ is open in $f^{-1}(\widehat U)$ inside $\widehat{\overline Y}$, 
and a fortiori inside $\widehat Y$. 
It clearly follows that if $\widehat U$ is open in $\widehat X$ then $\widehat V$
is open in $\widehat Y$. 

Conversely if $\widehat V$ is open in $\widehat Y$
then $\widehat U=f(\widehat V)$ is open in $\widehat X$
since étale maps are flat and thus
open (Proposition \ref{flat-open}).
\end{proof}

\subsection{}\label{notation-nash}
Let $K$ be a valued field, let $X$ be a $K$-algebraic variety and let $U$ be a $K$-definable subset of $X$.

We will use the notation $\mathscr Z_K(U)$ (where $X$ is implicitly understood) 
for 
the 
ordered and filtered set of all $K$-Zariski open subsets of
$X$ that contain $U$, and $\mathscr E_K(U)$ for the ordered
collection of
triples 
$(Y, V,f)$ where $Y$ is a $K$-variety, 
$V$ is a $K$-definable subset of $Y$
and 
$f\colon Y\to X$
is a map inducing a Nash isomorphism $V\simeq U$. 
We observe that $\mathscr E_K(U)$ is filtered: 
if $(Y,V,f)$ and $(Z,W,g)$ belong to
$\mathscr E_K(U)$ then $(Y\times_X Z, W\times_U V, h)$, 
where $h\colon Y\times_X Z\to X$ is the structure map, 
belongs to $\mathscr E_K(U)$ and refine both
$(Y,V,f)$ and $(Z,W,g)$. 
We also remark that one 
can see $\mathscr Z_K(U)$ as contained in $\mathscr E_K(U)$. 

We then set 
\[K[U]=\colim_{Y\in \mathscr Z_K(U)}K[Y]
\;\text{and}\;\nash  KU=\colim_{(Y,V,f)\in \mathscr E_K(U)}
K[Y].\]
By construction $\nash KU$ is a $K$-algebra, and by flatness of étale maps
$K[U]$ is a sub-algebra of $\nash K U$. It
follows from the 
construction that every $\phi\in \nash KU$
induces $K$-definable maps $\phi\colon 
U\to \A^1$ and $\abs \phi \colon U\to \Gamma_0$. 

Let $(Y,V,f)\in \mathscr E_K(U)$. 
By Lemma \ref{lemm-same-henselization} $K(f(v))\h=K(v)\h$ for
every $v$ in $V$; it follows that $\phi(u)\in K(u)\h$
 for all $u \in U$ and all $\phi\in \nash KU$; of course if $\phi\in K[U]$ then 
$\phi(u)\in K(u)$.

Note that for every $(Y,V,f)\in \mathscr E_K(U)$ there is a natural map
from $\nash KV$ to $\nash KU$ which is an isomorphism; it is compatible
with the identifications between $K(f(v))\h$ and $K(v)\h$ for every $v\in V$. 

\subsubsection{}\label{exhaustion-rvku}
Assume
that $U$ admits a nice $L$-definable exhaustion
for some valued extension $L$ of $K$ (this amounts to requiring
that it admits such an exhaustion over $K$ except in the trivially valued case, see
Remark \ref{exhaustion-scalar-extension}). 
If $(Y,V,f)$ belongs to $\mathscr E_U(K)$, it follows from 
Lemma \ref{lemm-nash-topology}
that $V$ also admits a nice $L$-definable exhaustion, so $\widehat V$ has compact closure in $\widehat Y$ 
in view of Lemma \ref{lemm-exhaustion}. Hence $\phi|_{\widehat V}$ is bounded for every $\phi \in K[Y]$;
as a consequence, the function $\abs \phi\colon U\to \Gamma$ is bounded
for every $\phi\in \nash KU$.

We then set
\[\redgrad KU=
\coprod_{r\in \abs{K^\times}^\Q}\{\phi\in  K[U],
\,\|\phi\|_{\infty, U}\leq r\}/
\{\phi\in K[U], \,\|\phi\|_{\infty, U}< r\};\] 
and
\[\rednash KU =\coprod_{r\in \abs{K^\times}^\Q}\{\phi\in \nash KU,
\,\|\phi\|_{\infty, U}\leq r\}/
\{\phi\in \nash KU, \,\|\phi\|_{\infty, U}< r\}. \] 
These are generalized reduced $\RV(K)$-algebras,
$\redgrad KU$ embeds naturally in $\rednash KU$, 
and $\rednash KU$ is the colimit of the $\redgrad KV$, 
taken on all triples $(Y,V,f)\in \mathscr E_K(U)$. 
If $(Y,V,f)\in \mathscr E_K(U)$ then 
$\redgrad KU$ embeds into $\redgrad KV$, and
$\rednash KU\simeq \rednash KV$. 

 If $\phi\in \nash KU$
and $\|\phi\|_{U,\infty}\neq 0$, we shall often denote by $\rv_U(\phi)$
the image of $\phi$ in the
$\|\phi\|_{U,\infty}$-summand of $\Lambda_K(U)$.

\subsubsection{}\label{nash-functions}
Assume that $U(M)$ is open in $X(M)$ for the valuative topology for some 
model $M$ of $\acvf$ containing $K$ (this is then the case for all such models), and that $U$ is contained in the reduced locus on $X$
(but we do not assume anymore that $U$ admits a nice
exhaustion over some extension of $K$). 
We observe then that if $(Y,V,f)$ belongs to 
$\mathscr E_K(U)$ then $V$ is contained in the reduced locus of $Y$ (by étaleness of $f$)
and that $V(M)$ is open in $Y(M)$ for any model $M$ of $\acvf$
containing $K$ (as $V(M)$ is open in $f^{-1}(U(M))$, for $\widehat V$ is open in
$f^{-1}(\widehat U)$ by Lemma \ref{lemm-nash-topology}). 

Then $\nash KU$ embeds into the set of $K$-definable
functions on $U$. Indeed, let $\phi \in \nash KU$ inducing the zero map on $U$. 
We want to prove that $\phi=0$. For proving this 
we can replace $(X,U)$ by $(Y,V)$ for any $(Y,V,f)$
in $\mathscr E_K(U)$; so we can assume
that $\phi$ comes from some function in $K[X]$, which we still denote by $\phi$; 
and we can also assume by shrinking it if needed that
$X$ is reduced. 
As $U$ is open for the valuative topology, 
its intersection with every irreducible component $Z$ of $X$ contains some point $z$ whose $K$-Zariski-closure 
is $Z$ \cite[Lemme 1.1]{ducros2012b},
It follows that the pointwise vanishing locus of $\phi$, which
contains $U$ by assumption,
contains a $K$-Zariski neighborhood $X'$ of $U$ in $X$; 
as $X$ is reduced $\phi|_{X'}=0$; a fortiori $\phi=0$ when viewed in $\nash KU$. 

In this setting we may think of $\nash KU$ as a model-theoretic
substitute for the ring of analytic functions on $U$, 
which is an $\acvf$ analog of the 
ring of Nash functions in $\mathrm{RCF}$. We will therefore call
the elements of $\nash KU$ \emph{Nash functions}
on $U$. 

Be aware that we do not pretend to have defined a \emph{sheaf} (for whatever Grothendieck topology) of Nash functions; 
since we will morally only deal with ``Stein" spaces, global rings defined as above will be sufficient for our purposes.

\begin{lemm}\label{cofinality}
Let $K$ be a
henselian 
valued field, let $X$ be a $K$-algebraic variety and let $U$ be a $K$-definable subset of $X$. 
Let $L$ be a finite extension of $K$. 

\begin{enumerate}[1]
\item Zariski-open subsets defined over $K$ are cofinal in $\mathscr Z_L(U)$. 
\item 
Triples defined over $K$ are cofinal in 
$\mathscr E_L(U)$. 
\item The natural maps 
\[K[U]\otimes_K L\to L[U]\;\;\;\text{and}\;\;\;\nash KU\otimes_K L\to
\nash LU\] are isomorphisms. 
\end{enumerate}
\end{lemm}

\begin{proof}
 Let $M$
be a model of $\acvf$ containing $L$ and let $E$ be the set of $K$-embeddings
from $L$ into $M$; we denote by $\iota$
the inclusion $L\subset M$. 

We start with (1). Let $Y\in \mathscr Z_L(U)$. 
For every $\sigma\in E$, let $Y^\sigma$ be the $M$-Zariski open subset of $X$ 
deduced from $Y$ through $\sigma$. Then $\bigcap_{\sigma \in E}Y^\sigma$ 
is a $K$-definable Zariski-open subset of $X$, containing $U$ and contained in $Y$, 
whence (1).

Now let us prove (2); the idea is basically the same, except
that intersections has to be replaced by fiber products. So 
let $(Y,V,f)\in \mathscr E_L(U)$. Our purpose is to exhibit a
 triple $(Z,W,g)$ in $\mathscr E_K(U)$ which refines $(Y,V,f)$. 
We first  note that as $f$ is étale at every point of $V$ we can shrink $Y$ so that 
$f$ is étale on the whole of $Y$. For every $\sigma\in E$, let $(Y^\sigma, V^\sigma, f^\sigma)$ be the
element of $\mathscr E_M(U)$ deduced from $(Y,V,f)$ through $\sigma$, and set 
$Z=\prod_X Y^\sigma$; note that 
$(Y^\iota, V^\iota, f^\iota)=(Y,V,f)$. Then $Z\to X$ is an étale map. As $Y$ is in fact defined over the separable 
closure of $K$ in $L$ by topological invariance of the étale site, étale descent shows that the $X$-étale scheme $Z$
is defined over $K$; let $g$ denote the structure map from $Z$ to $X$. 
Let $W$ be the $K$-definable subset $\prod_X V^\sigma$ of $Z$. We are going to show that 
$(Z,W,g)$ fulfills our requirements. One has by definition $g(W)\subset f(V)=U$. 

First we note that the projection $p\colon Z\to Y$ (corresponding to $\sigma=\iota$) is an $X$-morphism, so
is étale;  moreover $f\circ p=g$, and $p$ is defined over $L$. Hence it suffices to prove that $g$ induces an  
$L$-definable Nash isomorphism $W\simeq U$. 

Let $u\in U(M)$. As $f$ induces an $L$-definable Nash isomorphism $V\to U$, the map $f^\sigma$ induces for every 
$\sigma$ an  $M$-definable Nash isomorphism   $V^\sigma \to U$; for every $\sigma$, let $v^\sigma$ be the unique
pre-image of $u$ in $V^\sigma(M)$. Then $(v_\sigma)_\sigma$ is clearly the unique pre-image of $u$ in $W$, 
so $g$ induces a $K$-definable isomorphism $W\simeq U$.

The map $\widehat W\to \widehat U$ induced by $g$ is then a continuous bijection. 
Let us prove that this map is closed -- it will then be a homeomorphism.  
Let $(U_t)$ be a $K$-definable nice exhaustion of $U$; for each $t$, set
$V_t=f^{-1}(U_t)$. 
Using the definition of (definably) closedness and compactness through limits of
definable types, we see that it suffices to prove that $g^{-1}(\widehat{U_t})$ is
compact for every $t$. But $g^{-1}(\widehat {U_t})$ is equal to 
the pre-image of $\prod _{\widehat X} (f^\sigma)^{-1}(U_t)$ under the canonical 
map $\widehat Z\to \prod_{\widehat X}\widehat{Y^\sigma}$, which is proper (one can check it using
the characterization of compactness through limits of definable types).
As each of the $(f^\sigma)^{-1}(U_t)$
is compact, $g^{-1}(\widehat {U_t})$
is compact and (2) is proved.

Now (3) just follows from (1) and (2) by a limit argument, in view of
the fact that 
$L[Z]=K[Z]\otimes_K L$ for every $K$-variety $Z$. 
\end{proof}

\subsection{Reminders on tame extensions}
Let $K$ be a henselian valued field and let $K^{\mathrm{sep}}$ be a separable closure of $K$; 
set $G=\mathrm{Gal}(K^{\mathrm{sep}}/K)$. 

\subsubsection{}
Let us first remind the classical formulation of the theory. 

\paragraph{}
One defines the \emph{inertia}
subgroup of $G$ as the kernel $I$ of the natural action of $G$
on the residuel field $\k(K^{\mathrm{sep}})$, and an algebraic extension $L$
of
$K$ is called \emph{unramified} if $I$ acts trivially on the corresponding 
$G$-module $\Hom_K(L,K^{\mathrm{sep}})$. If $[L:K]<\infty$ this amounts to 
requiring that $\k(L)$ be  separable and of degree $[L:K]$ over $\k(K)$, and also 
that the valuation ring of $L$ be étale over that of $K$. 

\paragraph{}
One defines the \emph{ramification} subgroup $W$ of $G$ as follows: it is trivial 
if $\k(K)$ is of characteristic $0$; if $\k(K)$ is of characteristic $p>0$ then $W$
is the unique pro-$p$-Sylow of $I$ (proving the uniqueness is part of the theory). 
An algebraic extension $L$
of
$K$ is called \emph{tamely ramified} if $W$ acts trivially on the corresponding 
$G$-module $\Hom_K(L,K^{\mathrm{sep}})$. If $[L:K]<\infty$ this amounts to 
requiring that $\k(L)$ be  separable over $\k(K)$, that $(\abs{L^\times}:\abs{K^\times})$
be prime to the residue characteristic exponent, and that 
\[|\k(L):\k(K)]\cdot(\abs{L^\times}:\abs{K^\times})=[L:K].\]

If $K$ is of residue characteristic zero every algebraic extension of $K$ is tame.
In general, every finite extension of $K$
whose degree is prime to the residue characteristic exponent of $K$ 
is tame.

\subsubsection{}
There is an alternative handling of tame ramification, 
carried out by the first author in the 
first section of \cite{ducros2013b}, which uses $\RV$ and the theory of generalized ring 
extensions, and make the theory fully analogous to that of unramified extensions. 

In this approach the group $W$ is directly defined as the kernel of the natural 
action of $G$ on the generalized residue field $\RV(K^{\mathrm{sep}})$, and one proves that
a finite extension $L$ of $K$ is tamely ramified if and only if 
$\RV(L)$ is separable and of degree $[L:K]$ over $\RV(K)$.

\begin{coro}\label{coro-tame-reduction}
Let $K$ be a
henselian
valued field, let $X$ be a $K$-algebraic variety and let $U$ be a $K$-definable subset of $X$
admitting a nice $K$-definable exhaustion. 
Let $L$ be a tamely ramified finite extension of $K$. 
The natural maps
\[\redgrad KU \otimes_{\RV(K)}\RV(L)
\to \redgrad LU\;\;\;\text{and}\;\;\;
\rednash KU \otimes_{\RV(K)}\RV(L)
\to \rednash LU\]
are isomorphisms. 
\end{coro}

\begin{proof}

The generalized ring $\rednash KU$ is the colimit
of the $\redgrad KV$ for  $(Y,V,f)$ running through $\mathscr E_K(U)$, 
and by \ref{cofinality} (2) $\rednash LU$ is the colimit
of the $\redgrad LV$ for $(Y,V,f)$ also running through $\mathscr E_K(U)$, 
so it suffices to prove that 
the canonical map
$\redgrad KU \otimes_{\RV(K)}\RV(L)
\to \redgrad LU$ is an isomorphism. 

Let $a_1,\ldots, a_r$ be elements of $L^\times$ such that 
$(\rv(a_i))_i$ is a basis of $\RV(L)$ over $\RV(K)$. 
As $L$ is tamely ramified over $K$ it follows from 
\cite[2.21 and Lemme 2.3]{ducros2013b}
that $(a_i)$ is a basis of $L$ over $K$ and that for
every henselian valued
extension $F$ of $K$ and every element
$\lambda =\sum \lambda_i \otimes a_i$ of 
the tensor product $L\otimes_K F$ 
one has $\|\lambda\|_\infty=\max_i \abs{\lambda_i}\cdot \abs{a_i}$.  
Here $\|\cdot\|_\infty$ denotes the spectral norm of
$L\otimes_K F$; the latter ring is a product  $\prod F_j$ of finitely
many finite separable extensions of $F$, and $\|\lambda\|_\infty$
is then equal to the maximum of the absolute values of the 
components of $\lambda$ for the decomposition $L\otimes_k F=\prod F_j$. 

By Lemma\ref{cofinality} (3) one has 
$L[U]=L\otimes_K K[U]$, so $L[U]$
is a free module over $K[U]$ with basis $(a_i)$. 
Let $f$ be an element of $L[U]$. 
Let us write $f=\sum a_i f_i$ with each $f_i$ in $K[U]$. 
Let $p$ be a quantifier-free type over $K$ lying on $U$
and let 
$\{p_j\}$ be the set of
quantifier-free types over $L$ lying above $p$; there is 
a natural isomorphism
$L\otimes_K K(p)\h\simeq \prod L(p_j)\h$, which maps
$\sum_i a_i \otimes f_i(p)$ to $(f(p_j))_j$. 
As a consequence, 
\[\max_j \abs{f(p_j)}=\max_i \abs{a_i}\cdot \abs{f_i(p)}.\]
It follows that 
\[ \|f\|_{U,\infty}=\max_i \abs {a_i}\cdot \|f_i\|_{U,\infty}.\]
But this immediately implies that
$(\rv(a_i))_i$ is a basis of 
$\redgrad LU$ over $\redgrad KU$, which ends the proof. 
\end{proof}

\section{Abstract open
polydiscs and their Nash functions}\label{abstract-polydiscs}

\subsection{}\label{taylor-expansion}
We fix from now on a henselian
valued field $K$, 
an algebraic variety $X$ over $K$, a $K$-definable subset 
$U$ of $X$, and an algebraic map
$f\colon X\to \A^n_K$ inducing
a Nash isomorphism
between $U$ and a $K$-definable open polydisc $D$ centered at the
origin; since 
$D$ is $K$-definable, its polyradius $(r_1,\ldots, r_n)$ has components 
in
$\abs{K^\times}^\Q$. 
The variety $\A^n$ is reduced, $\widehat D$ is open 
in $\widehat{\A^n}$ and $D$ admits a nice definable exhaustion 
over any non-trivially valued extension of $K$. 

We can thus apply \ref{exhaustion-rvku} and \ref{nash-functions}. 
It follows that elements of $\nash KU$ are actual functions on $U$
(which are called Nash functions), that they are bounded, and that 
$K[U]$ and $
\nash KU$ give rise
to two generalized rings $\redgrad KU$ and $\rednash KU$. 
The map $f$ induces
isomorphisms $\nash KU\simeq\nash KD$ and $\rednash KU\simeq \rednash KD$.

We denote by $u$ the unique pre-image of $0$ on $X$. 
The point $u$ is $K$-definable, and
$K(u)$ is  finite and separable over $K$ as $f$ is étale at $u$;
as we assume here that $K$ is henselian, 
this implies that $u$ is a $K$-point.

\begin{enonce}[remark]{Conventions}
For simplicity we shall often write elements
of $\mathscr E_K(U)$ as pairs $(Y,V)$, the map $g\colon Y\to X$ being implicit; 
and the image of $f$ in $K[Y]$, which is rigorously equal to $f\circ g$, will still be denoted by $f$; 
note that with this convention $f$ induces a
Nash isomorphism between $V$ and $D$.

For studying functions of $\nash KU$ it will often be useful to replace $(X,U)$
by some suitable $(Y,V)$ in $\mathscr E_K(U)$ fulfiling some property in which case
we will say ``up to refining $(X,U)$ we can assume that this or that property holds...".

Typically, we will often 
refine $(X,U)$ so that: 
\begin{itemize}[label=$\diamond$] 
\item $f$ is étale, hence $X$ is smooth -- this can be achieved by replacing $X$
by the étale locus of $f$; 

\item $X$ is connected -- as $\widehat U$ is connected, being homeomorphic to the stable
completion of an open polydisc, this can be achieved by replacing $X$ by its unique connected
component containing $U$, which is $K$-definable; 
\item a given element $\phi$ of
$\nash KU$ comes from $K[X]$; note that under the assumption that $X$
is smooth and connected $K[X]$ embeds into $K[U]\subset \nash KU$, for it embeds in 
$K[X']$ for every $K$-Zariski neighborhood $X'$ of $X$, so in this situation we will
consider $K[X]$ as a subring of $\nash KU$ and simply say that $\phi$ belongs to 
$K[X]$;  
\end{itemize}
If  $n\geq 1$ and if $f$ is assumed to be étale and $X$ connected, then $X':=\{f_n=0\}$
is a smooth hypersurface of $X$, equipped with an étale map $(f_1,\ldots, f_{n-1})$
to $\A^{n-1}$, and it will also be possible to also assume that $X'$ is connected: 
indeed, if $U'$ denotes the intersection $X'\cap U$ then $(f_1,\ldots, f_{n-1})$
induces a homeomorphism between $\widehat {U'}$ and the stable completion of an 
$(n-1)$-dimensional open polydisc, so $\widehat {U'}$ is connected, hence $U$ intersects
only one irreducible component of $X'$, so one can remove the other components from $X$.

When we perform such a replacement of $(X,U)$ by $(Y,V)$, we of course replace
$u$ by its unique pre-image $v$ on $V$. This operation does not preserve the $K$-algebraic
local ring $\mathscr O_{X,u}$ in general, 
but it preserves its henselization, and a fortiori
its (formal) completion: indeed, by étaleness
at $v$ the map $Y\to X$ induces an isomorphism
$\mathscr O_{X,u}\h\simeq \mathscr O_{Y,v}\h$; it therefore also preserves its
(formal) completion. 

Of course we will also proceed to such replacements for studying $K[U]$, in which case they are much simpler: one only
has to replace $X$ with a $K$-Zariski open neighborhood of $U$; for instance, the connected component of the étale locus of $f$ 
that contains $U$. 

We shall say that a $n$-uple $g=(g_1,\ldots, g_n)$ of elements
of $\nash KU$ induces a Nash isomorphism between 
$U$ and some $K$-definable subset $\Delta$ of $\A^n$ (in practice, $\Delta$ will be an open polydisc)
if there is some $(Y,V)\in \mathscr E_K(U)$ such that $g_i$
is induced for every $i$ by some element of $K[Y]$ (still
denoted by $g_i$) and $g$ induces a  $K$-definable Nash
isomorphism between $V$ and $\Delta$. This will then be the case for every $(Y,V)\in \mathscr E_K(U)$ such that $g_i\in K[Y]$
for all $i$ (by étale descent of étaleness). If $g$ induces a
Nash isomorphism between 
$U$ and $\Delta$ it induces a
homeomorphism 
between $\widehat U$ and $\widehat \Delta$.

\end{enonce}

\subsubsection{}
Let $\phi \in \nash KU$. We will associate to it a ``Taylor expansion at $u$" as follows. 
First, by refining $(X,U)$ we can assume that $\phi \in K[U]$. 
Now as $f$ induces an isomorphism 
$\widehat{\mathscr O_{\A^n, 0}}\simeq \widehat{\mathscr O_{X,u}}$
the image of $\phi$ in  $\widehat{\mathscr O_{X,u}}$
can be naturally written as a formal power series $\sum a_I f^I$.

Let $\phi$ be an element 
of $\nash KU$ whose Taylor expansion at $u$ is zero, 
and let us prove that $\phi=0$. 
By refining $(X,U)$ we can assume that
$f$ is étale, $X$ is connected and $\phi\in K[X]$. 
As $X$ is smooth and connected, 
taking the Taylor expansion at $u$ defines an embedding
$K[X]\hookrightarrow K[\![f]\!]$, so $\phi=0$. 

As a consequence, taking the Taylor expansion at $u$ defines an embedding
from
$\nash KU$
into $K[\![f]\!]$
( in particular, $\nash KU$
is a domain), and we shall often use it to see Nash functions on $U$
as power series. 

\begin{lemm}\label{nash-divisibility}
Let $\phi=\sum a_I f^I$ be a function 
in $\nash KU$.
Assume $n\geq 1$ and
write formally the power series
$\phi$ as $\sum_i A_if_1^i$, where each $A_i$
belongs to $K[\![f_1,\ldots, f_{n-1}]\!]$. 
Let $X'$ be the zero-locus of $f_n$ on $X$
and set $U'=X'\cap U$. 

\begin{enumerate}[1]
\item Each $A_i$ is the Taylor expansion 
at $u$ of some element $\phi_i$ of 
$\nash K{U'}$, and one has $\phi_0=\phi|_{U'}$; if moreover
$f$ is étale, $X$ and $X'$
are connected and $\phi\in K[X]$, then each $\phi_i$ 
belongs to $K[X']$. 

\item The following are equivalent: 
\begin{enumerate}[j]
\item $\phi \in f_n\nash KU$; 
\item $A_0=0$; 
\item $\phi|_{U'}=0$. 
\end{enumerate}
If these conditions are fulfilled the
Taylor expansion of the function 
$\phi/f_n$ of $\nash KU$ at $u$ is equal to
$\sum_{i\geq 1}A_if_n^{i-1}$. If
we assume moreover that $f$ is étale, $X$
and $X'$ are
connected and
$\phi\in K[X]$,
then $\phi\in f_nK[X]$. 
\end{enumerate}
\end{lemm}

\begin{enonce}[remark]{Comments}
By construction $(f_1,\ldots, f_{n-1})$ induces 
a Nash isomorphism between 
$U'$ and the open polydisc of polyradius 
$(r_1,\ldots, r_{n-1})$, hence
$\nash K{U'}$
embeds into $K[\!f_1,\ldots, f_{n-1}]\!]$ trough
Taylor expansion at $u$ and the lemma's statement
makes sense. 
\end{enonce}

\begin{proof}[Proof of Lemma \ref{nash-divisibility}]
By refining $(X,U)$ we can assume that 
$f$ is étale, $X$ is connected and $\phi\in K[X]$; 
then $X$ is smooth and as the pre-image of
a coordinate hyperplane under an étale map, $X'$ is smooth
as well. 
The formal
completion of $X$ along its closed subscheme
$X'$ can be identified to $X'\times_K \spf K[\![f_n]\!]$. 
Let $\sum \phi_i f_n^i$ be 
the image of $\phi$ in $\mathscr O(X'\times_K \spf K[\![f_n]\!])$, 
where each $\phi_i$ is a regular function on $Y$. 
One can
see each $\phi_i$ as an element of $K[X']$, and by construction its
expansion in $\widehat {\mathscr O_{X',u}}\simeq K[\![f_1,\ldots, f_{n-1}]\!]$
is precisely $A_i$ ; moreover, $\phi_0=\phi|_{X'}$. This ends the proof of (1). 

Let us prove (2). 
As the Taylor expansion map $\nash K{U'}\to
K[\![f_1,\ldots, f_{n-1}]\!]$ is injective, (ii)$\iff$(iii). 
It is clear that (i)$\Rightarrow$(iii).

Now assume that (iii) holds. Up to refining $(X,U)$ we can assume that 
$X'$ is connected. Then
$K[X']$ embeds into $\widehat {\mathscr O_{X',u}}$, so the
equality $A_0=0$ implies that $\phi_0=0$, hence that $\phi$
belongs to $f_n K[X]$
and that  
the Taylor expansion of $\phi/f_n$ at $u$
is equal to $\sum_{i\geq 1}A_if_n^{i-1}$.
\end{proof}

\subsection{}
Assume that $n\geq 1$. As in the proof
of the Lemma above, let
us denote by $X'$ the closed subscheme $\{f_n=0\}$, 
of $X$, and
by $U'$ the intersection $X'\cap U$. Set $f'=(f_1,\ldots,f_{n-1})$. 
The map $f'\colon X'\to \A^{n-1}$
induces a
Nash isomorphism
between $U'$ and the open polydisc $D'$ of polyradius $r':=(r_1,\ldots, r_{n-1})$.
Let $\sigma\colon D'\to U'$ be the converse bijection. 
Then $\rho:=\sigma\circ f'\colon U\to Z$ is a $K$-definable retraction 
of the inclusion $Z\hookrightarrow U$, which induces a continuous retraction 
of $\widehat Z\hookrightarrow \widehat U$. 
Moreover, since $f'|_{U'}$ preserves henselizations by Lemma \ref{lemm-same-henselization}, 
so does $\sigma$, whence a natural embedding of valued
fields
$K(\rho(x))\h\hookrightarrow K(x)\h$ for every $x\in U$. 

Let $\psi$ be an element of $\nash K{U'}$. For every 
$x\in U$ we have a well-defined element 
$\psi(\rho(x))$ of 
$K(\rho(x))\h\subset K(x)\h$. 
We note that $(\psi\circ \rho)|_Z=\psi$.

\begin{prop}\label{restriction-nash-functions}
We keep the notation above. 

\begin{enumerate}[1]

\item For every $\psi \in \nash K{U'}$ the function 
$\psi \circ \rho$ belongs to $\nash KU$. 

\item Let $\phi=\sum a_If^I$ be an element of $\nash KU$.
The following are equivalent: 

\begin{enumerate}[j]
\item there exists $\psi \in \nash K{U'}$ such that 
$\phi=\psi \circ \rho$; 
\item the function $f_n$ does not actually appear 
in the expansion $\sum_I a_I f^I$ (otherwise said $a_I=0$
as soon as the $n$-th component of $I$ is non-zero).

\end{enumerate}
Moreover if these conditions are satisfied then
$\psi$ is necessarily equal to $\phi|_{U'}$ and
its Taylor expansion at $u$ is also $\sum_I a_I f^I$, viewed as a power series
in $(f_1,\ldots,f_{n-1})$. 
\end{enumerate}
\end{prop}

\begin{proof}
By refining $(X,U)$ we can assume that $f$
is étale (then $X$ is smooth).

\subsubsection{A general construction}\label{refinement-x'-u'-f'}
Let $(Y',V',g')$ be an element of 
$\mathscr E_K(U')$, and suppose
that $g'$ is étale. 
Set $Z=X\times_{\A^{n-1}}Y'$ where the fiber product is understood along
$f'\colon X\to \A^{n-1}$ and the étale map $f'\circ g'\colon Y'\to \A^{n-1}$. 
The $K$-variety $Z$ is equipped with two natural
morphisms $h\colon Z\to X$
(which is étale)
and $p\colon Z\to Y'$. 
The
composition $Y'\to X'\hookrightarrow X$ furnishes a section $\tau$ of $p$; 
let $Z'$ denote its image. 
The algebraic map $\tau\circ p$ is then a retraction of
the closed embedding $Z'\hookrightarrow Z$. 
Let $V$ be the $K$-definable subset $U\times_{\A^{n-1}}V'$ of $Z=X\times_{\A^{n-1}}Y'$. 

A point of $Z$ is a pair $(x,y')$
with $x\in X$ and $y \in Y'$
such that $f'(x)=f'(g'(y))$. 
A point $(x,y')$ of $Z$ belongs to $Z'$ 
if and only if and only if $x=g'(y')$, 
and 
the algebraic retraction $\tau \circ \rho$ maps
an arbirary point $(x,y')$ of $Z$ to $(g'(y'), y')$. 
The natural isomorphism between $Y'$ and $Z'$
is $y'\mapsto (g'(y'), y')$; hence modulo this isomorphism 
the map $h$ corresponds to $g'$. 

A point $(x,y')$ of $Z$ 
belongs to $V$ if and only if 
$x\in U$ and $y'\in V'$; 
a point of $V$ can therefore be described
as a pair $(u,v')$ with $u\in U, v'\in V'$ such that 
$f'(g'(v))=f'(u)$; this exactly means that $g'(v')=\rho(u)$. 
A point $(u,v')$ of $V$ belongs to $Z'$ if and only if 
$u=g'(v')$, which amounts to saying that 
$u=\rho(u)$, or also that $u\in U'$. 
As a consequence $Z'\cap V$ is exactly 
the pre-image of $U$' under the first projection, which is also the zero-locus
of $f_n\circ h$ on $V$. 

The algebraic retraction $\tau \circ p$ maps
an arbitrary point $(u,v')$ of $V$ to
the pair $(g'(v'), v')=(\rho(u), v')$.

By assumption $g'$ induces a
Nash isomorphism
between $V'$ and $U'$; let 
$\sigma'$ be the converse bijection. 
The map $h$ then induces a 
$K$-definable bijection $V\simeq U$ with 
converse bijection $\beta \colon 
u\mapsto (u, \sigma'(\rho(u))$. Both components of 
$\beta$
induce a contiuous map at the level of stable completions, 
so 
$\beta$ itself induces
a continuous map $\widehat U\to \widehat V$
(because the natural map $\widehat Z\to \widehat {X}\times_{\widehat {\A^{n-1}}}
\widehat {Y'}$ is topologically proper, as one can check by using the definition 
of definable compactness through definable types); therefore 
$h$ induces a Nash isomorphism between $V$ and $U$. 

The Zariski-closed subsets $Z'$ and $\{f_n \circ h=0\}$ of $Z$ have the same intersection 
with $V$. This intersection identifies through $h$ with $U'$, which itself identifies through
$f'$ with an $(n-1)$-dimensional open polydisc; it thus admits a point $v$ (in some model)
such that $K(v)$ has transcendence degree $n-1$ over $K$. Hence the smooth $K$-varieties
$Z'$ and $\{h_n\circ f=0\}$ have at least a common irreducible (and connected) component
$Z''$ which meets $V$. As the stable completion of $Z'\cap V=\{f_n\circ h=0\}\cap V$ is connected (for it is 
definably homeomorphic to the stable completion of an open polydisc), 
no other connected connected component of  $Z'$ or
$\{f_n\circ h=0\}$ can meet $V$. Hence if we denote by $\Omega$ the 
connected component of 
\[(Z\setminus (\{f_n\circ h=0\}\cup Z'))\cup Z''\]
that contains $Z''$, 
the following holds: $(\Omega, V,h)$ is an objet of $\mathscr E_K(U)$, the
map $h$ is étale, $\Omega$ is connected, and the natural retraction from $V$ to
$\{f_n\circ h=0\}\cap V$ is induced by an algebraic retraction 
from $\Omega$ to its 
Zariski-closed subset $\Omega':\{f_n\circ h=0\}$. 
By construction $\Omega'$ can be identified with an open subset
of $Y'$ containing $V'$, and $h|_{\Omega'}$ can be identified with $g'$.

\subsubsection{Proof of (1)}
One can prove (1) after refining $(X,U)$. 
So starting from some $(Y',V',g')$ in $\mathscr E_K(U')$ such that
$g'$ is étale and $\psi\in K[Y']$, and using the construction of
$(\Omega, V, h)$ carried out in
\ref{refinement-x'-u'-f'}, we can then assume that 
$X$ is connected, 
that the retraction $\rho$ is induced by a $K$-algebraic retraction from $
X\to X'$ (so $X'$ is connected) 
and that 
$\psi\in K[X']$. 
But then $\psi \circ \rho\in K[X]$ and we are done.

\subsubsection{Proof of (2)}
By refining suitably 
$(X,U)$ with the help of
 \ref{refinement-x'-u'-f'},
we can assume for proving both implications
that $X$ is connected, that $\phi\in K[X]$, 
and that the retraction $\rho$ is induced by a $K$-algebraic retraction from $X\to X'$
(so $X'$ is connected), 
and for (i)$\Rightarrow$(ii) we can moreover assume that 
$\psi\in K[X']$. 
But with these new assumptions both implications are obvious, as well as the final 
statement. 
\end{proof}

\begin{lemm}\label{lemm-generation-ideal}
Let $\phi\in \nash KU$ and let 
$d\geq 0$. Assume that the Taylor expansion of $\phi$ at $u$ belongs to 
$(f_1,\ldots, f_n)^dK[\![f_1,\ldots, f_n]\!]$. Let $\mathscr I_d$ be the set of
multi-indices $I$ with $\abs I=d$.

\begin{enumerate}[1]

\item The function $\phi$ can be written
$\sum_{I\in \mathscr I_d} f^I\phi_I$ where each $\phi_I$ belongs
to $\nash KU$ and where the Taylor expansions of
the functions $f^I\phi_I$ at $u$ have pairwise disjoint supports. 
\item Assume that $n=1$, so $\mathscr I_d=\{d\}$, 
and that
$\phi \in K[U]$. Then $\phi_d$ belongs to $K[U]$. 
\end{enumerate}
\end{lemm}

\begin{enonce}[remark]{Comments}
In assertion (2) the equality $\phi=f^d\phi_d$ determines uniquely $\phi_d$ since $\nash KU$ is a domain, 
so there is no ambiguity. 
\end{enonce}

\begin{proof}[Proof of Lemma \ref{lemm-generation-ideal}]
We argue by induction on $n$. The case $n=0$
is trivial. Let us prove the case $n=1$, and write $f$ instead of $f_1$. 
By refining $(X,U)$ we can
assume that $f$ is étale, $X$ is connected and $\phi\in K[X]$. 
Then $X$ is smooth and $f$ is a local parameter at $u$ that does not vanish elsewhere 
on $U$. As a consequence the rational function 
$\phi/{f^d}$ is regular in a $K$-Zariski neighborhood of $U$, 
so it belongs to $K[U]$. This proves (2) as well as (1) when $n=1$. 
Now we assume that $n>1$ and that the Lemma
holds in dimension $n-1$. We now proceed with a second induction here, on the degree $d$. If $d=0$ there is nothing to prove
(the required decomposition is $\phi=\phi$). So we know assume that $d>0$ and that the Lemma holds in dimension $n-1$ for any degree, 
and in dimension $n$ for $d-1$.

Let $\sum a_I f^I$ be the Taylor expansion of $\phi$ at $u$.

Let $\mathscr J$ be the set of multi-indices $(i_1,\ldots, i_n)$ with $i_n=0$; 
one then has formally
\[\sum a_I f^I=\sum_{I\in\mathscr  J}a_I f^I+f_n\sum_{I\notin \mathscr  J}a_I f^{I-(0,\ldots, 0,1)}.\]

Set $\psi=\phi|_{U'}$. 
Then $\psi \in \nash K{U'}$ and its Taylor expansion at $u$
is $\sum_{I\in J}a_I f^I$. As $\psi$ belongs
to $(f_1,\ldots, f_{n-1})^dK[\![f_1,\ldots, f_{n-1}]\!]$
the induction hypothesis
ensures that $\psi$ can be written $\sum_{I\in \mathscr J_d}f^I\psi_I$ where 
$\mathscr J_d=\mathscr J\cap \mathscr I_d$, where
the $\psi_I$ 
belong to $\nash K{U'}$ and where the Tayor expansion of the $f^I\psi_I$ have pairwise disjoint
support.

For every $I\in \mathscr J_d$ set $\phi_I =\psi_I\circ \rho$. 
By proposition \ref{restriction-nash-functions} each $\phi_I$ belongs to $\nash KU$
and has the same Taylor expansion at $u$ than $\psi_I$, so 
$\sum_{I\in \mathscr J_d} f^I\phi_I$ is an element of $\nash KU$ whose Taylor expansion at
$u$ is $\sum_{I\in \mathscr J_d}a_I f^I$, and the Taylor expansions of the 
$f^I\phi_I$ for $I\in \mathscr J_d$ are pairwise disjoint. 
Set $\chi=\phi-\sum_{I\in \mathscr J_d}f^I\phi_I$. 
By construction, $\chi$'s Taylor expansion is 
a multiple of $f_n$, which implies that $f_n$ divides $\chi$ in $\nash KU$
(Lemma \ref{nash-divisibility} (2)). 

Let $\phi_n\in \nash KU$ such that $\chi=f_n\phi_n$. 
The Taylor development of $\phi_n$
at $u$ belongs to $(f_1,\ldots, f_n)_{d-1}K[\!f_1,\ldots, f_n]\!]$, 
so by induction we can 
write 
$\chi=\sum_{I\in \mathscr I^{d-1}}f^I\chi_I$ where the $\chi_I$ belong
to $\nash KU$ and where the Taylor expansions of the $f^I\chi_I$ 
have pairwise disjoint supports.
Then we have 
\[\phi=\sum_{I\in \mathscr J_d}f^I\phi_I+\sum_{I\in \mathscr I_{d-1}}f^{I+(0,\ldots, 0,1)}\chi_I, \]
and this writing fulfills the requirement of the Lemma. 
\end{proof}

\begin{rema}\label{etale-unavoidable}
If $n\geq 2$ and $\phi\in k[U]$ it seems unlikely 
that it be
possible in general to
achieve (1) with the $\phi_I\in K[U]$; at least 
the functions $\phi_I$ provided by our proof
definitely do not belong to $K[U]$ in general. 

Indeed, let us consider the following example. 
We assume that the residue characteristic of $K$ is not $2$, we take
for $X$ the quadratic cover
\[\spec K[T_1,T_2][S]/(S^2-1-T_1-T_2)\]
of $\A^2$, 
and we still use the notation $T_1$ and $T_2$ for their pull-back on $X$
(so they play the role of $f_1$ and $f_2$ above). 
Let $U$ be the $K$-definable subset $\{\abs {S-1}<1\}$ of $X$. 
Then $(T_1,T_2)$ induces a Nash isomorphism between $U$ and the open unit bi-disc;
the unique pre-image $u$ of the origin is characterized by the equalites
$T_1(u)=T_2(u)=0$ and $S(u)=1$. 

The function $S-1$ vanishes at $u$, so the Lemma above furnishes a decomposition
$S-1=T_1\phi_1+T_2\phi_2$ with $\phi_1$ and $\phi_2$ in $\nash KU$ (and where the Taylor
expansions of the $T_i\phi_i$ are disjoint). 
Looking at its proof, we see that the function $\phi_2$ it actually 
exhibits is such that the Taylor expansion of $T_2\phi_2$
at $u$ is equal to $\sqrt {1+T_2}-1$. Therefore $\phi_2$ does not
belong to $K[U]$: if one wants to realize it as a regular function, 
one has to go to the double cover of $X$ obtained by adjoining a square
root of $1+T_2$ (such a square root does not exist in $K(X)$, as
$(1+T_2)/(1+T_1+T_2)$ is not a square in $K(T_1,T_2)$). 

\end{rema}

If $s=(s_1,\ldots, s_n)$ is a polyradius $< r$ consisting
of elements of $\abs{M^\times}$ for some model $M$
of $\acvf$ containing $K$ we shall denote
by $\eta_s$ the point of $\widehat D(M)$ given 
by the valuation $\sum a_I T^I\mapsto \max \abs{a_i}\cdot s^I$
on $M[T_1,\ldots, T_n]$. 

\begin{theo}\label{theo-norme-gauss}
Let $s$ be as above. 
Let  $U_s$ be the pre-image on $U$ of the closed polydisc with polyradius
$s$, and let $u_s$ be the unique pre-image of
the Gauss point $\eta_s$ in $\widehat U$. 
Let  $\phi=\sum a_If^I$ be an element of $\nash KU$.

One has then
\[\|\phi\|_{U_s,\infty}=\abs{\phi(u_s)}=\max_I\, \abs{a_I}\cdot s^I,\]
and the maximum is achieved for only finitely many coefficients $a_I$. 
\end{theo}

\begin{proof}
The statement is insensitive 
to ground field extension, so
we can assume that $K$ is a model of
$\acvf$
and that each $s_i$ belongs to $\abs{K^\times}$; 
then $\eta_s$ and $u_s$ belong respectively
to $\widehat D(K)$ and $\widehat U(K)$. 
We argue by induction on $n$; there is nothing to prove if $n=0$.
The key point for the induction step will be the case $n=1$ which we handle now.

\subsubsection{The case $n=1$}\label{gaussformula-series-dim1}
We assume that $n=1$, and we write $T,f, r,s$ instead of $T_1, f_1, r_1,s_1$. 
If $a_i=0$ for every $i$ then $\phi$ is zero in the connected component of the smooth locus of $X$
that contains $U$; in particular $\phi|_U=0$ and the result is obvious. So we assume from now on that 
$\{i, a_i\neq 0\}$ is non-empty.

\paragraph{}\label{taylor-expansion-minoration}
Let us fix $i$ with $a_i\neq 0$ and first prove that
$\abs{\phi(u_s)}=\|\phi\|_{U_s,\infty}\geq \abs {a_i}\cdot s^i$. 
If there is some index $j<i$ such that
$ \abs {a_j}\cdot s^j\geq \abs{a_i}\cdot s^i$
it suffices to prove that $\abs{\phi(u_s)}=\|\phi\|_{U_s,\infty}\geq \abs {a_j}\cdot s^j$, so we can 
replace $i$ by $j$; hence we reduce to the case
where $\abs{a_j}\cdot s^j<\abs{a_i}s^i$ for all $j<i$. 
Then $\abs{\sum_{j<i}a_j f^j}$ is everywhere $<\abs{a_i}\cdot s^i$ 
on $\widehat U_s$, so it suffices to prove that $\|\phi-\sum_{j<i}a_jf^j\|_{U_s,\infty}$ 
is achieved at $u_s$ and that it is 
greater or equal than $\abs{a_i}\cdot s^i$. 
Otherwise said, we can assume that $a_if^i$ is the first term of $\phi$'s expansion. 
Then by Lemma \ref{nash-divisibility} (2) the quotient $\phi/f^i$
is well-defined in $\nash KU$.

By construction $U$ is contained in  the smooth locus 
of $X$, the function $f|_U$ vanishes only at $u$, and generates the maximal
ideal of the discrete valuation ring $\mathscr O_{X,u}$; therefore the rational function $(g/f^i)$ is defined on a $K$-Zariski
neighborhood of $U$, and can thus be seen as an element of $\nash KU$. 
As $\abs{f^i(u_s)}=s^i$ and as $s^i=\|f\|_{U_s,\infty}$, it suffices to show that $\|(\phi/f^i)\|_{U_s,\infty}=\abs{(\phi/f^i)(u_s)}\geq \abs{a_i}$.
Hence we can replace
$\phi$ by $\phi/f^i$; that is, we can assume that $i=0$. 
By refining $(X,U)$, we can assume that
$X$ is smooth and connected and that
$\phi\in K[X]$. 

If we prove that $\|\phi\|_{U_s,\infty}$ is achieved at $u_s$ we shall be done
because of the obvious
minoration $\|\phi\|_{U_s,\infty}\geq \abs{\phi(u)}=\abs{a_0}$. 
We argue by contradiction, so we assume that 
$\abs{\phi(u_s)}<\|\phi\|_{U_s,\infty}$. 
As $\widehat {U_s}$ is definably homeomorphic 
to the stable completion of the closed polydisc $D_s$ of radius $s$, it is compact, so $\abs \phi$ achieves its maximum $\|\phi\|_{U_s,\infty}$ 
on $\widehat {U_s}$. 
The set of points of $\widehat {U_s}$ at which $\abs \phi=\|\phi\|_{s,\infty}$ is closed. We are going to prove
that it is also open; as $\widehat U_s$ is connected (it is indeed homeomorphic to $\widehat {D_s}$), this will force
$\abs \phi$ to be constant of $\widehat {U_s}$, and imply in particular that $\abs {\phi(u_s)}=\abs {\phi(u)}=\abs{a_0}$, contradiction. 

So let $v$ be a point of $\widehat {U_s}$ such that $\abs {\phi(v)}=\|\phi\|_{U_s,\infty}$. 
By our assumption 
$v\neq u_s$. As $\eta_s$ is the topological boundary of $\widehat {D_s}$ in $\widehat D$, the point
$u_s$ is the topological boundary of $\widehat {U_s}$ in $\widehat U$, so $\widehat {U_s}\setminus \{u_s\}$ is open in 
$\widehat U$, hence in $\widehat X$. 
As a consequence $\abs \phi$ has a local maximum at $v$ in $\widehat X$
(which is non-zero, as $a_0=\phi(u)\neq 0$, so $\|\phi\|_{U_s,\infty}\neq 0$). 
If $v$ is a simple point, then $\abs \phi$ is constant around $v$ as $\phi(v)\neq 0$. If 
$v$ is not simple, then for every germ of interval $I$ emanating from $v$
on $\widehat X$, 
the function $\abs \phi$ is non-increasing along $I$.
Then the generalized rational function 
$\rv(\phi(v))$ on $C_v$ has no poles, so is constant, which means that it is equal to $\rv(\mu)$ for some $\mu\neq 0 $ in $K$. 
But $\phi$ is
then equal to
$\mu+h$ where $h$ is such that $\abs{h(v)}<\abs \mu$. As a consequence $\abs \phi=\abs \mu$ around $v$ and we are
done. 

\paragraph{}
The map $f$ induces a homeomorphism between $\widehat U$ and $\widehat D$. Moreover it follows from 
Lemma \ref{lemm-same-henselization} (applied to any realization of $v$) that $f$ induces an isomorphism between the residue
curve at $u_s$ and that at $\eta_s$, which is isomorphic to the projective line
over $\k(K)$.
A more natural way to see these curves is as generalized
varieties over
the generalized residue field $\RV(K)$, 
so that $C_{\eta_s}$ is the generalized projective line 
over $\RV(K)$ with coordinate $\rv(T(\eta_s))$; under this identification a germ of interval 
emanating from $\eta_s$ lies in $\widehat {D_s}$ if and only if the corresponding closed point of
$C_{\eta_s}$ is not the point at infinity.
It follows that $C_{u_s}$ is the generalized projective line
over $\RV(K)$ 
with coordinate $\rv(f(v_s))$, and that a germ of interval 
emanating from $v_s$ lies in $\widehat {U_s}$ if and only if the corresponding closed point of
$C_{u_s}$ is not the point at infinity.

We have seen that $\abs{\phi(v_s)}=\|\phi \|_{U_s,\infty}$. This implies that on every
germ of interval 
emanating from $u_s$ and drawn on $\widehat {U_s}$ the function $\abs \phi$ is less than  or equal to
$\abs {\phi(u_s)}$. As a consequence, $\rv(\phi(u_s))$ cannot have any pole on $C_{u_s}$ except at the point
at infinity; it follows that $\rv(\phi(u_s))$ is a polynomial in $f(u_s)$. This means that there exists a
finite set of integers $J$ and a finite non-empty
family $(b_j)_{j\in J}$ of elements of $K$ such that 
$\abs{b_j}s^j=\|\phi\|_{U_s,\infty}$ for all $j$ and $\rv(\phi(u_s))
=\sum \rv(b_j)\rv(f(u_s))^j$; one can rephrase the latter equality 
by saying that $\abs{\phi(u_s)-\sum_j b_jf(u_s)^j}<\|\phi\|_{U_s,\infty}$. 
The Taylor expansion at $u$ of
$\phi -\sum_j b_j f^j$ is $\sum_i c_i f^i$ with
$c_i=a_i$ for $i\notin J$ and $c_j=a_j-b_j$ for every $j\in J$. 
In view of \ref{taylor-expansion-minoration} the inequality
$\abs{\phi(u_s)-\sum_j b_jf(u_s)^j}<\|\phi\|_{U_s,\infty}$ implies that $\abs{c_i}s^i<\|\phi\|_{U_s,\infty}$
for all $i$. Otherwise said, $\abs{a_i}s^i<\|\phi\|_{U_s,\infty}$
for all $i\notin J$ and  $\abs{a_j}s^j=\|\phi\|_{U_s,\infty}$ for all $j\in J$, 
which ends the proof when $n=1$.

\subsubsection{}\label{fn-zero}
Now we handle the general case, 
assuming that the result holds in dimension 
$n-1$. We write formally $\phi =\sum_i A_if_n^i$, so that the coefficient
$A_i$ is (again, formally)
equal to $\sum_{J\in \Z^{n-1}}a_{(J,i)}(f')^J$ where
$f'=(f_1,\ldots, f_{n-1})$.  

We denote by $D_{s'}$
the closed polydisc of polyradius $s':=(s_1,\ldots, s_{n-1})$, 
by $X'$ the closed subscheme of $X$ defined by the equation $f_n=0$, and by
$U'$
be the intersection $X'\cap U$. Then $f'$ induces a  $K$-definable Nash isomorphism 
between $U'$ and $D'$, and by Lemma \ref{nash-divisibility} (1) there is for each $i$ some
$\phi_i$ in $\nash K{U'}$ whose Taylor expansion at $u$ is equal to $A_i$, and we can moreover assume, 
up to refining $(X,U)$, that $f$ is étale, $X$
is connected and  every $\phi_i$ belongs to
$K[X]$.

Let $v$ be any point of $U'$ and let $Z_v$ be the closed
$K$-subscheme of $X$ defined by the equations $f_j=f_j(v)$
for $1\leq j\leq n-1$. Then $f$ induces an étale map from $Z_v$ to
the subscheme $\{T_j=f_j(v)\}_{1\leq j\leq n-1}$ of $\A^n_K$, which we identify to
the affine line with coordinate function $T_n$. The map $f_n$ then induces
a definable bijection $Z_v\cap U\simeq \Delta$, where $\Delta$ is the one-dimensional
open disc with radius $r_n$. Let $w\in Z\times U$ be the unique
pre-image of the 
origin of $\Delta$ (otherwise said, the unique point of $Z\cap U$
at which $f_n$ vanishes, or the unique point of $Z\cap U'$); the expansion of $\phi|_Z$ in $\widehat {\mathscr O_{Z,w}}\simeq
K[\![f_n]\!]$ is then
by construction equal to $\sum_i \phi_i(v)f_n^i$. 

\paragraph{}
Let $M$ be a model of $\acvf$ containing $K$ and let $v\in U'(M)$. 
Let us denote by $\rho(v)$ the maximum of $\abs \phi$ on $Z_v$, 
and let $t_v$ be any point of $Z_v$ 
such that 
the type of $f_n(t_v)$ over $M$ is equal to $\eta_{s_n}$. 
By the above and by \ref{gaussformula-series-dim1}
one has
\[\rho(v)=\abs{\phi(t_v)}=\max_i \abs {\phi_i(v)}s_n^i,\]
and this maximum
is achieved for only finitely many indices $i$. 

Now let $\omega$ be a point of $U'$ whose
image under $(f_1,\ldots, f_{n-1})$ induces the type
$\eta_{s_1,\ldots, s_{n-1}}$ on $M$. The following facts follow from the induction
hypothesis: 
\begin{itemize}[label=$\diamond$]
\item for all $i$ the maximum $\theta_i$ of the function $\abs{\phi_i}$ on 
$U'$ is achieved at $\omega$; 

\item for all $i$ we have $\theta_i=\max_{J} \abs{a_{J,i}}(s')^J$ 
where $s'=(s_1,\ldots, s_{n-1})$, and this maximum is achieved for
only finitely many multi-indices $J$. 
\end{itemize}

As a consequence, $\|\phi\|_{U_s,\infty}$  is achieved at
$t_\omega$, it is equal to $\max_I \abs{a_I}\cdot s^I$, and this
maximum is achieved for only finitely many multi-indices $I$. 
Since the type of $f(t_\omega)$ over $M$ is equal to $\eta_s$ by construction, 
this ends the proof.
\end{proof}

\begin{coro}\label{gauss-norm-open-disc}
Let $\phi=\sum a_If^I$ be an element of $\nash KU$. 
Then $\max_I \,\abs{a_i}\cdot r^I$ exists and
$\|\phi\|_{U,\infty}=\max_I\, \abs {a_I}\cdot r^I$.

\end{coro}

\begin{enonce}[remark]{Comments}
One cannot in general expect the maximum
of the $\abs{a_I}\cdot r^I$ 
to be achieved for only finitely many indices. For instance, on the one-dimensional open
unit disc, $(1-T)^{-1}$ has Taylor expansion $\sum_i T^i$. 
\end{enonce}

\begin{proof} [Proof of Corollary \ref{gauss-norm-open-disc}]

Let $M$ be a model of $\acvf$ containing $K$ and 
some $\lambda$
with $\abs{\lambda}<1$ and $(\lambda, 1)\cap \abs{K^\times}^\Q=\emptyset$. 

Set $s_i=r_i\abs \lambda$ for all $i$. 
It follows from Theorem \ref{theo-norme-gauss}
that $\max_I\, \abs{a_I}\cdot s^I$ exists (and is achieved for only finitely many indices, 
but we shall not need it here). Let $J$ such that $\abs{a_J}\cdot s^J$
realizes this maximum. Then for all $I$ one has $\abs{a_I}\cdot s^I\leq \abs{a_J}
\cdot s^J$, and since
non-strict inequalities are preserved when one mods out by a convex subgroup, this 
implies that $\abs{a_I}\cdot r^I\leq \abs{a_J}\cdot r^J$
for all $I$, and we are done.

It follows then again from Theorem \ref{theo-norme-gauss}, by letting $s$ tend to
$r$ in any model of $\acvf$
containing $K$,  
that $\|\phi\|_{U,\infty}=\max_I \abs {a_I}\cdot r^I$. 
\end{proof}

\begin{coro}\label{invertible-constant}
Let $\phi$ be an invertible element in $\nash KU$. Then 
$\abs \phi$ is constant on $U$.
\end{coro}

\begin{proof}
For all $s<r$ it follows from Theorem
\ref{theo-norme-gauss}
that 
both $\abs{\phi}$ and $\abs{\phi^{-1}}$ achieve
their maximum on $\widehat {U_s}$ at the same point, namely $u_s$. 
This implies that $\abs \phi$ is constant on $U_s$ for all $s$, so 
it is constant on the whole of $U$. 
\end{proof}

\section{Tame descent of abstract open
polydiscs}
\label{tame-descent}

\begin{theo}\label{reduction-abstract-disc}
Let $X$ be an algebraic variety over a henselian valued
field $K$, and let $U$ be a $K$-definable subset of  $X$.  
Let $f=(f_1,\ldots,f_n)$ be an algebraic map from
$X$ to $\A^n_K$ inducing a
Nash isomorphism
between $U$ and an open polydisc $D$.  

\begin{enumerate}[1]
\item The elements $\rv_U(f_1),\ldots, \rv_U(f_n)$ 
of $\redgrad KU$
are 
algebraically independent over $\RV(K)$.

\item Let $A$
be the completion of $\RV(K)[\rv_U(f_1),\ldots, \rv_U(f_n)]$
along its  ideal generated by the $\rv(f_i)$,
and let $\mathfrak n$ be
its maximal ideal. The generalized ring $\rednash KU$ 
embeds into  $A$. 
It is local, its
maximal ideal $\mathfrak m$ is generated by 
$(\rv_U(f_1),\ldots, \rv_U(f_n))$, 
a system of generators of $\mathfrak m$ is minimal if and only 
if it has cardinality $n$, and $\mathfrak m^d=\mathfrak n^d\cap \rednash KU$ for all $d$. 

\item The generalized ring $\redgrad KU$
is local as well, and the inclusion
of $\redgrad KU$
in $\rednash KU$ is local. 
If $n=1$
the maximal ideal $\mathfrak p$ of
$\redgrad KU$
is generated by $\rv_U(f_1)$, a system of generators of
$\mathfrak p$ is minimal if and only 
if it has cardinality $1$, and $\mathfrak p^d=\mathfrak n^d\cap \redgrad KU$
for all $d$. 

\item Let $g=(g_j)_{1\leq j\leq n}$ be a family of 
non-zero
elements of $\nash KU$, with cardinality $n$. 
The following are equivalent: 
\begin{enumerate}[j]
\item the family $(g_j)$
induces a Nash isomorphism between $U$ and an open 
polydisc containing the origin; 
\item the family $(\rv_U(g_j))$ is a system of generators of $\mathfrak m$;
\item the family $(\rv_U(g_j))$ is a system of generators of $\mathfrak n$.
\end{enumerate}
\item Assume that $n=1$
and let $g$ be an element of
$K[U]$. The following are equivalent: 

\begin{enumerate}[j]
\item $g$
induces a Nash isomorphism between $U$ and an open 
disc containing the origin; 
\item $\rv_U(g)$ generates $\mathfrak p$;
\item $\rv_U(g)$ generates $\mathfrak n$. 
\end{enumerate}

\end{enumerate}

\end{theo}

\begin{proof}
Let $(r_1,\ldots, r_n)$ be the polyradius of $D$; otherwise said, 
$r_i=\|f_i\|_{U,\infty}$ ; let $u$ be the unique pre-image of the origin under
$f$ on $D$.

\subsubsection{Proof of (1)}
For every polynomial $\sum a_I T^I$ in $K[T]$ one has
\[\left \|\sum a_I f^I\right \|_{U,\infty}=\left \|\sum a_I T^I\right \|_{D,\infty}=
\max \abs{a_I}\cdot r^I\]
(as one deals with polynomials here, one does not have to use
Corollary \ref{gauss-norm-open-disc}); this shows that the functions $\rv_U(f_i)$ are
algebraically independent over $\RV(K)$. 

\subsubsection{Proof of (2) and (3)}
Let $\phi=\sum a_I f^I$ be a non-zero 
element of $\nash KU$
at $u$.
By Corollary  \ref{gauss-norm-open-disc}, 
one has
$\left \|\phi\right\|_{U,\infty}=\max_I \abs{a_I}\cdot r^I$; let us 
denote by $N(\phi)$ the set of those $I$ such that $\abs{a_I}\cdot r^I$. 
If $\psi=\sum b_I f^I$ is another non-zero element of 
$\nash KU$ it is immediate that 
$\rv_U(\phi)=\rv_U(\psi)$ if and only if 
\[\sum_{I\in N(\phi)}\rv(a_I) \rv_U(f)^I=\sum_{I\in N(\psi)}\rv(b_I) \rv_U(f)^I,\]
so $\phi=\sum_I a_I f^I\mapsto \sum_{I\in N(\phi)}\rv(a_I) \rv_U(f)^I$ induces an embedding
from $\rednash KU$ into $A$ which we shall use from now on to see $\rednash KU$ 
as a subring of $A$; with this identification one has
$\rv_U(\phi)=\sum_{I\in N(\phi)}\rv(a_I) \rv_U(f)^I$.

\paragraph{}
Let us prove that an element of 
the generalized ring $\rednash KU$, \resp $\redgrad KU$, is invertible if and only
if it is invertible in $A$, \ie,
if and only if its constant coefficient is non-zero. 
The ``only if" part is obvious. Let us now consider a non-zero $\phi=\sum a_I f^I$
in $\nash KU$, \resp $K[U]$, such that the constant term of $\rv_U(\phi)$ is non-zero, 
which exactly means that $0\in N(\phi)$, that is, that $\abs{a_0}=\|\phi\|_{U,\infty}$. 
One then has $\abs{a_I}r^I\leq \abs{a_0}$ for all $I\neq 0$, and thus  $\abs{a_I}s^I< \abs{a_0}$
for all $I\neq 0$ and every $s=(s_1,\ldots, s_n)$ with $\abs{s_i}<r_i$ for all 
$i$. 
Set $\psi=\phi-a_0$; then $\psi=\sum_{I\neq 0} a_I f^I$, so by
the above and Theorem \ref{theo-norme-gauss} one has
for every $s=(s_1,\ldots, s_n)$ with $\abs{s_i}<r_i$ for all 
$i$ the majoration $\|\psi\|_{V_s,\infty}<\abs {a_0}$, where $V_s=\{\abs f \leq s\}$. 
It follows that $\phi=a_0+\psi$ does not vanish on any $V_s$ as above, so $\phi$ does not
vanish on $U$. Then
$\phi^{-1}$ is an element of $\nash KU$, \resp $K[U]$: indeed
in both cases we can refine $(X,U)$ so that $f$ is
étale, $X$ is connected and $\phi\in K[X]$.
Then $U\subset D(\phi)$ 
so $\phi^{-1}\in K[U]\subset  \nash KU$. 
Then $\rv_U(\phi^{-1})$ is an inverse
of $\rv_U(\phi)$ in $\redgrad KU$ and in $\rednash KU$.

It follows from the above that $\rednash KU$ and $\redgrad KU$ are local, 
and that in both cases, the maximal ideal is that of series with zero
constant term; one has therefore 
$\mathfrak m=\mathfrak n\cap \rednash KU$, and
the inclusion map $\redgrad KU\hookrightarrow 
\rednash KU$ is local.

\paragraph{}\label{nd-md}
Let us fix an integer $d$, and denote
by $\mathscr I_d$ the set of multi-indices $I$
with 
$\abs I=d$. It is obvious that
$\mathfrak m^d$ contains each of the $\rv_U(f)^I$ for
$I\in \mathscr I_d$, and that $\mathfrak m^d$ is contained
in $\mathfrak n^d\cap \nash KU$. 
We shall prove that $\mathfrak n^d\cap \rednash KU$ 
is contained in the ideal 
$(\rv_U(f)^I))_{I\in \mathscr I_d}$
of $\rednash KU$,  which will prove
that $\mathfrak m^d=\mathfrak n^d\cap \nash KU$ 
and that it is generated by the $\rv_U(f)^I$;  this will imply
(by taking $d=1$) that $\mathfrak m$ is  generated by the $\rv_U(f_i)$. 

Let $\phi=\sum a_I f^I$ be a non-zero element of $\nash KU$ whose reduction 
$\rv_U(\phi)$ belongs to $\mathfrak n^d\cap \nash KU$. This means that 
$N(\phi)$ does nt contain any multi-index $I$ with $\abs I<d$. But one has then
$\rv_U(\phi)=\rv_U(\phi-\sum_{\abs I<d}a_If^I)$, which allows to assume that $a_I=0$
for all $I$ with $\abs I<d$. 

Then by Lemma \ref{lemm-generation-ideal} one can write 
$\phi=\sum_{I\in \mathscr I_d} f^I\phi_I$ where each $\phi_I$ belongs to $\nash KU$ and where the Taylor expansions
of the $f^I\phi_I$ have pairwise disjoint support. 
But one then has in view of Corollary  \ref{gauss-norm-open-disc}
$\|\phi|\|_{U,\infty}=\max_{I\in \mathscr I_d} \|f^I\phi_I\|_{U,\infty}$ and 
\[\rv_U(\phi)=\sum_{I\in \mathscr I'}\rv_U(f)^I\rv_U(\phi_I),\]
where $\mathscr I'$ is the set of $I\in \mathscr I_d$ such that $\|f^I\phi_I\|_{U,\infty}
=\|\phi\|_{U,\infty}$; this the kind of writing we were looking for.

\paragraph{}As $(\rv_U(f_i))_i$ is a minimal generating system of the maximal ideal of $A$, it is 
a fortiori minimal as a generating system of $\mathfrak m$. And by the generalized Nakayama Lemma, 
all minimal generating systems of $\mathfrak m$ have the same cardinality, namely
the dimension of $\mathfrak m/\mathfrak m^2$ over $\RV(K)$; hence they all have cardinality $n$.

\paragraph{}
Let us assume in this paragraph that $n=1$, write $f$ instead of $f_1$, and let 
$\mathfrak p$ be the maximal ideal of $\redgrad KU$. It is clear that $\rv_U(f)$ belongs to $\mathfrak p$. We shall prove that $\mathfrak n^d\cap \redgrad KU$
is contained in the ideal 
$(\rv_U(f))$
of $\redgrad KU$, which will end the proof of (4) (reasoning like in 
\ref{nd-md}). 

Let $\phi=\sum a_i f^i$ be a non-zero element of $K[U]$ whose reduction 
$\rv_U(\phi)$ belongs to  $\mathfrak n^d\cap \redgrad KU$. 
This means that $N(\phi)$
does not contained any integer $<d$ or, 
otherwise said, that $\abs {a_i}<\|\phi\|_{U,\infty}$ for all $i<d$. But one has then
$\rv_U(\phi)=\rv_U(\phi-\sum_{i<d}a_if^i)$, which allows to assume that $a_i=0$
for $i<d$. 

Then by Lemma \ref{lemm-generation-ideal} one can write 
$\phi=f^d\psi$ for some $\psi\in K[U]$. Thus $\rv_U(\phi)$
is equal to $\rv_U(f)^d\rv_U(\psi)$, so $\rv_U(\phi)$ belongs to the ideal 
$(\rv_U(f)^d)$.

\subsubsection{Proof of (4)}
It follows from (2) that (i)$\Rightarrow$(ii)
and that (ii)$\iff$(iii). So let us now assume that (iii)
holds and prove (i). 
For every index $j$
write $g_j=\sum_I a_{I,j}f^I$; 
one has then
$\rv_U(g_j)=\sum_{I}\alpha_{I,j}\rv_U(f)^I$
where 
$\alpha_{I,j}=0$ if $I\notin N(g_j)$ and
$\alpha_{I,j}=\rv(a_{I,j})$ otherwise. 
Set $\gamma_j=\|g_j\|_{U,\infty}$. 
Asssumption (iii) amounts to saying that 
$\alpha_{0, j}=0$ for all $j$ and
(by Nakayama) that the matrix $(\alpha_{i,j})_{i,j}$ is invertible. 
Let $M$, \resp $N$, denote
the $K$-vector space $K^n$ equipped
with the norm
$(\lambda_1,\ldots, \lambda_n)\mapsto \max \abs{\lambda_i}\cdot 
\gamma_i$, \resp $\max \abs{\lambda_i}r_i$. 

By construction $\abs{a_{ij}}\leq \gamma_jr_i^{-1}$ for all $i$; hence
the matrix $(a_{ij})$ 
induces a linear
map from $M$ to $N$
that does not increase norms. 
This map in turn induces a $\RV(K)$-linar map 
between the graded reductions $\RV(M)$ and $\RV(N)$, 
precisely given by the matrix $(\alpha_{ij})$ in the canonical 
bases of both $\RV(K)$-vector spaces. 
The fact that $(\alpha_{ij})$ is bijective can then be rephrased by saying
that $(a_{ij})$ induces an isometry $\iota$ between $M$ and $N$. 
The isometries $\iota$ and $\iota^{-1}$, viewed as $K$-endomorphisms 
of $\A^n_K$, map open balls to open balls; hence for proving (i) 
one can replace $g$ by $\iota^{-1}\circ g$, which 
amounts to assume 
that $(\alpha_{ij})$ is the identity matrix
(in fact we have achieved more, since $(a_{ij})$
itself is the identity matrix; but we shall need to modify the $f_j$
later in the proof, and $(a_{ij})$ might be modified through these changes,
while $(\alpha_{ij})$ will remain unchanged). In other words, one has 
$\rv_U(g_j)=\rv_U(f_j)$ modulo $\mathfrak n^2$
for all $j$ (this implies that $\gamma_j=r_j$ for all $j$).

\paragraph{}\label{first-computations}
Our first aim is now to prove 
that $g$ induces a Nash
isomorphism between $U$
and $D$. The ground field does not matter for this purpose, so we can assume that $K$
is a model of $\acvf$. 
The assumption that $\rv(g_j)=\rv(f_j)$
modulo $\mathfrak m^2$ can be rephrased by saying that the Taylor expansion 
$\sum a_{I,j}f^I$ of $g_j$ at $u$ fulfills the following conditions, if we denote by $\delta_i$
the $n$-uple whose $i$-th component is $1$ and all of whose other components are zero. 

\begin{itemize}[label=$\diamond$] 
\item $\abs{a_{I,j}}\cdot r^I\leq r_j$ for all $I$;
\item $a_{\delta_j,j}$ if of the form $1+\epsilon_j$ with $\abs{\epsilon_j}<1$; 
\item $\abs {a_{0, j}}<r_j$ and $\abs{a_{\delta_i,j}}\cdot
r_i<r_j$ for all $i\neq j$.
\end{itemize}

It follows from the above
that for every $s<r$ and every $j$ one has $\max_I \abs{a_{I,j}}s^I<r_j$, 
which implies in view of Theorem \ref{theo-norme-gauss}
that $\abs{g_j}$ is strictly bounded by $r_j$ on 
$\{\abs f\leq s\}$; as this holds for all $s<r$, 
we see that the supremum norm $\|g_j\|_{U,\infty}=r_j$ is not achieved on $U$. 
As a consequence, $g(U)\subset D$.

Let $m$ be the number of indices $j$ such that $g_j=f_j$. We argue by descending induction on $m$. 
There is nothing to prove if $m=n$, so we assume $m<n$ and that the result holds for $m+1$. 
Since $m<n$ there exists $j$ such that $g_j\neq f_j$. Up to renumbering the $g_j$ and the $f_j$
we can assume that $g_1\neq f_1$. 
We are then going to prove that $(g_1,f_2,\ldots, f_n)$ induces a
Nash isomorphism between $U$
and $D$. Setting $h_1=g_1$ and $h_j=f_j$ if $j>1$ 
we will have $\rv(g_j)=\rv(h_j)$ for all $j$, and there will be $m+1$ indices $j$ with 
$g_j=h_j$; so the induction hypothesis (applied with $h$
instead of $f$) will then ensure that $(g_1,g_2,\ldots, g_n)$
induces a  $K$-definable Nash  isomorphism between $U$ and $D$. 

Let $Y$ be the zero-locus of $(f_2,\ldots, f_n)$ and set $V=Y\cap U$. 
Then $f_1|_Y\colon Y\to \A^1$ induces a  $K$-definable Nash isomorphism 
between $V$ and the open one-dimensional disc $\Delta$ of radius $r_1$. 
The point $u$ belongs to $V$, and the Taylor expansion in $f_1$
of $g_1|_Y$
at $u$
is the pure $f_1$-part of that of $g_1$ at $u$; \ie, this
is $\sum_i b_if_1^i$ where $b_i=a_{(i,0,\ldots, 0), 1}$. 
One has then $\abs {b_0}<r_1$, $b_1=1+\epsilon_1$ (with $\abs{\epsilon_1}<1$), 
and $\abs{b_i}r_1^i\leq r_1$ for all $i$. 
This implies that 
$\rv_V(g_1)=\rv_V(f_1)$ (in particular $\|g_1\|_{V,\infty}=r_1$).

\paragraph{Let us prove that $g_1$ has a unique zero on $V$}
By Corollary \ref{gauss-norm-open-disc} one 
has $\|g_1\|_{V,\infty}=\max \abs{b_i}r_1^i$, and the latter is equal to 
$r_1$ by the above. As $\abs{g_1(u)}=\abs {b_0}<r_1$ we see that 
$\abs{g_1}$ is not constant on $V$, so it is not invertible
on $V$ by Corollary \ref{invertible-constant}; hence 
$g_1$ vanishes at some point of $V$. 
Let us prove that this point is unique. Let $v$ be a zero of
$g_1$ and set $h=f_1-f_1(v)$. Then $h$ induces a 
Nash isomorphism between $V$ and $\Delta$, and since
$\abs{f_1(v)}<r_1$ one has $\rv_V(h)=\rv_V(f_1)$, so 
$\rv_V(g_1)=\rv_V(h)$. This means that the Taylor expansion 
of $g_1$ at $v$ in $h$ is of the form 
$\sum c_i h^i$ with $c_0=0$ (since $g_1$ vanishes at $v$), 
with $c_1=1+\epsilon'$ with $\abs{\epsilon '}<1$, and with 
$\abs{c_i}\cdot r_1^i\leq r_1$ for all $i$. 
Let $s\in [0,1)$ and let 
$E(v,s)$ be the pre-image under $h=f_1-v$ of the closed disc
of radius $s$. Then for all
$i\geq 2$ one has
\[\abs {c_i}\cdot s^i=(s/r_1)^i\abs{c_i}\cdot r_1^i\leq (s/r_1)^ir_1<s,\]
so by Theorem 
\ref{theo-norme-gauss} one has
\[\|g_1-h\|_{E(v,s),\infty}=\|g_1-f_1+f_1(v)\|_{E(V,s),\infty}=
\left\|\sum_{i\geq 2}c_i f_1^i\right\|_{E(v,s),\infty}<s.\]
Now let $w$ be another zero of $g_1$ and set $s=\abs{f_1(v)-f_1(w)}$. 
Then $E(v,s)=E(w,s)$ and 
by the above one has
\[\|g_1-f_1+f_1(v)\|_{E(V,s),\infty}<s\;\text{and}\;\|g_1-f_1+f_1(w)\|<s,\]
whence the inequality $\abs{f_1(v)-f_1(w)}<s$, contradiction. 

\paragraph{The map $(g_1,f_2,\ldots, f_n)$ induces a bijection $U\simeq D$}
We have already seen  (\ref{first-computations})
that 
$g(U)\subset D$. 
By the paragraph above, the origin has a unique pre-image 
in $U$ under $(g_1,f_2,\ldots, f_n)$.

Now let $(d_1,\ldots, d_n)$ be any point of $D(K)$. 
Then $(f_j-d_j)_j$ induces a
Nash isomorphism
between $U$ and $D$, and $\rv(g_1-d_1)=\rv(g_1)=\rv(f_1)\rv(f_1-d_1)$. 
So by the above applied to $(f_j-d_j)_j$ instead of $(f_j)_j$ we
see that there is a unique $K$-point of $U$ at which 
$(g_1-d_1,f_2-d_2,\ldots, f_n-d_n)$ vanishes; this proves that $(d_1,\ldots, d_n)$
has a unique pre-image on $U$ under $(g_1,f_2,\ldots, f_n)$.

\paragraph{The map $(g_1,f_2,\ldots, f_n)$ induces a
homeomorphism $\widehat U\simeq \widehat D$}
We already know
that
it  induces a continuous bijection 
$\widehat U\simeq \widehat D$; it remains to show that the latter is topologically
proper. 
So fix $s<r$ and let us prove that the pre-image
$V$ under $(g_1,f_2,\ldots, f_n)$
of the the closed polydisc of polyradius $s$ has a compact stable completion. 
It suffices to prove that $\widehat V$
is contained
in some compact subset of $\widehat U$
or, equivalently, that it is contained in $\{\abs f \leq t\}$ for some $t<r$; and this
can be checked at the level of the definable sets themselves, without refereeing to 
their stable completions. 

So let $v\in V$. Our purpose it to bound for every $j$ the element 
$\abs{f_j(v)}$ of $\Gamma$ by some quantity $<r_j$ independent of $v$. 
As $\abs{f_j(v)}\leq s_j$ as soon as $j\geq 2$, it remains to bound
$\abs{f_1(v)}$.

One has $g_1=\sum_I a_{I,1}f^I$. We can split formally this sum and write
\[g_1=\sum_{I\in \mathscr J}a_{I,1}f^I+\sum_{I\notin \mathscr J}a_{I,1}f^I,\]
where $\mathscr J$ denotes the set of indices $I$ whose first component is non-zero. 
It now follows from Lemma \ref{nash-divisibility} and Proposition \ref{restriction-nash-functions}
that this formal splitting actually underlies a genuine decomposition 
$g_1=f_1\phi+\psi$ where $\phi$ and $\psi$ belong to $\nash KU$ and have 
\[\sum_{I\in \mathscr J} a_{I,1}f^{I-\delta_1}\;\text{and}\;\sum_{I\notin \mathscr J} a_{I,1} f^I\]
as respective Taylor expansions. 

We then have $f_1\phi=g_1-\psi$. As the Taylor expansion of $\psi$ does not involve
$f_1$, it follows from Theorem \ref{theo-norme-gauss} that 
\[\|\psi\|_{V,\infty}=\max_{I\notin \mathscr J}\abs{a_I}\cdot s^I<r_1,\]
where the latter inequality comes from the fact that $r_1=\|g_1\|_{U,\infty}=\max_I \abs{a_{I,1}}f^I$, 
again by Theorem  \ref{theo-norme-gauss}. As $\|g_1\|_{V,\infty}=s_1<r_1$, it now suffices to show
that $\abs \phi$ is identically equal to $1$ on $V$. 

The Taylor expansion of 
the function $\phi$ is equal to $a_{\delta_1,1}+\sum_{I\in \mathscr J, \abs I\geq 2}a_{I,1}f^{I-\delta_1}$.
One has $a_{\delta_1,1}=1+\epsilon_1$ with $\abs \epsilon_1<1$, so $\abs {a_{\delta_1,1}}=1$. 
Set $\chi=\phi-a_{\delta_1,1}$. 
For all $I\in \mathscr J$ with $\abs I\geq 2$ one has $\abs{a_{I,1}}\cdot r^I\leq r_1$, so 
$\abs{a_{I,1}}\cdot r^{I-\delta_1}\leq 1$. Let $t_1$ be any element 
of $\Gamma$ with $t<r_1$, and set $t_j=s_j$ if $j\geq 2$. 
Then as $\abs I\geq 2$ one has $\abs{a_{I,1}}\cdot t^{I-\delta_1}<1$. 
By Theorem \ref{theo-norme-gauss} it follows that 
the maximum of $\chi$ on $\{\abs f\leq t\}$ is $<1$. Since $t_1$
was chosen arbitrary, $\abs \chi<1$ at every point of $\{\abs {f_2}\leq s_2,\ldots, \abs{f_n}\leq s_n\}$, 
and in particular at every point of $V$; 
as a consequence, $\abs \phi=\abs{\chi+a_{\delta_1,1}}$ is identically equal to $1$ on $V$, which we wanted
to show. 

\paragraph{}
We have shown that $g$ induces a $K$-definable bijection $U\simeq D$ and 
even a $K$-definable homeomorphism $\widehat U\simeq \widehat D$. 
It remains to ensure that $g$ is étale at every point of $U$. 
So let $v=(v_1,\ldots, v_n)$ be a point of $U$. 
For all $j$ one has $\abs{g_j(v)}<r_j$ so that we have $\rv_U(g_j-g_j(v))=\rv_U(g_j)=\rv_U(f_j)$; 
as the translation by $(g(v_1),\ldots, g(v_n))$ is an automorphism of $\A^n$, it suffices to show
that $(g-g(v_1),\ldots, g-g(v_n))$ is étale at $v$. So we can assume that $g(v)=0$. 

Now $(f-f(v_1), \ldots, f-f(v_n))$ also induces a
Nash isomorphism 
between $U$ and $D$, and $\rv_U(f_j-f_j(v))=\rv_U(f_j)$ for all $j$. 
So we can replace $f$ by $f-f(v)$, which amounts to assuming that $f(v)=0$, that is, 
that $v=u$. 

As $g(u)=0$ the Taylor expansion of every $g_j$ at $u$ has no constant term. 
And as $\abs{a_{\delta_j, j}}=1$ and $\abs{a_{\delta_j,i}}<1$ for all $j$ and all $i\neq j$, 
$\abs{\det((a_{\delta_j,i})_{i,j})}=1$, and in particular the matrix 
$(a_{\delta_j,i})_{i,j})$ is invertible. By Nakayama's Lemma this implies that $(g_j)$ is a regular system of
parameters of $X$ at $u$, so $g$ is étale at $u$.

\subsubsection{Proof of (5)}
The equivalence between (ii) and (iii) comes from (3), and the equivalence between (i) and (iii) 
comes from (4).
\end{proof}

\begin{theo}[tame descent for
abstract open polydiscs]\label{theo-tame-descent}
Let $X$ be a $K$-algebraic variety and let $U$ be a $K$-definable 
subset of $X$. Assume that there exist a tamely ramified finite extension $L$
of $K$ and a tuple $f = (f_1, \ldots, f_n)$ of functions in $\nash LU$ inducing a
Nash isomorphism
between $U$ and an open polydisc. Then there exists a tuple $g = (g_1, \ldots, g_n)$ of functions in $\nash KU$ inducing a  
Nash
isomorphism between $U$ and an open polydisc.
Moreover, if $n=1$
and $f = f_1\in L[U]$ then $g=g_1$ can be chosen in $K[U]$. 
\end{theo}

\begin{proof}
As $L$ is a tamely ramified
extension of
$K$, it
follows from Corollary \ref{coro-tame-reduction}
that
\[\rednash LU=\rednash KU\otimes_{\RV(K)}\RV(L)\;\;\;
\text{and}\;\;\;\redgrad LU=\redgrad KU\otimes_{\RV(K)}\RV(L).\]

By Theorem \ref{reduction-abstract-disc}
(2), 
the generalized ring $\rednash LU$ is local, hence $\rednash KU$
is local as well by the generalized going-up lemma.
Let $(\alpha_i)$ be a family of generators of the maximal 
ideal of $\rednash KU$. The
$\rednash KU/(\alpha_i)$-algebra
$\rednash LU/(\alpha_i)$ 
is (again by the going-up lemma)
local artinian, and also
étale since $\RV(L)$ is separable over $\RV(K)$ by one of the possible definitions of 
a tamely ramified extension. 
Hence $\rednash LU/(\alpha_i)$ is a field, so $(\alpha_i)$ generates
the maximal ideal of $\rednash LU$. By Theorem \ref{reduction-abstract-disc} 
(2) there is a sub-family of
$(\alpha_i)$ of cardinality $n$ that already generates the maximal ideal of $\rednash LU$.
If $g$ denotes any lifting of this subfamily in $\nash KU$, Theorem \ref{reduction-abstract-disc}
(4)
ensures that $g$ induces a  
Nash isomorphism between 
$U$ and an open polydisc. This ends the proof of the first statement.

Assume now that $n=1$
and that $f\in L[U]$. 
Reasoning as above and using assertion (3) (instead of (2)) 
of Theorem \ref{reduction-abstract-disc}, we get the existence of 
$g\in K[U]$ such that $\rv(g)$ generates the maximal ideal of 
$\redgrad KU$, and we
conclude by using Theorem \ref{reduction-abstract-disc}
(5).
\end{proof}

\section{Abhyankar points lie on atomic abstract open polydiscs}\label{atomic-polydiscs}

\begin{lemm}\label{etale-isomorphism}
Let $K$ be a valued field and let $f\colon Y\to X$ be a morphism between
$K$-varieties. Let $x\in X$ and let $y$ be a pre-image of $x$. Assume that $f$
is quasi-finite at $y$, that the valued field $K(x)$ is
defectless and that $K(y)$ is an immediate extension of $K(x)$. Then there exists a $K$-subvariety  $X'$ of $X$ containing $x$, 
a $K$-subvariety $Y'$ of $Y\times_X X'$ containing $y$ which is
finite
and étale
over $X'$, a $K$-regular function $\psi$ on $Y'$, 
a $K$-regular function $\phi$ on 
$X'$ with $\abs{\psi(y)}=\abs{\phi(x)}<1$
and a $K$-definable subset $D$ of $X'$ containing $x$ such that
the following hold for every $z\in D$: 

 \begin{itemize}[label=$\diamond$]
\item 
$\abs{\phi(z)}<1$ and
there exists one and only one element $\sigma(z)$ of $Y'$ such that $f(\sigma(z))=z$
and $\abs{\psi(\sigma(z))}=\abs{\phi(z)}$;

\item $\abs {\psi(t)}=1$ for every $t\neq \sigma(z)$ 
in the fiber $f^{-1}(z)$; 
\item 
$K(\sigma(z))\h=K(z)\h$. 
\end{itemize}
\end{lemm}

\begin{proof}
Up to replacing $X$ by the reduced $K$-Zariski closure of $x$
and $Y$ by the reduced $K$-Zariski-closure of $y$
we can assume that $X$ and $Y$ are integral and
$x$ and $y$ are $K$-Zariski dense. 
And then up to replacing $Y$ by the quasi-finite locus of $f$ we can assume that
$f$ is quasi-finite. 

The field $K(y)$ is a finite extension of $K(x)$. Thus $K(y)\h$ is a finite
extension of $K(x)\h$. By assumption this finite extension is immediate and since $K(x)$ is defectless, so is $K(x)\h$; 
we therefore have $K(y)\h=K(x)\h$. As a by-product $K(y)\subset K(x)\h$, so $K(y)$
is separable over $K(x)$. As $Y$ is reduced and $x$ is $K$-Zariski-generic, 
this implies that $f$ is étale at $y$; so by replacing $X$ and $Y$ by suitable
dense $K$-Zariski open subsets we can assume that $f$ is finite étale. 

By the above, $y$ appears as a $K(x)\h$-rational point of the $K(x)$-variety $f^{-1}(x)$.  
As $K(x)\h$ is contained in the definable closure of $K(x)$, the point $y$
is definable over $K(x)$, so it is equal to $\sigma(x)$ for some suitable $K$-definable
map. Then every point $t$ of $f^{-1}(x)$ having the same
type than $y$ over $K$ satisfies the equation $\sigma((f(t))=t$, so that $t=y$. 
Therefore $y$ is, as a point of the fiber $f^{-1}(x)$, entirely determined by its type over $K$; otherwise
said, it is entirely determined by the induced valuation $v$ on $K(Y)$. As $Y\to X$ is finite, the set
$\mathscr V$ 
of valuations on $K(Y)$ extending the valuation $u$ on $K(X)$ induced by $x$
is finite, and is a system of independent
valuations. 
Therefore there exists $\psi \in K(Y)$ such that $v(\psi)<1$ and $w(\psi)=1$ for every 
valuation $w\in \mathscr V\setminus \{v\}$. 
Up to replacing $\psi$ by $\psi^N$ for
$N$ large enough we can assume that $v(\psi)\in u(K(X))^\times$,
so  we can find a rational function $\phi$ on $X$ with 
$u(\phi))<1$ and $v(\psi)=u(\phi)$.
Then $\abs {\phi(x)}<1$, the point $y$ is the only point of $f^{-1}(x)$ at which 
$\abs \psi=\abs \phi$, and
$\abs{\psi(t)}\geq 1$ for all $t\neq y$
in $f^{-1}(x)$. Let us shrink $X$ and $Y$ so that $\phi$ is a regular function on $X$ and $\psi$ is a regular function 
on $Y$. 
Then we can take for $D$ the set of those points $z\in X$ such that $\abs{ \phi(z)}<1$, such that there
exists one and only one point
$\sigma(z)$ in the fiber $f^{-1}(z)$ at which $\abs \psi=\abs \phi$ and such that $\abs {\psi(t)} =1$
for all $t\neq \sigma(z)$ in $f^{-1}(z)$. 

And since $f$ is étale at every point of $Y$, 
Lemma \ref{lemm-same-henselization} above then ensures that $K(\sigma(z))\h=K(z)\h$ for all $z\in D$. 
\end{proof}

\begin{coro}
Let $K$ be a defectless valued field and let $X$ be an $n$-dimensional 
$K$-variety. Let $f_1,\ldots, f_n$ be $n$ functions belonging to $K[X]$; let
$f$ denote the $K$-morphism $(f_1,\ldots, f_n)\colon X\to \A^n_K$. Let $x\in X$ be
such that the $f_i$ are invertible at $x$, the $\rv(f_i(x))$ are algebraically independent over $\RV(K)$,  and $\RV(K(x))
=\RV(K)(\rv(f_1(x)), \ldots, \rv(f_n(x)))$. 
Then there exist a $K$-Zariski open subset $X'$ of $X$ containing $x$, 
a $K$-Zariski open subset $U$ of $\A^n_K$
such that $f$ induces a finite étale map from $X'$ to $U$, 
a
$K$-regular function $\phi$ on $U$ and a $K$-regular function $\psi$ on 
$X'$ with $\abs{\psi(x)}=\abs{\phi(f(x))}<1$, 
and a $K$-definable subset $D$ of $U$ containing $f(x)$ 
satisfying the following: 

\begin{itemize}[label=$\diamond$] 
\item for every $z\in D$ one has $\abs{\phi(z)}<1$ and
there exists one and only one element $\sigma(z)$ of $X'$ such that $f(\sigma(z))=z$
and $\abs{\psi(\sigma(z))}=\abs{\phi(z)}$ ; 
\item for every $z\in D$ 
and for every $t\neq \sigma(z)$ in $f^{-1}(z)$ one has 
$\abs{\psi(t)}=1$;
\item 
for all $z\in D$ one has $K(\sigma(z))\h=K(z)\h$. 
\end{itemize}
\end{coro}

\begin{proof}
As $\dim X=n$, the transcendence degree of $K(x)$ over $K$ is bounded by $n$. 
But since $\td(\RV(K(x))/\RV(K))=n$ by our assumptions, $\td(K(x)/K)$ is actually equal to $n$, the valued field
$K(x)$ is an Abhyankar extension of $K$, and $(f_1(x),\ldots, f_n(x))$ is an Abhyankar basis of $K(x)$
over $K$. The point $f(x)$ is therefore $K$-Zariski generic  on $\A^n_K$, which implies
that $f$ is quasi-finite at $x$.  The field $K(f(x))$ is the rational 
function field generated by the algebraically independent 
elements $f_1(x),\ldots, f_n(x)$ and its valuation is the Gauss valuation with
parameters $\abs{f_1(x)},\ldots, \abs{f_n(x)}$. Therefore $\RV(f(x))$ is the rational function field over $\RV(K)$
generated by the algebraically independent 
elements $(\rv(f_i(x)))_i$, so $\RV(K(f(x))$
is equal to $\RV(K(x))$; otherwise said, $K(x)$ is an immediate
extension of $K(f(x))$. And as $K$ is defectless, so is its Abhyankar extension $K(f(x))$. The corollary then follows directly from
the lemma above. 
\end{proof}

\begin{theo}\label{maintheo-balls}
Let $K$ be 
an algebraically closed valued field
and let $X$ be a $K$-variety. 
Let $x$ be a point of $X$ such that $K(x)$ is an Abhyankar extension
of $K$, let $n$ be the 
transcendence degree of $K(x)$ over $K$ 
and $e$ be the rational rank of $\abs{K(x)^\times}$ over $\abs{K^\times}$; 
let $\gamma=(\gamma_1,\ldots, \gamma_e)$ in $(\abs{K(x)^\times})^e$ be a basis of the
free abelian group $\abs{K(x)^\times}/\abs{K^\times}$, and
let $f_1,\ldots f_m$ be regular invertible functions on a $K$-Zariski
neighborhood of $x$ such that the 
$\rv(f_i(x))$ generate $\RV(K(x))$ over $\RV(K)$. Set $d=\rv(f(x))\in \RV(K(x))^m$ and
 $\lambda=\rv(f(\cdot)) \colon X\to \RV^m$.  
 There exist: 

\begin{itemize}[label=$\diamond$]
\item an $n$-dimensional integral $K$-subvariety $X'$ of $X$
containing $x$ on which all $f_i$ 
are defined; 
\item a finite étale $K$-morphism of schemes
$\pi\colon X'\to \Omega$ where $\Omega$ is a dense
affine $K$-Zariski open subset of $\A^n$;
\item a $K$-regular function $\phi$ on $\Omega$ and a $K$-regular function $\psi$ on 
$X'$ satisfying the condition $\abs{\psi(x)}=\abs{\phi(\pi(x))}<1$; 
\item a $Kd$-definable and $\RV(K)$-Zariski dense point $c$ of $\A^\gamma_\RV\times \A^{n-e}_\k$
such that the open $Kd$-definable ball $D:=\rv^{-1}(c)\subset \Omega$ contains $\pi(x)$ and is
$Kd$-atomic
(see \ref{generalized-variety}
for the meaning of $\A^\gamma_\RV$);
\item a finite $Kc$-definable subset $A$ of $\RV^m$ containing $d$, 
\end{itemize}
such that 
\begin{itemize}[label=$\bullet$]
\item $\pi^{-1}(D)=\lambda^{-1}(A)$; 

\item for all $z \in D$ one has $\abs {\phi(z)}<1$, 
there exists one and only one element $\sigma(z)$ of
$\pi^{-1}(z)\cap \lambda^{-1}(d)$ at which $\abs \psi=\abs \phi$ and for every
$t\neq \sigma(z)$ in $\pi^{-1}(z)\cap \lambda^{-1}(d)$ one has $\abs{\psi(t)}=1$.
\end{itemize}
\end{theo}

\subsubsection*{Comments}
As $c$ is Zariski-generic, the Abhyankar inequality ensures that $\rv^{-1}(c)$ consists
itself of $K$-Zariski dense points of $\A^n_K$. This is the reason why $\rv^{-1}(c)$ 
is automatically contained in $\Omega$ without having to enforce this explicitly.

Also note that one has necessarily $\sigma(\pi(x))=x$.

\begin{proof}
For proving the Theorem, we can replace $X$ by any
of its
$K$-subschemes containing $x$. 
So we can assume that $X$ is integral of dimension $n$ and that $x$ is $K$-Zariski generic. 
Note that by the Abhyankar inequality every point of $\lambda^{-1}(d)$ is $K$-Zariski  generic on $X$, 
so shrinking the $K$-scheme $X$ when needed will be harmless. 

\subsubsection{}
By the general theory of finitely generated abelian goupes (applied to the 
free group $\abs{K(X)^\times}/\abs{K^\times}$) there exists a matrix $M$ in 
$\GL_m(\Z)$ such that if one denotes by $f'_i$ (for $1\leq i\leq m$)
the $i$-th
component of $M(f_1,\ldots, f_m)$ for $M$ acting monomially, then $\abs{f'_i(x)}=\gamma_i;
\text{mod}\;\abs{K^\times}$
for $1\leq i\leq e$ and $\abs{f'_i(x)}\in \abs{K^\times}$ for $e+1\leq i\leq m$.
Choose elements $\ell_1,\ldots, \ell_m$ of $K^\times$ such that 
$\abs{f'_i(x)/\ell_i}=\gamma_i$ for all 
$1\leq i\leq e$ and 
$\abs{f'_i(x)/\ell_i}=1$
for some $e+1\leq i\leq m$.
Set $g_i=f'_i/\ell_i$ and $\mu=\rv(g(\cdot))\colon X\to \RV^m$. 
It follows from our constructions than $(g_i)$ can be obtained from $(f_i)$ by a
bijective monomial
transformation with coefficients in $\abs{K^\times}$ and integral exponents whose inverse is of the same form. 
As a consequence the $\rv(g_i(x))$ also generate $\RV(K(x))$ over $\RV(K)$, 
the element $\mu(x)$ of $\RV^m$ is inter-definable with $d$,
and $\lambda^{-1}(d)=\mu^{-1}(\mu(x))$. We thus can replace the $f_i$ by the $g_i$, which amounts to
assuming that $\abs{f_i(x)}=\gamma_i$ for $i\leq e$ and that $\abs{f_i(x)}=1$ for $i>e$. 
We then have $d=(\delta,z)$ for some $\delta=(\delta_i)_{1\leq i\leq e}$ with each 
$\delta_i$ non-zero of degree $\gamma_i$ and with  $z\in \k^{m-e}$, and the $\RV(K)$-Zariski-
closure of $d$ is equal to $\A^\gamma_\RV\times Z$ for some irreducible $\k(K)$-algebraic
subvariety $Z$ of $\A_\k^{n-e}$. As $\RV(K)(d)=\RV(K(x))$ and as $K(x)$ is Abhyankar
over $K$, the dimension of $\A^\gamma_\RV\times Z$ is equal to $n$, so $\dim(Z)=n-e$.

\subsubsection{}
By generic smoothness there exists a family  $h=(h_{e+1},\ldots, h_n)$ of polynomials
in $\k(K)[T_{e+1},\ldots, T_m]$
that induces
a
generically étale map from
$Z$ to $\A^{n-e}_\k$. 
Choose $H_{e+1},\ldots, H_n\in K[T_{r+1},\ldots, T_m]$ with integral coefficients that lift $h_{r+1},\ldots, h_n$, and
set 
\[\pi=(f_1,\ldots, f_r, H_{e+1}(f_{e+1},\ldots, f_m),\ldots, H_n(f_{e+1},\ldots, f_m));\]
this is a map from $X$ to $\A^n$. 
Set $c=(\mathrm{Id}, h)(d)=(\delta, h(z))$; this is a $Kd$-definable element of $\A^\gamma_\RV\times \A^{n-e}_\k$, 
and $\rv(\pi(x))=c$. 
As the map $h|_Z\colon Z \to \A^{n-e}_\k$ is generically étale, $h(z)$ is $\k(K)$-Zariski dense
in $\A^{n-e}_\k$, so $c$ is 
$\RV(K)$-Zariski dense in $\A^\gamma_\RV\times \A^{n-e}_\k$. 
Therefore by the Abhyankar inequality $\pi(x)$ is Zariski-generic, and
$\pi$ is generically finite; by construction, it induces the 
valuation $\eta_{K,(\gamma,1)}$ on $K(T_1,\ldots, T_n)$ 
(here the $1$ in the index stands for $n-e$
terms all equal to $1$, and we keep track
of the ground field $K$ 
in the notation 
since we will also have to consider this Gauss valuation with parameter $(\gamma,1)$ on a rational function field
over some valued extension of $K$), thus $\RV(K(\pi(x))=\RV(K)(\rv(\pi(x)))=\RV(K)(c)$.

%


Since $K(\pi(x))$ is an Abhyankar extension of $K$ (which is itself algebraically closed) $K(\pi(x))$ is defectless, so that
the degree of $K(x)\h$ over $K(\pi(x))\h$ is equal to the residue degree, namely that of $\RV(K)(d)$ over $\RV(K)(c)$. 
As $\RV(K)(d)$ is by construction separable over $\RV(K)(c)$, it follows that $K(x)\h$ is separable (and even tamely
ramified) over $K(\pi(x))\h$, so $K(x)$ is separable over $K(\pi(x))$.  Moreover it follows from the fact that 
$X$ is integral and $\pi(x)$ is $K$-Zariski generic that the scheme-theoretic fiber $\pi^{-1}(\pi(x))$ is reduced. Hence
$\pi$ is étale at $x$ and up to shrinking $X$
we can thus assume that $\pi$ is finite étale over some dense $K$-Zariski
affine open subset $\Omega$ of $\A^n_K$.

\subsubsection{}\label{comments-atomic}
Set $A=(\mathrm {Id},h)^{-1}(c)$.  
By the Abhyankar inequality, the $\RV(K)$-definable
subset $\lambda(X)$ is of dimension $\leq n$, so its fiber over $c$ is necessarily finite
by $\RV(K)$-genericity of $c$; therefore $A$ is finite. 

Set $D=\rv^{-1}(c)$. 
Then $D$ is a $Kc$-definable $n$-dimensional open polydisc in $\A^n$, 
which is contained in $\Omega$ since each of its point is $K$-Zariski-generic by Abhyankar
inequality, and 
$\pi^{-1}(D)=\lambda^{-1}(A)$. 
Moreover the polydisc $D$ is $\acl(Kc)$-atomic: using the fact that $c$ does not belong to $\acl(K)$ (unless
$n=0$ in which case $c$ is empty, $D$ is a point and the statement is obvious!), one sees it through
a straightforward adaptation of the proof of \cite[Lemma 3.8]{hrushovski-k2006} (be aware that in \loccit, the word
\textit{transitive} is used instead of \textit{atomic}). As $d\in \acl(Kc)$ the polydisc $D$ is particular $Kd$-atomic. 

Set $F=\lambda^{-1}(d)$.
The set $F$ is $Kd$-definable, and 
$\pi(F)\subset D$ by definition. 
By $Kd$-atomicity
$\pi(F)$ is either empty, either equal to the whole of $D$. But as $F\neq \emptyset$ (it contains $x$), we have $\pi(F)=D$. 
If $G$ is any $Kd$-definable subset of $F$, it follows once again from
atomicity that all fibers of $\pi|_G$ have the same cardinality $\nu(G)$; let us choose such a $G$ with $x\in G$ and
with $\nu(G)$ minimal; note that $\nu(G)>0$ since $G$ is non-empty (it contains $x$).
If $G'$ is any $Kd$-definable subset of $G$ then either $G'$ or its complement
$G''$ contains $x$, so that 
$\nu(G')=\nu(G)$ or $\nu(G'')=\nu(G)$ by minimality of $\nu(G)$; it follows that $G'=G$ or
$G''=G$, so $G$ is $Kd$-atomic. 
Therefore the type over $Kd$ of any point $t\in G$ does not depend on $t$; 
a fortiori, its type over $K$ does not depend on $t$ and we denote it by $p$.
We denote by $v$ the corresponding valuation on $K(X)$. It is an extension 
of the valuation $\eta_{K,(\gamma,1)}$ along the finite extension $K(T_1,\ldots, T_n)\hookrightarrow K(X)$. 
Since the extensions of $\eta_{K,(\gamma,1)}$ to $K(X)$ form a system of independent valuations
there exist $\psi \in K(X)$ 
such that $v(\psi)<1$ and $w(\psi)=1$ for every extension $w\neq v$ of 
$\eta_{K,(\gamma,1)}$ to $K(X)$. Up to replacing $\psi$ by $\psi^N$
for $N$ large enough, we may assume that there exists $\phi\in K(T_1,\ldots, T_n)$
such that $v(\psi)=\eta_{K,(\gamma,1)}(\phi)$. 
Up to shrinking $\Omega$ (and $X$) we can assume that 
$\phi$ is a $K$-regular function on
$\Omega$ and $\psi$ is a $K$-regular function on $X$. By construction $\abs \phi<1$ 
at every point of $D$ and $\abs \psi=\abs \phi$ identically on $G$. 
As $\nu(G)>0$ every point $z$ of $D$ has a pre-image $s$ on $G$, and $\abs{\psi(s)}$ is then equal to $\abs{\phi(z)}$.

Our purpose is now to show that for every $z\in D$ there is a unique element $s$ of 
$\pi^{-1}(z)$
with
$\lambda(s)=d$ and $\abs {\psi(s)}=\abs{\phi(z)}$; 
by $Kd$-atomicity of $D$ it suffices
to exhibit one model $N$ of $\acvf$ containing $Kd$ and one point $z\in D(N)$
for which this property holds. 

\subsubsection{}
Let $M$ be an algebraically closed
valued field containing $K$
such  that $\abs{M^\times}=\abs{K^\times}$
and $z\in \k(M)$. 
Let $N$ be a model of $\acvf$ containing $M$
and an $M$-Zariski generic $n$-uple $\xi=(\xi_1,\ldots, \xi_n)$ such that 
the valuation of $M(T_1,\ldots, T_n)$ induced by 
$\xi$ is the composition $u\circ \eta_{M,(\gamma, 1)}$
where $u$ is the lexicographic valuation centered at the $\k(M)$-rational
point $h(z)$ of $\A_{\k(M)}^{n-r}$, with respect to the regular system of parameters
$(T_i-h_i(z))_{r+1\leq i\leq n}$. The point $\xi$ belongs to $D(N)$
and we are going to show that 
that there is a unique $\omega\in \pi^{-1}(\xi)$ 
with $\lambda(\omega)=d$
and $\abs{\psi(\omega)}=\abs{\phi(\xi)}$. We know by the above that there is at least one such
$\omega$, so we fix one. 

\subsubsection{}\label{v-w}
The point $\omega$ is $M$-Zariski generic on $X$, and the induced valuation 
on $M(X)$ is of the form $\theta\circ w$, where $w$ is an extension of $\eta_{\gamma,1}$
to $M(X)$ and $\theta$ is some valuation on $\k((M(X),w))$ lying above $u$. 
Its restriction to $K(T_1,\ldots, T_n)$ is thus equal to $u'\circ \eta_{K,(\gamma,1)}$, 
where $u'$ is the restriction of $u$ to $\k(K)(T_{r+1},\ldots, T_n)$. 
But by construction the center of $u$ on $\A^{n-r}_{\k(M)}$ lies above the generic
point of $\A^{n-r}_{\k(K)}$, so the valuation $u'$ is trivial, and the valuation on 
$K(T_1,\ldots, T_n)$ induced by $\omega$ is finally equal to 
$\eta_{K,(\gamma,1)}$. Hence the valuation induced
by $\omega$ on $K(X)$ is an extension of $\eta_{K,(\gamma,1)}$, 
and the equality $\abs{\psi(\omega)}=\abs{\pi(\xi)}$ ensures (by the choice of 
$\psi$)
that this extension is precisely
$v$. Otherwise said, the type of $\omega$ over $K$ is equal to $p$. 

Now as $v$ is a valuation extending the Gauss valuation 
$\eta_{K,(\gamma, 1)}$ to $K(X)$, and $w$
is an valuation extending the Gauss valuation 
$\eta_{M,(\gamma,1)}$ to $M(X)$, it follows from
\cite[\S A.1.3]{ducros-h-l-y2024} that $w$ is the unique extension of
$\eta_{M,(\gamma, 1)}$ to $M(X)$
restricting to $v$, and that $\RV(M(X),w)$ is the (generalized) fraction field of
the tensor product
$\RV(K(X), v)\otimes_{\RV(K)}\RV(M)$. 
As $\abs{M\gpm}=\abs{K\gpm}$ this means that 
$w(M(X)\gpm)$ is equal to $v(K(X)\gpm)$ and that 
$\k(M(X),w)$ is the usual fraction field of
$\k(K(X),v)\otimes_{\k(K)}\k(M)$. 

But $v$ is the valuation 
on $K(X)$ induced by $x$, 
for the latter lies on $G$. 
So one has 
$\k(K(X), v)=\k(K)(Z')$, 
$\k(M(X),w)=\k(M)(Z')$ and 
\[w(M(X)\gpm)=v(K(X)\gpm)=\abs{K\gpm}\gamma^\Z.\]

The valuation induced on $M(X)$ by
$\omega$ is equal to $\theta\circ w$. 
As $\lambda(\omega)=d=(\gamma,z)$, it follows that $\rv(f_i(\omega))=z_i$ for all $i\geq r+1$, so that 
the valuation $\theta$ of $\k(M)(Z')$ is necessarily centered at the $\k(M)$-rational point $z$ of the
scheme
$Z'_{\k(M)}$. 

\subsubsection{}\label{valuation-unique}
The map $h\colon Z_{\k(M)}'\to \A^{n-r}_{\k(M)}$ is étale. The valuation $\theta$ is therefore the only valuation on 
$\k(M)(Z')$ centered at $z$ and lying above $u$, and 
one has moreover the equality
$(\k(M)(Z'),\theta)\h=(\k(M)(T_{r+1},\ldots, T_n),u)\h$. 

Indeed, let $S$ denote the spectrum of 
the ring of integers of $(\k(M)(T_{r+1},\ldots, T_n),u)\h$, 
and let $s$ be its closed point.  
The special fiber of
the $S$-scheme $Z'_{\k(M)}\times_{\A^n_{\k(M)}}S$ has only one pre-image $t$
over $z$, and $t$ and $s$ have the same residue field
since $z$ and $h(z)$ have the same residue field $\k(M)$. As
$S$ is henselian the connected
component $S'$ of $Z'_{\k(M)}\times_{\A^n_{\k(M)}}S$ that contains $t$ is the spectrum of 
a finite local $S$-scheme; since $Z'_{\k(M)}\times_{\A^n_{\k(M)}}S$ is étale over $S$
and since $s$ and $t$ have the same residue field, $S'\simeq S$. Now is $\varpi$ is any
valuation on $\k(M)(Z')$ centered at $z$ and lying above $u$ and if $S''$ denotes the spectrum
of the ring of integers of $(\k(M)(Z'),\varpi)\h$, then $S''$ dominates $S'$, so is equal to $S'$
as the latter is already the spectrum of a henselian valuation ring (because $S'\simeq S$). 
This proves at the same time the uniqueness of $\varpi$ and the fact that its henselization 
coincides with that of $u$. 

\subsubsection{}\label{immediate-extension}
By the above, the valuations $w$ and $\theta$ are uniquely determined; \ie, they do not depend
on the choice of $\omega$; therefore
the type of $\omega$ over $M$ (which is entirely encoded
in the valuation $\theta\circ w$) also does not depend on $\omega$. 
And since $\theta$ and $u$ have the same henselization, the former is an immediate 
extension of the latter. We have also seen at the end of \ref{v-w}
that
$w(M(X))\gpm)=\abs{K\gpm}\gamma^\Z$, which implies that $w(M(X))\gpm)
=\eta_{M,(\gamma,1)}(M(T_1,\ldots, T_n)\gpm)$; thus $\theta \circ w$ is an immediate extension 
of $u\circ \eta_{M,(\gamma,1)}$.

\subsubsection{}
Now consider all pairwise distinct
pre-images $\omega=\omega_1,\ldots, \omega_\ell$ of $\xi$ on $\Delta(N)$ at
which $\abs \psi=1$. 
Fix $j$. 
It follows from \ref{immediate-extension} 
that $M(\omega)$ is an immediate
extension of $M(\xi)$. As the latter is Abhyankar over $M$, which is itself
algebraically closed, it is defectless. Lemma \ref{etale-isomorphism}
then ensures that $M(\omega)\h=M(\xi)\h$. 
As a consequence
$\omega$ is definable over $M(\xi)$, so there exists an $M$-definable map $\tau$
from $\A^n$ to $X$ such that $\tau(\xi)=\omega$, so 
$\tau(\pi(\omega))=\omega$. But by \ref{immediate-extension}, all $\omega_j$ have the same type over $M$, 
so $\tau(\pi(\omega_j))=\omega_j$ for all $j$, which means that $\omega_j=\omega$. 

\subsubsection{Conclusion}
We have thus proved that for every $z\in D$ there is a unique element
$s$ of 
$\pi^{-1}(z)$ such that $\lambda(s)=d$ and  $\abs{\psi(s)}=\abs{\phi(z)}$, which we denote by $\sigma(z)$. 
Now let $t$ be an element of $\pi^{-1}(z)\cap \lambda^{-1}(d)$ different from 
$\sigma(z)$. The valuation $w$ induced by $t$ on $K(X)$ is an extension of
$\eta_{K,(\gamma,1)}$, and it is not equal  to $v$ since
$\abs{\psi(t)}\neq 1$. Then $w(\psi)=1 $, which
exactly means that $\abs{\psi(t)}=1$, and we are done (there is no explicit $X'$
because we have shrunken $X$ several times during the proof). 
\end{proof}

\section{Nash structure of curves: covering by nice charts}\label{nash-structure-coverings}
Our purpose is now
to elucidate the structure of the stable completion of a curve beyond its topology. 
The genuine structure theorem we have in mind will in fact be stated and proved only in the next section, 
but it will rest on a first structure theorem 
(with a less tractable formulation) which is the object of the current section. 

We will use freely the results of sections \ref{topology-curve} and \ref{topology-curve-complement}
on the topology of curves, and especially the notion of an admissible skeleton
(Definition \ref{defi-admissible-skeleton}) and the associated maps
$\rho,\widehat \rho, \rr$ and $\wrr$ (\ref{rho}, \ref{construction-retraction-huber}).

\begin{defi}\label{defi-nice-charts}
Let $K$ be a  valued field
and let 
$X$ be an algebraic $K$-curve. A \textit{nice $K$-chart}
on $X$ is a pair $(U,f)$ where $U$ is a $K$-definable subset of $X$
and where
$f\colon X'\to \Omega$ is an étale map defined over $K$ whose source $X'$ is a $K$-Zariski open subset of $X$ containing $U$
and whose target $\Omega$ is a dense $K$-Zariski open subset of $\A^1_K$; moreover
we require that
there exist a $K$-regular map $\psi$ on $X'$, a $K$-regular map
$\phi$ on $\Omega$,  and that one of the following hold:  

\begin{itemize}[label=$\diamond$]
\item (Charts of type 1) The set 
$\Omega$ contains an open disc $D$ centered at the origin
and:
\begin{itemize}[label=$\bullet$]
\item for all $z\in D$ one has
$\abs{\phi(z)}<1$ and 
there exists a unique element $\sigma(z)$ in 
$f^{-1}(z)$ 
such that $\abs{\psi(\sigma(z)}=\abs{\phi(z)}$; 
\item for every $z\in D$ and every $t\neq \sigma(z)$ in $f^{-1}(z)$
one has $\abs{\psi(t)}=1$;
\item $\sigma(D)=U$. 

\end{itemize}

\item (Charts of type 2)
The map $f$ is finite, the chart $(U,f)$ comes
implicitly with an extra datum consisting of a $K$-definable
point $s$ of $\widehat U$ called the \emph{vertex} of the chart
such that $\{s\}$ is an admissible skeleton
of the stable completion $\widehat U$, 
the image $f(s)$ is equal to $\eta_{0,\abs{f(s)}}$ 
and $s$ is the unique pre-image of $\eta_{0,\abs{f(s)}}$ 
on $\widehat U$, 
and there exist a
family $h=(h_1,\ldots, h_m)$ of regular functions
on $X'$,  and for every $\zeta\in C_{s,\widehat U}$ an open disc
$D_\zeta$ of radius
$\abs {f(s)}$, a finite subset $A_\zeta$ of $\k^m$ and an element
$a_\zeta$ of $A$, satisfying the following, where
$\lambda$ denotes
the map $\rv\circ h\colon X'\to \RV^m$: 

\begin{itemize}[label=$\bullet$] 
\item $D_\zeta$, $A_\zeta$ and $a_\zeta$ depend $K$-definably 
on $\zeta$; 
\item for every $\zeta$
one has $f^{-1}(D_\zeta)=\coprod_{a\in A_\zeta}\lambda^{-1}(a)$; 
\item for every $\zeta\in C_{s,\widehat U}$ and every $z\in D_\zeta$ one has 
$\abs{\phi(z)}<1$; 
\item for every $\zeta\in C_{s,\widehat U}$ and every $z\in D_\zeta$
there is a unique point $\sigma(z)$
in $f^{-1}(z)\cap \lambda^{-1}(a_\zeta)$ such that 
$\abs{\psi(\sigma(z))}=\abs{\phi(z)}$; 
\item for every $\zeta\in C_{s,\widehat U}$, every $z\in D_\zeta$
and every $t\neq \sigma(z)$ in $f^{-1}(z)\cap
\lambda^{-1}(a_\zeta)$
one has $\abs{\psi(t)}=1$;
\item for every $\zeta\in C_{s,\widehat U}$ one has
$\sigma(D_\zeta)=\rr_{U,\{s\}}^{-1}(\zeta)$. 
\end{itemize}

\item (Charts of type 3)
The map $f$ is finite, the
set $\Omega$ contains an open annulus
$A=\{R_1<\abs T<R_2\}$ of $\A^1$
with $R_1$ and $R_2$ in $\abs K \cup\{\infty\}$
such that: 
\begin{itemize}[label=$\bullet$]
\item $\abs{\phi(z)}<1$ for all $z\in A$; 
\item for all $z\in A$ there exists a unique element $\sigma(z)$ in 
$f^{-1}(z)$ 
such that $\abs{\psi((\sigma(z))}=\abs{\phi(z)}$; 
\item for every $z\in A$ and every $t\neq \sigma(z)$ in $f^{-1}(z)$
one has $\abs{\psi(t)}=1$; 

\item $\sigma(A)=U$. 

\end{itemize}

\end{itemize}

\end{defi}

\subsection{}Let
$K$ be a valued field and let $(U,f)$ be a nice
$K$-chart. 
In what follows,
we will distinguish three cases according to the type
of the chart. For each type we will use the notation of the corresponding definition above. 

\subsubsection{The case where $(U,f)$ is of type 1}
It follows then from Example \ref{example-nash-iso}
that $f$ induces a Nash isomorphism between $U$ and $D$, 
and from Lemma \ref{lemm-nash-topology}
that $\widehat U$ is open in $\widehat {X'}$, hence
in $\widehat X$, and that it is closed in $f^{-1}(\widehat D)$.

\subsubsection{The case where
$(U,f)$ is of type 2}\label{nice-charts2-definable-homeo}
Let $\zeta\in C_{s,\widehat U}$. The pre-image
$f^{-1}(\widehat{D_\zeta})$ is an open $K\zeta$-definable subset of $\widehat {X'}$
(so it is open in 
$\widehat X$ too). 
By assumption
one has $f^{-1}(\widehat{D_\zeta})=\coprod_{a\in A_\zeta}\widehat{\lambda^{-1}(a)}$; each of the 
$\widehat{\lambda^{-1}(a)}$ is open in $\widehat{X'}$ (it is described by strict inequalities), so in particular
$\widehat{\lambda^{-1}(a_\zeta)}$ is open and closed in $f^{-1}(\widehat{D_\zeta})$. 

As $\wrr^{-1}(\zeta)$ is by definition of 
a
nice $K$-chart of type 2 the intersection of $\widehat{\lambda^{-1}(a_\zeta)}$
with $\{\abs \psi<1\}$, it is open in $f^{-1}(\widehat{D_\zeta})$. 
As a consequence,
and since the étale map $X'\to \Omega$ induces
an open map at the level of stable completions
(Proposition \ref{flat-open}), the definable bijection $f|_{\rr^{-1}_{U,s}(\zeta)}\colon 
\rr^{-1}(\zeta)\to D_\zeta$
induces a definable homeomorphism $\wrr_{U,s}^{-1}(\zeta)\simeq \widehat{D_\zeta}$, 
so $f$ induces a Nash ismorphism 
$\rr^{-1}_{U,s}(\zeta)\simeq D_\zeta$. 
In view of Lemma \ref{lemm-nash-topology}, 
$\wrr_{U,s}^{-1}(\zeta)$ is open in $\widehat{X'}$, 
hence in $\widehat X$, and it is closed
in $f^{-1}(\widehat D_\zeta)$.

\subsubsection{The case where $(U,f)$ is of type 3} 
It follows then from Example \ref{example-nash-iso}
that $f$ induces a Nash isomorphism between $U$ and $A$, 
and from Lemma \ref{lemm-nash-topology}
that $\widehat U$ is open in $\widehat {X'}$, hence
in $\widehat X$,
and that it is closed in $f^{-1}(\widehat A)$.

\begin{rema}\label{common-remark}
Let $(U,f)$ be a nice $K$-chart of type (1) or (3). 
It follows from \ref{lemm-same-henselization} that for every model $M$ of $\acvf$ containing $K$
and every $M$-point of $x$ of $U$ the valued fields $M(u)$ and $M(f(u))$
have the same henselization; in particular they have the same value group and the same 
residue field. It follows that if $x$ is a non-simple point of $\widehat U$, the map
$f$ induces an isomorphism between the residue curves $C_x$ and $C_{f(x)}$. 
\end{rema}

\begin{rema}
If $a\in K^\times$ and $(U,f)$ is a nice chart of type 1, \resp 2, \resp 3, then 
so is $(U,f/a)$ (with the same vertex in the case of type 2). Therefore if $\abs {K^\times}$ is
divisible, so that every $K$-definable element of $\Gamma$ belongs to $\abs{K^\times}$, 
every nice chart of type 1 (\resp 2) can be rescaled so that $D$ is of radius $1$
(\resp so that the vertex lies above $\eta_{0,1}$). 
\end{rema}

\begin{theo}\label{theo-core}
Let $K$ be an algebraically closed valued field, 
let $X$ be a
smooth algebraic $K$-curve and let $U$ be a 
$K$-definable subset of $X$ such that $\widehat U$ has no
isolated simple point. There exist finitely many nice $K$-charts $(U_i,f_i)$
on 
the curve $X$
such that $\widehat U=\bigcup \widehat {U_i}$. 
\end{theo}

\begin{proof}
It is quite long,  with several steps.  Its core relies on Theorem
\ref{maintheo-balls} that exhibits atomic balls. 

\subsubsection{Preliminaries}
We can assume that $K$ is maximally complete.  Indeed, 
if $K$ is trivially valued there is nothing to do.  If not, $K$
is a model of $\acvf$ and as
a nice $K$-chart can be described using finitely many parameters in $K$,  and as
being a
nice $K$-chart is a first-order formula in those parameters, we can prove 
the theorem on any model larger than $K$, which allows one
to assume that $K$ is maximally complete. 

Let $x\in \widehat U(M)$
for some model $M$ of $\acvf$ containing
$K$. By compactness it suffices to exhibit a nice $K$-chart $V\subset U$
such that $x\in \widehat V$. 
Let $\alpha\in U(N)$ be a realization of $x$ (which is a type over $M$) for some model $N\supset M$. 
We shall distinguish three cases according to the type of 
$\alpha$ \textit{over $K$}.

\subsubsection{The ``immediate extension" case}\label{implicit-function-theorem}
Assume that $K(\alpha)$ is an immediate extension of $K$. 
Since $K$ is maximally complete this amounts to saying that $\alpha\in X(K)$. 
Our purpose is to build a nice $K$-chart $(V,f)$ of type 1
with $V\subset U$ such that $\alpha\in V$, 
which will ensure that $x\in \widehat V$. 

As the curve $X$ is smooth, there exists some $K$-Zariski neighborhood $X'$ of 
$\alpha$ in $X$ and a $K$-definable étale map $f$ from $X'$ to
$\A^1$ mapping $\alpha$ to the origin.
In view of the local structure of étale maps we can assume
that there exists an affine $K$-Zariski open neighborhood 
$\Omega=\spec R$ of $0$ in $\A^1_K$ and a monic polynomial $P\in R[\tau]$
such that $X'$ can be identified with the invertible locus of $P'$ in $\spec R[\tau]/P$.
Now $\beta=\tau(\alpha)$ is a simple root of $P_0$, 
where $P_0$ denotes the polynomial in $K[\tau]$ obtained by evaluating each coefficient
of $P$ at $0$. By suitably translating $\tau$ we can assume that $\beta=0$, so $P=\sum a_i\tau_i$ with 
$a_0(0)=0$ and $a_1(0)\neq 0$ (because $\beta=0$ is a simple root of $P_0$). 

We handle
first the case where $K$ is trivially valued.
Let $F$ be any model of $\acvf$ containing $K$ and let $z$ be an element of $F$ with
$\abs z<1$. As all $a_i$ are rational functions with coefficients in $K$ and no poles at the origin,
and as $a_1(0)\neq 0$ does not vanish at the origin, we have $\abs{a_0(z)}<1$, $\abs{a_1(z)}=1$ and
$\abs{a_i(z)}\leq 1$ for all $i\geq 2$. Then the basic theory of Newton polygons show that
the monic polynomial $P_z$ has a unique root $\sigma(z)$ with $\abs{\sigma(z)}<1$ and that 
$\abs{\sigma(z)}=\abs{a_0(z)}$, and that $\abs \gamma=1$ for every root
$\gamma\neq \sigma(z)$ of $P$; it is also clear that $\sigma(z)\in X'$. 
Let $D$ be the open unit disc centered at the origin
(note that $D\subset \Omega$ since $K$ is trivially valued) and 
set $V=\sigma(D)$. 
We observe that $\sigma(0)$ is necessarily equal to $\alpha$. 
Then $(V,f)$ is a nice $K$-chart of type $1$ that contains $\alpha$ (it fulfills the definition with 
$\phi=a_0$ and $\psi=1$); moreover $\sigma^{-1}(U)$ is a $K$-definable subset of $D$ containing
the origin, so it is either $\{0\}$, either $D$, and it cannot be  $\{0\}$ since $x$ is not isolated in $\widehat U$; 
hence $V\subset U$ and we are done.

Assume now that the valuation of $K$ is non-trivial, so
$K$ is a model of $\acvf$. Again, the theory of Newton polygons ensures that there exists 
an open $K$-definable disc $D$ centered at the origin and contained in $\Omega$ such that: 

\begin{itemize}[label=$\diamond$]
\item $a_1$ does not vanish on $D$;
\item for every point $z$ of the disc $D$,
the polynomial $R_z$ has a unique root $\sigma(z)$ such that
$\abs {\sigma(z)}=\abs{a_0(z)/a_1(z)}$, and $\sigma(z)$ lies in $X'$; 
\item there exist some $\lambda \in \abs{K^\times}$ such that for every $z\in D$
and every root $\gamma\neq \sigma(z)$ of $P_z$ one has $\abs{\gamma}
> \abs \lambda >\abs{a_0(z)/a_1(z)}$.
\end{itemize}
Now  after shrinking if needed $\Omega$ (and $X'$) so
that $a_1\in \mathscr O(\Omega)^\times$, we set
$\phi=\frac 1 \lambda \cdot \frac {a_0}{a_1}$; this is a $K$-regular function on 
$\Omega$. And after shrinking once again if needed $\Omega$
and $X'$ so that 
for all $z\in \Omega$ the polynomial $P_z$ does not vanish at $\lambda$, we set
$\psi=\frac \tau{\tau-\lambda}$; this is a $K$-regular function on $X'$. By construction
for every $z\in D$ and every $t\in f^{-1}(z)$ either $t=\sigma(z)$ and $\abs{\psi(t)}=\abs {\phi(z)}<1$,
or $t\neq \sigma(z)$ and $\abs{\psi(t)}=1$.
As $x$ is not isolated in $\widehat U$ we can shrink $D$ so that $\sigma(D)\subset U$. 
Set $V=\sigma(D)$. Then $(V,f)$ is a nice $K$-chart of type $1$ 
with $\alpha\in V\subset U$, and we are done.

 \subsubsection{The ``group extension" case}
 Assume now that $\abs{K(\alpha)^\times}\supsetneq \abs{K^\times}$. 
 Using Abhyankar inequality and the fact that 
 $\abs{K^\times}$ is divisible, we see that $\abs{K(\alpha)^\times}$ is equal to $\abs {K^\times} \oplus r^\Z$ for some 
 $r\notin \abs{K^\times}$, and that $\k(K(\alpha))=\k(K)$. 
 Choose a $K$-Zariski open subset $X'$ in $X$ such that $\alpha\in X'$ and a $K$-regular function $f$
 on $X'$ such that $\abs{f(\alpha)}=r$. Then $f(\alpha)$ does not belong to $K$, so $f$ is quasi-finite at $\alpha$; and $K(\alpha)$
 is by construction an immediate extension of $K(f(\alpha))$, which is Abhyankar over $K$ and thus defectless. 
 By Lemma \ref{etale-isomorphism} there exists a $K$-Zariski open
 neighborhood $X'$ of $\alpha$ in $X$
 such that $f$ induces a finite étale map 
 from $X'$ to some $K$-Zariski dense open $\Omega$ of $\A^1_K$, 
 a $K$-regular function $\psi$ on $X'$ such that $\abs {\psi(\alpha)}<1$, 
 a $K$-regular function $\phi$
 on $\Omega$  and a $K$-definable subset $D$
 of $\Omega$ containing $f(\alpha)$ and satisfying the following property: 
 for every $z\in D$ one has $\abs{\phi(z)}<1$, there is a unique element 
 $\sigma(z)$ of $f^{-1}(z)$ such that $\abs{\psi(\sigma(z))}=\abs{\phi(z)}$, 
 and $\abs{\psi(t)}=1$ for every $t\neq \sigma(z)$ in
 $f^{-1}(z)$.

Set $\beta=f(\alpha)$. As $\abs \beta$
is not torsion modulo the divisible group 
$\abs {K^\times}$ one has $\abs{\sum \lambda_i \beta^i}=\max \abs{\lambda_i}\cdot \abs \beta^i$
for every polynomial $\sum \lambda_i T^i$ in $K[T]$; thus if $P\in K[T]$ is
such that $\abs{P(\beta)}\leq 1$ (\resp $<1$) then there exists a $K$-definable open annulus containing
$\beta$ on which
 $\abs P\leq 1$ (\resp $<1$). Consequently, the $K$-definable subset $\sigma^{-1}(U)$ of $D$, which contains
 $\beta$, contains a $K$-definable open annulus $A$ with $\beta \in A$. 
 Set $V=\sigma(A)$. Then by construction $(V,f)$ is a nice $K$-chart of
 type 3 contained in $U$; since $\alpha\in V$, the stable completion of $V$ contains $x$. 
 
 \subsubsection{The ``residue extension" case}
 We finally assume that $\k(K(\alpha))\supsetneq \k(K)$, so by Abhyankar inequality and by divisibility of $\abs{K^\times}$
 the group $\abs{K(\alpha)^\times}$ is equal to $\abs {K^\times}$ and the field $\k(K(\alpha))$ is a function field of transcendence degree 
 $1$ over $\k(K)$. Up to shrinking $X$ (\ie, replacing it by a
 $K$-Zariski open subset containing $\alpha$) we can choose a family
 $\lambda=(f_1,\ldots, f_m)$
 of $K$-regular invertible functions on $X$ such that $\abs{f_i(\alpha)}=1$ for all $i$ and such that the $\rv(\lambda_i(\alpha))$
 generate $\k(K(\alpha))$ over $\k(K)$. We set $\lambda =(\rv \circ f_i)_i\colon X\to\RV^m$ and $d=\lambda(\alpha)$.

\paragraph{}\label{properties-core-case}
In view of Theorem \ref{maintheo-balls} we can shrink $X$ once again to ensure that there
exist: 
 \begin{itemize}[label=$\diamond$]

 \item a $K$-definable finite étale map $f\colon X\to \Omega$ where $\Omega$ is a $K$-Zariski dense
 affine open subset of
 $\A^1_K$, a $K$-regular function $\phi$ on $\Omega$
 and a $K$-regular function $\psi$ on $X$; 
 \item a $K$-transcendental $d$-definable element $c$ of $\k$,
such that the open $Kd$-definable ball $D:=\rv^{-1}(c)\subset \Omega$ contains $f(\alpha)$ and is
$Kd$-atomic; 
 \item a finite $c$-definable subset $A$ of $\k(K)^m$ containing $d$,
\end{itemize}
such that: 
\begin{itemize}[label=$\bullet$]
\item $f^{-1}(D)=\coprod_{a\in A}\lambda^{-1}(a)$; 
\item for every $z\in D$ one has
$\abs{\phi(z)}<1$ and there exists a unique element $\sigma(z)$
of $f^{-1}(z)\cap \lambda^{-1}(d)$ such that $\abs{\psi(\sigma(z))}=\abs{\phi(z)}$; 
\item for every $z\in D$ one has $\abs{\phi(z)}>1$ and 
and for every $t\neq \sigma(z)$  in $f^{-1}(z)\cap \lambda^{-1}(d)$ one has
$\abs{\psi(t)}=1$. 
\end{itemize}

Set $E=\sigma(D)$. This is a $Kd$-definable subset of $X$
definably isomorphic to $D$ via $f$, so it is in particular 
$Kd$-atomic. We also have
$E=f^{-1}(D)\cap \lambda^{-1}(d)\cap \{\abs \psi=\abs \phi\}$
as well as $E=f^{-1}(D)\cap \lambda^{-1}(d)\cap \{\abs \psi<1\}$; 
applying stable completion to both sides of the
latter equality shows
that $\widehat E$ is open in $\widehat X$. 
By openness of étale (in fact, even flat)
maps at the
level of stable completions (Proposition \ref{flat-open}), 
the map $f$ induces a Nash isomorphism between $E$ and
$D$.

 The type of $\alpha$ over $K$ is stably dominated, so it defines a point
 $s$ of $\widehat X(K)$, which lies on $\widehat U(K)$. As $\rv(f(\alpha))=c$, which is $K$-transcendental, the
 image $f(s)$ is the Gauss point $\eta_{0,1}$ of $\widehat{\A^1_K}$. Note that $\eta_{0,1}\notin
 \widehat D$, so $s\notin \widehat E$. 
 The fiber $f^{-1}(\eta_{0,1})$ in $\widehat U$ is finite and $K$-definable. In view of the explicit description 
 of the topology of $\widehat X$ there exists a $K$-definable subset $U'$ of $U$, defined by a conjunction of 
 strict inequalities, such that $s$ is the unique pre-image of $\eta_{0,1}$ on
 $\widehat {U'}$. Then as $s\in \widehat{U'}$ the point $\alpha$ belongs to $U'$, so $U'\cap E$ is 
 a non-empty $Kd$-definable subset of $E$, hence $U'\cap E=E$ by
 $Kd$-atomicity, and $E\subset U'$. Set $U''=U'\cap \{\abs \psi=\abs \phi\}$.
 Then $E\subset U''$ and $s$ is the only pre-image of $\eta_{0,1}$ under $f$ in $\widehat{U''}$.

Let us choose a $K$-definable
 admissible skeleton $\Sigma$ of $\widehat {U''}$ containing $s$
 and set $W=\rr_{U'',\Sigma}^{-1}(s)
\cap U''$. 
Then $W$ is a $K$-definable subset of $U''$, and $\widehat W
 =\wrr_{U'',\Sigma}^{-1}(s)$, so $\{s\}$ is a $K$-definable admissible skeleton of
 $\widehat W$. As $s\in \widehat W$ the simple point $\alpha$ belongs to $W$. The intersection 
 $W\cap E$ is then a non-empty $Kd$-definable subset of $E$, so it is equal to 
 the whole of $E$; hence $E\subset W$.

\paragraph{}
Set $\xi=\rr_{W,\{s\}}(\alpha)$. 
Our purpose is now
to show that $\xi$ is inter-definable with $d$ over $K$
and that $\widehat E=\wrr_{W,s}^{-1}(\xi)$ (and so that
$E=\rr_{W,s}^{-1}(\xi)$). 

Let $o$ be a point of $\widehat E$ and let $I$ be the unique segment on $\widehat W$ with endpoints $o$ and $s$; let $t
\in (o,s]$ be the ``smallest" point of $I$ outside $\widehat E$. 
Since $f$ induces a definable homemorphism
from $\widehat E$ to $\widehat D$, the image $f([o,t))$ is an interval on $\widehat D$ that is closed in  $\widehat D$; since
moreover
$f([o,t])$
is compact, one has necessarily $f(t)=\eta_{0,1}$. As $s$ is the only pre-image of $\eta_{0,1}$ on $\widehat W$, we
see that $t=s$, so $[o,s)\subset \widehat E$ and
$\widehat E\cup\{s\}$ identifies to the canonical compactification of $\widehat E$ (which is a tree with one branch at infinity, 
since it is homeomorphic to $\widehat D$). 
If $ o'$ is another point of $E$, and if $I'$ denotes  the unique segment on $\widehat W$ 
joining $o'$ to $s$, 
then $I'\cap I=[o,s]\cap [o',s]$ and it is of the form $[\tau,s]$ for some $\tau \in \widehat E$; hence the germs of
$I$ and $I'$ at $s$ are the same, so $\wrr_{W,s}(o)=\wrr_{W,s}(o')$. Thus $\wrr_{W,s}$ is constant on
$\widehat E$ with value $\xi$, 
for $\wrr_{W,s}(\alpha)=\xi$. The set $E$ being $Kd$-definable, this already shows that $\xi$ is $Kd$-definable. 

Now let $w$ be a point of $\widehat W$ with $\wrr_{W,s}(w)=\xi$. 
Then the unique interval $[w,s]$ on $\widehat W$ has the same germ
at $s$ than $[\alpha,s]$, so in particular there exists $\tau\in \widehat E$ with $\tau\in [w,s]$. The intersection $\widehat E\cap [w,s)$ is then
a non-empty definable open subset of $[w,s)$ (we use the fact that 
$\widehat E$ is open in $\widehat X$). If it were not equal to the whole of $[w,s)$ the connected
component of $\tau$ in this intersection would be of the form $(\tau',s)$; then $\tau'$ would belong to
$\partial \widehat E=\{s\}$, whence a contradiction. As a consequence $w\in \widehat E$
and $\widehat E=\wrr_{W,s}^{-1}(\xi)$. As $\lambda$ takes the constant value $d$ on $E$ we see that $d$ is $K\xi$-definable, 
as announced. 

\paragraph{Conclusion}
Let $\Theta$ be the subset of $C_{s,\widehat W}$ consisting of points $\zeta$ such that there exist: 

\begin{itemize}
 \item a $K\zeta$-definable element $c_\zeta$ of $\k(K)$; 
 \item a finite $K\zeta$-definable subset $A_\zeta$ of $\k(K)^m$; 
 \item a $K\zeta$-definable element $a_\zeta$ of $A_\zeta$ 
\end{itemize}
such that the following hold, where $D_\zeta$ denotes the open unit ball  $\rv^{-1}(c_\zeta)$: 
\begin{itemize}[label=$\bullet$]
\item $f^{-1}(D_\zeta)=\coprod_{a\in A_\zeta}\lambda^{-1}(a)$;
\item $\abs \phi<1$ everywhere on $D_\zeta$; 
\item $\rr_{W,s}^{-1}(\zeta)=f^{-1}(D_\zeta)\cap \lambda^{-1}(a_\zeta)\cap \{
\abs\psi=\abs \phi\}$;
\item $\rr_{W,s}^{-1}(\zeta)=f^{-1}(D_\zeta)\cap \lambda^{-1}(a_\zeta)\cap \{
\abs\psi<1i\};$
\item the map $\rr_{W,s}^{-1}(\zeta)\to D_\zeta$ 
induced by $f$ 
is bijective. 
\end{itemize}
The set $\Theta$ is ind-$K$-definable, and it contains $\xi$ which is $\k(K))$-Zariski generic (indeed $c$ is $K\zeta$-definable and is
transcendental over $\k(K)$). Thus $\Theta$ contains a $\k(K)$-Zariski-dense open subset $C'$ of $C_{s,\widehat W}$ and we can assume
that $c_\zeta, A_\zeta$ and $a_\zeta$ depend $K$-definably and uniformly on the element $\zeta$ of $C'$. Set $V=\rr_{W,s}^{-1}(C')$. 
It follows from our construction than $(V,f)$ is a nice $K$-chart of type 2 and that $V\subset U$; and as $\xi \in C'$ by $\k(K)$-Zariski genericity, the point $a$
belongs to $V$, so $x\in \widehat V$.
\end{proof}

\section{Nash structure of curves: triangulations}
\label{nash-structure-triangulations}


We fix for the whole section a valued field $K$ and an algebraic closure
$\overline K$ of $K$, equipped with an extension of the valuation of $K$.

What will actually matter to us here is $K$-definability rather than
$K$-rationality, so it would be essentially harmless to 
replace $K$ by its definable closure, \ie,
to assume
that $K$
is perfect Henselian; but this would not make the proofs nor the statements
simpler. We denote by $G$ the automorphism group of
$\overline K$ over $K$ (as valued fields; so this is the decomposition group
of the chosen valuation on $\overline K$). 

\subsection{Split and non-split definable intervals}\label{split-intervals}
Let $I$ be a $K$-definable topological space which 
is $\overline K$-definably homeomorphic to a (generalized) open interval
$J$. As $G$ acts trivially on $\Gamma_0$, it acts trivially
on $J$ and so acts through a finite
quotient on $I$ (namely through $\mathrm{Aut}(L/K)$ where $L$ is any
finite Galois extension of $K$ inside $\overline K$ over which $I$
is homeomorphic to $J$). So if some $\sigma\in G$
preserves the two orientations of $I$ (or said otherwise, 
preserves its two branches at infinity), then $\sigma$ acts on $I$
as a monotonic torsion 
homeomorphism of an interval, so it acts trivially. 
As a consequence either $G$ acts trivially on $I$,
or it acts through
a finite quotient of order two whose non-trivial element acts as 
an order-reversing involution on $I$; such an involution has a unique
fixed point, which is a $K$-definable point of $I$ (and even the unique
$K$-definable point of $I$) and will be called its \emph{center}. 

\begin{defi}
Let $\Sigma$ be a $\Gamma$-internal one-dimensional 
$K$-definable topological space. A $K$-definable \emph{vertex set}
of $\Sigma$ is a finite $K$-definable subset $\mathscr V$ of
$\Sigma$ such that each connected component of
$\Sigma\setminus \mathscr V$ is $\overline K$-definably
homeomorphic to a (generalized) open interval. 
\end{defi}

\subsection{}
Let $\Sigma$ be a $\Gamma$-internal one-dimensional 
$K$-definable topological space and let
$\mathscr V$ be a $K$-definable vertex set of $\Sigma$. 

\subsubsection{}
If $\mathscr W$ is a $K$-definable subset of $\Sigma$ containing 
$\mathscr V$, it is immediate that $\mathscr W$ is also a 
$K$-definable vertex set of $\Sigma$. 

\subsubsection{}
Let $I\in \pi_0(\Sigma\setminus \mathscr V)$. It follows from 
\ref{split-intervals} that we are in one (and only one) of
these two cases: 

\begin{itemize}[label=$\diamond$]
\item the stabilizer of 
$I$ in $G$ acts trivially on $I$; this happens if and
only if $I$ is $L$-definably
homeomorphic to an open interval for every $K\subset L\subset \overline K$
such that $I$ is $L$-definable, or also if and only if $I$ is inter-definable with 
at least one of its two branches at infinity (it will then be inter-definable
with both); 
\item the stabilizer of 
$I$ in $G$ acts on $I$ through a quotient of order two, whose non-trivial element
acts on $I$ as an orientation-reversing involution having exactly one fixed
point $v_I$, inter-definable with $I$, which we call the center of $I$. 
\end{itemize}

\begin{enonce}[theorem]{Lemma-Definition}\label{definably-orientable}
Let $\Sigma$ be a $\Gamma$-internal one-dimensional 
$K$-definable topological space 
and let $\mathscr V$
be a $K$-definable vertex set 
of $\Sigma$. 

\begin{enumerate}[1]
\item The following are equivalent:

\begin{enumerate}[j]
\item there exists a $K$-definable family 
$(o_I)_{I\in \pi_0(\Sigma\setminus \mathscr V)}$
where each $o_I$ is an orientation of $I$;
\item for every $I\in \pi_0(\Sigma\setminus
\mathscr V)$ the two orientations of $I$
are (individually) interdefinable with $I$.
\end{enumerate}

When they are fulfilled we say that $(\Sigma, \mathscr V)$
is \emph{$K$-definably orientable}.

\item If $(\Sigma, \mathscr V)$ is $K$-definably orientable, so is
$(\Sigma, \mathscr W)$ for any finite $K$-definable set
$\mathscr W$ containing $\mathscr V$. 
\item In general, there always exists some finite $K$-definable
set $\mathscr W\supset \mathscr V$ such that 
$(\Sigma, \mathscr W)$ is $K$-definably orientable. 
\end{enumerate}
\end{enonce}

\begin{proof}
Let us prove (1). 
If (i) holds then for each $\sigma\in G$ one has
$\sigma(o_I)=o_{\sigma(I)}$, so if $\sigma(I)=I$ 
then $\sigma$ preserves $o_I$, whence 
(ii). 
Let us now assume that (ii) holds. Pick
a system of reprentatives $\mathscr J$ of the orbits of $\pi_0(\Sigma
\setminus \mathscr V)$ under $G$ and choose an arbitrary
orientation $o_J$ on $J$ for every $J\in \mathscr J$. Now if $I\in \pi_0(\Sigma
\setminus \mathscr V)$ there exist some (unique) $J\in \mathscr J$ and some
$\sigma \in \mathrm{Aut}(\overline K/K)$
such that $I=\sigma(J)$ 
and by (ii) the orientation $\sigma(o_J)$
only depends on $I$; we denote it by $o_I$, and the family $(o_I)$ is $K$-definable,
whence (i). 

Now let us prove (2). Assume
that $(\Sigma,\mathscr V)$ is $K$-definably orientable. 
If $J$ is a connected component of
$\Sigma\setminus \mathscr W$ it is contained in some connected component
$c(J)$ of $\Sigma\setminus \mathscr V$; then
if $(o_I)_I$ is as in (i), the family 
$(o_{c(J)}|_J)_{J\in \pi_0(\Sigma\setminus \mathscr W)}$ 
will be a $K$-definable system of orientations, so $(\Sigma, \mathscr W)$
is $K$-definably orientable.

Let us finally prove (3). Let $\mathscr I$ be the set of connected
components of $\Sigma\setminus \mathscr V$ 
which are not inter-definable with their orientations. 
Let $I\in \mathscr I$; for such an $I$, denote its center by $v_I$. 
Then $I\setminus \{v_I\}$ has two connected components, and
each of them is inter-definable with its two orientations 
(individually), for one of these orientations is the one pointing toward
$v_I$. 
Set $\mathscr W=\mathscr V\cup \{v_I\}_{I\in \mathscr I}$; then 
$\mathscr W$ is $K$-definable and $(\Sigma, \mathscr W)$ is $K$-definably orientable. 
\end{proof}

\begin{defi}\label{defi-strong-adm}
Let $K$ be a valued field
and let $\overline K$ be an algebraic closure of $K$ equipped with an extension 
of the valuation of $K$.  
Let $X$ be an algebraic $K$-curve and let $U$ be a $K$-definable subset of $X$. A
$K$-skeleton $\Sigma$
of $\widehat U$ is called \textit{Nash-admissible} if it is admissible
and if there exists a 
$K$-definable vertex
set $\mathscr V$ of $\Sigma$ containing
all simple points of $\Sigma$
and fulfilling the following properties:

\begin{enumerate}[1] 

\item the pair $(\Sigma, \mathscr V)$ is $K$-definably orientable;

\item for every connected component $I$ of $\Sigma\setminus V$ there exists a 
$\overline K$-regular function $f_I$ defined on a $\overline K$-Zariski
neighborhood of $\rho_{U,\Sigma}^{-1}(I)\cap U$ such that $(\rho_{U,\Sigma}^{-1}(I)\cap U,f_I)$ is a nice
$\overline K$-chart of type 3; 
\item one has $\rho_{U,\Sigma}^{-1}(v)=\{v\}$ for every $v\in \mathscr V\cap U$
(as $\mathscr V$ contains all simple points of $\Sigma$ this means that
$\Sigma$ contains 
all $U$-branches at is simple points, so $\rr_{U,\Sigma}$
and $\wrr_{U,\Sigma}$ make sense); 
\item for every vertex $v\in \mathscr V\setminus U$: 

\begin{enumerate}[b]
\item there exist a finite $\overline K$-definable subset $F_v$ of
$C_{v,
\widehat U\setminus \Sigma}$, 
a 
$\overline K$-regular function $f_v$ defined on a $\overline K$-Zariski
neighborhood of $\rr_{U,\Sigma}^{-1}(C_{v,\widehat U\setminus \Sigma}
\setminus F_v)\cap U$ such that $(\rr_{U,\Sigma}^{-1}(C_{v,\widehat U\setminus \Sigma}
\setminus F_v)
\cap U,f_v)$ is a nice
$\overline K$-chart of type 2 with 
vertex $v$; 
\item for every $\zeta\in F_v$, there exists 
a 
$\overline K$-regular function $f_\zeta$ defined on a $\overline K$-Zariski
neighborhood of $\rr_{U,\Sigma}^{-1}(\zeta)\cap U$ such that $(\rr^{-1}_{U,\Sigma}(\zeta)\cap U,f_\zeta)$ is a nice
$\overline K$-chart of type 1. 
\end{enumerate}
\end{enumerate}
\end{defi}

\begin{enonce}[remark]{Comments}\label{comments-rationality}
Let us insist on the rationality assumptions in the above 
definition. 
The skeleton $\Sigma$ is assumed to be $K$-definable, as well as the set $\mathscr V$, and
$(\Sigma, \mathscr V)$ is assumed to be $K$-definably orientable. 

All individual elements of
$\mathscr V$, all nice charts involved
and all connected components
of $\Sigma \setminus \mathscr V$ are a priori merely
$\overline K$-definable.

If $L$ is a valued extension of $K$, the property for a $K$-skeleton of
being Nash-admissible
as an $L$-skeleton is a priori weaker than being Nash-admissible as an $K$-skeleton; but it turns out
that both are in fact equivalent, see Theorem \ref{structure-theo}(2)
below. 
\end{enonce}

\subsection{}
Let $K$, $X$ and $U$ be as in the definition 
above. Let $\Sigma$
be a Nash-admissible $K$-skeleton of $\widehat U$ and let $\mathscr V$ be 
a $K$-definable vertex set of $\Sigma$ witnessing its Nash-admissibility. 
Let $\Omega$ be a connected component of
$\widehat U\setminus \mathscr V$. Taking into account the fact
that $\rho_{U,\Sigma}^{-1}(v)=\{v\}$ for every $v\in \mathscr V\setminus U$,
we see in view of \ref{complement-of-skeleton} that $\Omega$ can be of one of the two following kinds: 

\begin{enumerate}[a] 
\item $\Omega=\widehat \rho_{U,\Sigma}^{-1}(I)$ for some connected component $I$ of $\Sigma\setminus \mathscr V$; 
\item $\Omega=\wrr_{U,\Sigma}^{-1}(\zeta)$ for some $v\in \mathscr V\setminus U$ and 
$\zeta\in C_{v,\widehat U\setminus \Sigma}$. 
\end{enumerate}

In both cases $\Omega$
is open in $\widehat X$, and of the form
$\widehat V$ for $V$ a definable subset of
$U$ (definable over $\overline K$ in case (a) and over $\overline K\zeta$ in case (b)). 

Moreover in case (a) the function $f_I$ induces a $\overline K$-definable
Nash isomorphism
between $V$ and $A$ for some open annulus $A$ in $\A^1$; in particular 
$f$ identifies $I$ with a closed subset of $\widehat A$ homeomorphic
to an open interval, which is necessarily the only interval 
containing the two branches at infinity of $\widehat A$; therefore $\Omega$ is a tree with exactly two branches
at infinity and $I$ is the unique interval joining them, which we call the canonical skeleton 
of $\Omega$. 

In case (b) either the function $f_v$ or the function $f_\zeta$ (if $\zeta\in F_v$) induces
a definable Nash isomorphism
between $V$ and $D$ for some open $\overline K\zeta$-definable
disc $D$ in $\A^1$.
In particular, $\Omega$ is a tree with
exactly one branch at infinity. 

It follows that $\Sigma$ can be entirely reconstructed from $\mathscr V$: it is the union of
$\mathscr V$ and of the canonical skeletons of the finitely many connected components
of $\widehat U\setminus \mathscr V$ having two branches at infinity.

\begin{defi}
Let $K$ be a valued field. 
A \textit{$K$-triangulation}
of $\widehat U$ is a finite $K$-definable 
vertex set
of a Nash-admissible $K$-skeleton $\Sigma$ of $\widehat U$ that witnesses
the Nash-admissibility of $\Sigma$. By the above $\Sigma$ is then entirely 
determined by $\mathscr V$ and is called the skeleton of the triangulation $\mathscr V$. 

Any family of data $((f_I)_{I\in\pi_0(\Sigma
\setminus \mathscr V)}, (f_v, F_v,(f_\zeta)_{\zeta \in F_v})_{v\in \mathscr V'})$ as in Definition 
\ref{defi-strong-adm} 
(where $\mathscr V'$ denotes the set of non-simple points of
$\mathscr V$) will be called a \emph{$\mathscr V$-atlas} on $\widehat U$. 
\end{defi}

\begin{rema}
Let $\mathscr V$ be a finite $K$-definable subset of $\widehat U$ and let 
$\Sigma$ be a $K$-skeleton of $\widehat U$ containing 
$\mathscr V$. It follows from the definitions
(see also Comments \ref{comments-rationality})
that $\mathscr V$ is a
$K$-triangulation with skeleton $\Sigma$ if and only if 
$\mathscr V$ is a
$\overline K$-triangulation with skeleton $\Sigma$
and $(\Sigma, \mathscr V)$ is $K$-definably orientable. 
\end{rema}

\subsection{Inherited triangulation}\label{inherited-triangulation}
Let $\mathscr V$ be a triangulation of $\widehat U$, let $\Sigma$
be its skeleton, and let $((f_I)_I, (f_v,F_v, (f_\zeta)_\zeta)_v)$ be a $\mathscr V$-atlas
on $\widehat U$. 
Let $E$ be a 
$K$-definable 
subset of $\pi_0(\Sigma\setminus \mathscr V)$ and let 
$\mathscr W$ be a $K$-definable
subset of $\mathscr V$. Let $\Tau$ be the $K$-definable
subset 
$\mathscr W\cup \bigcup_{I\in E}I$ of $\Sigma$, and set $V=\rho_{U,\Sigma}^{-1}(\Tau)$; 
we will implicitly use the properties of $V$ and $\Tau$ stated in 
\ref{subskeletons}. 
It follows from the definitions that
$\pi_0(\Tau\setminus \mathscr W)=E$, 
that $\mathscr W$ is a triangulation of $\widehat V$
with skeleton $\Tau$ 
and that $((f_I)_{I\in E}, (f_v, F_v,(f_\zeta)_{\zeta\in F_v})_{v\in \mathscr W'})$
is a $\mathscr W$-atlas of $\widehat V$. 

\subsection{First examples: around the projective line}\label{example-triangulation-p1}
Let $\Sigma$ be the
$K$-skeleton $[0,\infty]=\{\eta_{0,r}\}_{0\leq r\leq\infty}$ of
$\widehat{\P^1}$. 
It is $K$-definably homeomorphic 
to an interval, is admissible and contains all $\P^1$-branches
at its simple points. 
The map $\rho_{\P^1, \Sigma}$ maps
$0$ to  $0$, $\infty$ to $\infty$, and any non-zero $a$ to
$\eta_{a,\abs a}$. 
As a consequence, for every 
definable interval $I$ contained in $[0,\infty]$
we have $\rho_{\P^1,\Sigma}^{-1}(I)=\{\abs T\in I\}$.

%

\subsubsection{}\label{triangulation-P1}
Let $E$ be a finite (possibly empty) subset
of 
$\abs{\overline K^\times}$. 
It follows from
Example \ref{sigmarv-projective-line}
that $\mathscr V:=\{\eta_{0,r}\}_{r
\in E}
\cup\{0,\infty\}$ is a $K$-triangulation
of $\widehat \P^1$ with skeleton $\Sigma$, and that 
$((T)_{I\in \pi_0(\Sigma\setminus \mathscr V)}, (T, \emptyset)_{r
\in E})$
is a $\mathscr V$-atlas of $\widehat \P^1$. 

\subsubsection{}\label{triangulation-annulus}
Let $R_1$ and $R_2$ be two elements of 
$\abs {\overline K}\cup\{\infty\}$
with $R_1\leq R_2$ and let $\Tau$ be an interval of $\Sigma$
with endpoints $\eta_{0,R_1}$ and $\eta_{0,R_2}$ (each of them
may belong to
$\Tau$ or  not). Let
$E$ be a finite subset of $\abs {\overline K}\cup\{\infty\}$
such that $\eta_{0,r}\in \Tau$ for all 
$r\in E$ and $R_1$ (\resp $R_2$) belongs
to $E$ if $\eta_{0,R_1}$ (\resp $\eta_{0,R_2}$)
belongs to 
$\Tau$. 
Set $\mathscr W=\{\eta_{0,r}\}_{r\in E}$, 
and let $\mathscr V$ 
be the set
$\{0,\infty\}
\cup\{\eta_{0,R_1},\eta_{0,R_2}\}\cup\mathscr W$. 
Then $\mathscr V$
is a $K$-triangulation of $\widehat \P^1$ with skeleton $\Sigma$, and $\Tau$ is the union of 
 $\mathscr W$ and of finitely many connected components of $\Sigma\setminus \mathscr W$. 
 So in view of 
 \ref{triangulation-P1} and of 
 \ref{inherited-triangulation}, $\mathscr W$ is a $K$-triangulation 
 of $\widehat \rho_{\P^1,\Sigma}^{-1}(\Tau)$
 with skeleton $\Tau$; and $((T)_{I\in 
 \pi_0(\Tau\setminus \mathscr W)},
 (T, \emptyset)_{r\in E\cap \abs{\overline K^\times}})$
 is a $\mathscr W$-atlas of $\widehat \rho_{\P^1,\Sigma}^{-1}(J)$. 

\subsubsection{}\label{triangulation-disc}
We keep the notation of \ref{triangulation-annulus}
but we assume that $R_1=0$, that $\eta_{0,0}=0$ belongs to $\Tau$ and that $E$
contains at least one element of $\abs{\overline K^\times}$; we denote by $e$ 
the smallest element of $E\cap \abs{\overline K^\times}$, and by $V$ the pre-image
$\rho_{\P^1,\Sigma}^{-1}(\Tau)$. Not that by our assumptions
$V$ is either
an open or closed disc with radius in $\abs{\overline
K^\times}$ or the affine line or the projective
line. 
Let $\Tau'$ be the intersection of $\Tau$ and of $\{\eta_{0,r}\}_{r\geq e}$. Then $\Tau'$
is a sub-interval of $\Tau$ and there is a $K$-definable strong deformation retraction from
$\Tau$ to $\Tau'$, so $\Tau'$ is a $K$-admissible skeleton of $\widehat V$. Moreover
for every $v\in \widehat V$ we are by construction in one of the following two cases: 
\begin{itemize}[label=$\diamond$]
\item if $\widehat \rho_{V,\Tau}(v)\in \Tau'$ then $\widehat \rho_{V,\Tau'}(v)
=\widehat \rho_{V,\Tau}(v)$; 
\item  if $\widehat \rho_{V,\Tau}(v)\notin \Tau'$ then 
$\widehat \rho_{V,\Tau'}(v)=\eta_{0,e}$ and 
$[v,\eta_{0,e}]=[v,\widehat \rho_{V,T}(v)]\cup[\widehat \rho_{V,T}(v),\eta_{0,e}]$. 
\end{itemize}
We thus see that $\rho_{V,\Tau'}^{-1}(\eta_{0,e})$ is the closed disc of
radius $e$. 

It follows that $\mathscr W\setminus \{0\}$ is still a triangulation of $\widehat V$, 
with skeleton $\Tau'$, and that 
$((T)_{I\in \pi_0(\Tau'\setminus \mathscr W)}, 
(T, \emptyset)_{r\in E\cap \abs{\overline K^\times}})$ is a ($\mathscr W\setminus \{0\}$)-atlas
of $\widehat V$. 

\subsection{Triangulations of charts}\label{triangulation-chart}
Let
$(U,f)$ be a nice
$K$-chart on some $K$-algebraic curve.
We are going to exhibit $K$-triangulations on 
$U$. The simplest case will be that of type (2), where the construction will directly
follow from the definition; the cases of type (1) and (3) will be dealt with
using the Nash isomorphism $f$ to pull-back the triangulations 
on discs and annuli built above. 

\subsubsection{The case where $(U,f)$ is of type 2}\label{triangulation-chart2}
Let $v$ be the vertex
of the chart. It follows directly from the definition that $\{v\}$ is a $K$-triangulation 
of $\widehat U$
with skeleton $\{v\}$
and that $(f,\emptyset)$ is an $\{v\}$-atlas of $\widehat U$.

\subsubsection{The case where $(U,f)$ is a
of type 1}\label{triangulation-chart1}
Set $D=f(U)$. 
Let $r$  be any element of 
$\abs {\overline K}$ strictly smaller than 
the radius $\rho$
of $D$ 
and let $x$ be the unique pre-image
of $\eta_{0,r}$ in $\widehat U$ under $f$. 
It follows from \ref{triangulation-disc}, which we
apply
with
$R_1=0, R_2=1, \eta_{0,\rho}\notin \Tau$
and $E=\{r\}$, that $\{\eta_{0,r}\}$ is a $K$-triangulation of $\widehat D$
with skeleton $[\eta_{0,r}, \eta_{0,1})$, and that $\{T\}$ 
(\resp $(T,(T, \emptyset))$)  
is an $\{\eta_{0,r}\}$-atlas of $\widehat D$ if 
$r=0$ (\resp if $r\neq 0$). So $\{x\}$ is a $K$-triangulation 
of $\widehat U$ with skeleton the pre-image of
$[\eta_{0,r}, \eta_{0,1})$, which is the semi-open interval
of $\widehat U$ starting at $x$ and defining the unique branch at infinity of 
$\widehat U$; if $r=0$ (that is, if $x$
is a simple point) then $\{f\}$ is an $\{x\}$-atlas of $\widehat U$; 
if not, then $(f, (f, \emptyset))$ is an $\{x\}$-atlas of $\widehat U$.

\subsubsection{The case where $(U,f)$ is
of type 3}\label{triangulation-chart3}
Let $E$ be any finite subset (possibly empty) of elements
of $\abs{\overline K^\times}\cap (R_1,R_2)$, where $R_1$
and $R_2$ are the two radii of the annulus $f(A)$. 
It follows from \ref{triangulation-annulus}
(applied assuming that neither $\eta_{0,R_1}$ nor  $\eta_{0,R_2}$ 
lies on $\Tau$) that
$\mathscr V:=\{\eta_{0,r}\}_{r\in E}$ is a $K$-triangulation of $\widehat A$
with skeleton the canonical skeleton $\Sigma$ of $\widehat A$, and that 
$((T)_{I\in \pi_0(\Sigma\setminus \mathscr V)}, (T, \emptyset)_{r\in E})$ is a $\mathscr V$-atlas of
$\widehat A$.
So $\mathscr W:=f^{-1}(\mathscr V)$ is a $K$-triangulation 
of $\widehat U$ with skeleton the canonical skeleton $\Sigma'$ of $\widehat U$, and
$((f)_{I\in \pi_0(\Sigma'\setminus \mathscr W)}, (f,\emptyset)_{w\in \mathscr W})$
is a $\mathscr W$-atlas of $\widehat U$.

\begin{theo}\label{properties-triangulation}
Let $U$ be a $K$-definable subset of a $K$-algebraic curve 
$X$, and let $\mathscr V$ be a $K$-triangulation 
of $\widehat U$ with skeleton $\Sigma$. 
Let  $((f_I)_I, (f_v, F_v, (f_\zeta)_{\zeta \in F_v})_v)$ be a $\mathscr 
V$-atlas on $\widehat U$ and let $E$ be the finite set consisting 
of all $f_I, f_v, f_\zeta$. 

\begin{enumerate}[1]
\item The set $\Sigma_{\RV,U}$ is $\RV$-internal. More precisely there exist a 
$\overline K$-definable topological embedding 
$\Sigma\hookrightarrow [0,\infty]^m$ and a $\overline K$-definable
embedding $\Sigma_{\RV,U}\hookrightarrow \RV^m$ such that the diagram
\[\begin{tikzcd}
\Sigma_{\RV,U}\ar[d,"\sigma"']\ar[r,hook]&\RV^m\ar[d]\\
\Sigma\ar[r,hook]&{[0,\infty]^m}\end{tikzcd}\]
commutes. 
\item Let 
$\zeta\in \Sigma_\RV$. 
There exists 
$f\in E$
defined on a $\overline K$-Zariski
neighborhood of 
the connected component $\wrr^{-1}_{U,\Sigma}(\zeta)$
of $\widehat U\setminus \Sigma$ 
that induces a Nash isomorphism
between $\wrr^{-1}_{U,\Sigma}(\zeta)$
and a
$\overline K\zeta$-definable open disc of $\A^1$.

\item The map $\rho_{U,\Sigma}$ is topologically proper; otherwise said, $\Sigma$
contains all branches at infinity of $\widehat U$.

\item Any finite $K$-definable subset of 
$\Sigma$ containing $\mathscr V$ is still a
$K$-triangulation of $\widehat U$ with skeleton $\Sigma$. 
\item Let $\Tau$ be an
admissible skeleton containing
$\Sigma$ and let $\mathscr W$ be the union
of $\mathscr V$ and of the set of topological
nodes of $\Tau$, \ie, those points that do not have a neighborhood
in $\Tau$ homeomorphic to an open interval.
Then $\Tau$ is Nash-admissible and $\mathscr W$ is a 
$K$-triangulation of $\Tau$.
\end{enumerate}
\end{theo}

\begin{proof}
We prove each statement separately.

\subsubsection{Proof of (1)}
Let $I$ be a connected component of $\Sigma\setminus \mathscr V$
and let $R_1$ and $R_2$ be the infimum and the supremum of $\abs {f_I}$ on $I$. 
The
map $f_I$ induces a Nash isomorphism between $\rho_{U,\Sigma}^{-1}(I)$
and the open annulus $R_1<\abs T <R_2$; it follows then from 
Example \ref{sigmarv-projective-line}
that we have a commutative diagram
\[\begin{tikzcd}
\sigma^{-1}(I)\ar[r,"\sim"]\ar[d,"\sigma"']&\{\rv(\lambda)\}_{R_1<\abs \lambda<R_2}\ar[d]\\
I\ar[r,"\sim", "\abs{f_I}"']&(R_1,R_2)\end{tikzcd}\]
with $\overline K$-definable arrows, the bottom map being
a homeomorphism. Let $v$ be an element of $\mathscr V$ belonging to $\partial I$; 
so at least one branch of $I$ has limit $v$. 
If $v$ is not simple then $f_I$ is defined at $v$ and $\abs{f_I(v)}=R_i$ for a necessarily unique $i\in \{1,2\}$
(which shows that only one branch at infinity of $I$ has limit $v$: $\overline I$ is not a circle). If $v$ is simple then 
the limit of $\abs {f_I}$ along any branch at infinity of $I$ having limit $v$ necessarily belongs to
$\{0,\infty\}$ (if not, $\abs{f_I}$ would
be constant along $J$, for it would be defined 
and invertible at the pre-image $v'$ of $v$ on the normalization of $X$ defined by the branch $J$). If follows that 
$\abs {f_I}$ induces a homeomorphism between $\overline I$ and either some interval 
stuck between $(R_1,R_2)$ and $[R_1,R_2]$, 
or the quotient $[0,\infty]/\!\sim$ where $\sim$ is the identification between $0$ and $\infty$; the last case
occurs  if and only if $R_1=0, R_2=\infty$ and $\overline I$ is a circle, the two branches at infinity of $I$ having
the same limit $v\in \mathscr V$, which is necessarily a singular point of $X$,
belonging to the indeterminacy locus of the rational function 
$f_I$. We remark that 
the mapping $t$ to $(t,t)$ if $t\leq 1$ and to $(t^{-1}, t^{-2})$ otherwise defines a $\emptyset$-definable
embedding from $[0,\infty]/\!\sim$ into $[0,\infty]^2$. So (by lifiting if necessary the preceding embedding
to the $\RV$-sort) 
we eventually get a commutative diagram 
\[\begin{tikzcd}
\sigma^{-1}(I)\ar[d,"\sigma"']\ar[r,hook]&\RV^m\ar[d]\\
\overline I\ar[r,hook]&{[0,\infty]^m}\end{tikzcd}\]
for $m\in\{1,2\}$, where all maps are $\overline K$-definable, and the bottom
map is a topological embedding. Since $\abs {f_I(v)}=\abs{f_J(v)}$ modulo $\abs{\overline K^\times}$ for all
non-simple $v\in \mathscr V$ and all pairs $(I,J)$ of connected components of $\Sigma\setminus \mathscr V$
having $v$ as a limit point (because $v$ is $\overline K$-definable) one can glue the above diagrams
for $I\in \pi_0(\Sigma\setminus \mathscr V)$ by making each of them live in some $\overline K$-definable 
``affine coordinate subspace" and get a commutative diagram 
\[\begin{tikzcd}
\Sigma_{\RV, U}\setminus \sigma^{-1}(\mathscr V)\ar[d,"\sigma"']\ar[r,hook]&\RV^m\ar[d]\\
\Sigma\ar[r,hook]&{[0,\infty]^m}\end{tikzcd}\]
where all maps are $\overline K$ definable, where the bottom map is a topological embedding, and where $m$
is an integer. The existence of a commutative diagram like required by statement (1) then follows from the fact that
for each $v\in \mathscr V$ the pre-image $\sigma^{-1}(v)$ is the $\RV$-internal (it is even $\k$-internal) curve
$C_{v,{\widehat U\setminus \Sigma}}$ (one may have to slightly enlarge the integer $m$). 

\subsubsection{Proof of (2) and (3)}
Set $v=\sigma(\zeta)$; this is also the constant value
of $\widehat \rho_{U,\Sigma}$ on the component $\wrr_{U,\Sigma}^{-1}(\zeta)$. 
If $v$ does not belong to $\mathscr V$ then it lies on some connected
component $I$ of $\Sigma \setminus \mathscr V$, and $\wrr_{U,\Sigma}^{-1}(\zeta)$ is a connected component
of $\widehat \rho_{U,\Sigma}^{-1}(I)\setminus I$; it follows then
from Example \ref{sigmarv-projective-line}
that $f_I$ induces a Nash isomorphism between $\wrr_{U,\Sigma}^{-1}(\zeta)$
and a $\overline K\zeta$-definable open disc. 
Assume that $v$ belong to $\mathscr V$. Then 
$\zeta\in C_{v,\widehat U\setminus \Sigma}$. If $\zeta \in F_v$, \resp
if $\zeta \notin F_v$, the map
$f_\zeta$, \resp $f_v$,  induces a $\overline K$-definable Nash isomorphism between $\wrr_{U,\Sigma}^{-1}(\zeta)$ and an
$\overline K\zeta$-definable 
open disc,
ending the proof of (2). 

In particular, each connected component of $\widehat U\setminus \Sigma$ is a tree with exactly one branch at infinity,
which implies (3) (see the discussion
in \ref{complement-of-skeleton}). 

\subsubsection{Proof of (4)}
Let $\mathscr W$ be a finite $K$-definable subset of $\Sigma$ containing
$\mathscr V$. Then $\mathscr W$ is a $K$-definable vertex set of $\Sigma$, and $(\Sigma, \mathscr W)$
is $K$-definably orientable. For showing that $\mathscr W$ is a $K$-triangulation of $\widehat U$, it therefore
suffices to show that it is a
$\overline K$-triangulation of $\widehat U$, so we can assume $K=\overline K$. Then we can argue by induction on the cardinality
of $\mathscr W\setminus \mathscr V$, reducing to the case where $\mathscr W=\mathscr V\cup\{w\}$
for some $w\in \Sigma\setminus \mathscr V$. Let $J$ be the connected component of
$\Sigma\setminus \mathscr V$ that contains $w$. 
It then
follows from \ref{triangulation-chart3}
that $\mathscr W$ is a $K$-triangulation of $\widehat U$ with
skeleton $\Sigma$, which admits an atlas consisting of the family
$((f_I)_{I\neq J}, (f_v, F_v, (f_\zeta)_{\zeta \in F_v})_v)$ together with 
one copy of $f_J$ for each connected component of $J\setminus \{w\}$
and for the vertex $w$, and with the empty
set of points of the curve $C_{w,\widehat U\setminus \Sigma}$.

\subsubsection{Proof of (5)}
The set $\mathscr W$ is $K$-definable.
Let $(o_I)_{I\in \pi_0(\Sigma
\setminus \mathscr V)}$ be a $K$-definable
system of
orientations. 

If $\Upsilon$ is a connected component of $\Tau\setminus \Sigma$
then $\partial \Upsilon$ consists of one point $w\in \Sigma$ and $\overline \Upsilon$ is a (compact) tree;
note that if $w\notin \mathscr V$ then $w$
is necessarily  a topological node
of $\Tau$, so $w\in \mathscr W$ in any case. Then every 
connected component $J$ of $\overline \Upsilon\setminus \mathscr W$ is $\overline K$-definably homeomorphic to
an open interval, and for such a $J$ we denote by $o_J$ its orientation pointing toward $w$
in the tree $\overline \Upsilon$. 

If $J$ is a connected component of $\Tau\setminus \mathscr W$ then either $J$ is
a connected component of $\Upsilon\setminus \mathscr W$
for some connected component 
$\Upsilon$ of $\Tau \setminus \Sigma$, either it is
contained in some connected component $I$ of $\Sigma\setminus \mathscr V$, in which 
case we set $o_J=o_I|_J$. 
By construction the system 
of orientations $(o_J)_{J\in \pi_0(\Tau\setminus
\mathscr W)}$ is $K$-definable,
so $\mathscr W$ is a $K$-definable vertex set of $\Tau$
and $(\Tau, \mathscr W)$ is $K$-definably orientable. It thus
suffices to prove that 
$\mathscr W$ is a $\overline K$-definable triangulation of $\Tau$, so we can now
assume that $K=\overline K$. 

As $\Tau$ is $K$-admissible,
it can be built from $\Sigma$ by adding successively $K$-definable
topological edges
(the assumption that $K$ is algebraically closed is used to 
ensure their $K$-definability). Arguing
by induction on the number of edges
we can assume that $\Tau$ is equal to 
$\Sigma\cup[x,w]$ where $x$ is some $K$-definable point
of $\widehat U\setminus \Sigma$ and $w$ is a $K$-definable point of
$\Sigma$. By (4) $\mathscr V\cup\{w\}$ is still a $K$-triangulation of
$\widehat U$ with skeleton $\Sigma$, so we can assume that $w\in \mathscr V$.
Then $\mathscr W=\mathscr V\cup\{x\}$.

Set
$\xi=\rr_{U,\Sigma}(x)$; we can see it as a 
point of $C_{w,\widehat U\setminus \Sigma}$.
Set $V=\rr_{U,\Sigma}^{-1}(\xi)$; this is a $K$-definable
subset of $U$, and
$\widehat V=\widehat \rho_{U,\Tau}^{-1}([x,w))$ in view of
Lemma \ref{skeleton-disc}. 

Set $g=f_w$ if $\xi\notin F_w$ and
$g=f_\xi$ otherwise. Then $g$ induces a Nash isomorphism
between $V$ and a $K$-definable open disc 
 of $\widehat \A^1$, say of radius $r$. 
 The  point 
$g(x)$
is equal to $\eta_{a,s}$ for some $a\in K$
and some $s\in \abs K\cap [0,r)$. Set $h=g-a$; by construction $(V,h)$
is a nice $K$-chart of type 1 
and 
$h(x)=\eta_{0,s}$
It follows then
from \ref{triangulation-chart1} 
that $\{x\}$ is a $K$-triangulation of $\widehat V$ with skeleton 
$[x,v)$, that $\{h\}$ is an $\{x\}$-atlas
of $\widehat V$ if $x$ is simple, and that $(h,(h,\emptyset))$ is an 
$\{x\}$-atlas of $\widehat V$ if
$x$ is not simple.

We deduce from the above that $\mathscr W=\{x\}\cup \mathscr V$ is a $K$-triangulation of
$\widehat U$ with skeleton $\Tau$, and that
a $\mathscr W$-atlas of $\widehat U$ can be built as follows: 
one concatenates 
the family $((f_I)_{I\in \pi_0(\Sigma\setminus \mathscr V) }, 
(f_v, F_v, (f_\zeta)_{\zeta\in F_v})_{v\in \mathscr V\setminus \{w\}})$, 
with $(f_w, F_w\setminus \{\xi\}, (f_\zeta)_{\zeta\in F_w\setminus \{\xi\}})$
for the vertex $w$, 
the function $h$ for the connected component $(x,w)$
of $\Tau\setminus \mathscr W$ and 

\begin{itemize}[label=$\diamond$] 
\item if $x$ is simple (\ie, $s=0$), nothing else; 
\item if $x$ is not simple (\ie, $s\neq 0$),  the function $h$ for the vertex $x$, 
and the empty set of points of the residue curve $C_{x, \widehat V\setminus \Tau}$.\qedhere
\end{itemize}
\end{proof}

\begin{lemm}\label{glue}
Let $U$ be a $K$-definable subset of an algebraic $K$-curve $X$. 
Let $\Sigma$ be an admissible $K$-skeleton of $\widehat U$, and let $(\Sigma_i)$
be a finite family of
$K$-definable subsets of $\Sigma$ covering $\Sigma$. For each $i$, set $U_i=\rho_{U,\Sigma}^{-1}
(\Sigma_i)$, and assume that $\Sigma_i$ is a 
Nash-admissible $K$-skeleton of $\widehat U_i$. 
Choose for every $i$ a $K$-triangulation $\mathscr V_i$ of $\widehat U_i$ with skeleton $\Sigma_i$.
Let $\mathscr V$ be the union of the $\mathscr V_i$ and of the $\partial \widehat U_i$
(note that $\partial \widehat U_i$ 
coincides with $\partial \Sigma_i\subset \Sigma$ by Corollary
\ref{coro-boundary-uhat}). 
Then $\Sigma$ is 
Nash-admissible, and $\mathscr V$ is a $K$-triangulation of $\widehat U$ with skeleton $\Sigma$.
\end{lemm}

\begin{proof}
By Theorem \ref{properties-triangulation} (4) the union 
$\mathscr V_i\cup (\mathscr V\cap \Sigma_i)$ is for every $i$ still a $K$-triangulation with skeleton $\Sigma_i$,
so we can assume that $\mathscr V_i=\mathscr V\cap \Sigma_i$ for all $i$. 
Choose for every $i$
 a $\mathscr V_i$-atlas $\mathscr A_i$ of $\widehat U_i$.

Let $I$ be a connected component of $\Sigma \setminus \mathscr V$. As $\mathscr V$ contains 
$\partial  \Sigma_i$ for all $i$
there exists a (non necessarily unique) $i$ with 
$I\subset \Sigma_i$, so $I$ is a connected component of $\Sigma_i\setminus \mathscr V_i$; 
as a consequence $I$ is 
$\overline K$-definably homeomorphic to an open interval, 
and it is inter-definable with each of its orientations. 
Let $f_I$ be the function associated to $I$ by the atlas $\mathscr A_i$. Then 
as $\rho_{U,\Sigma}^{-1}(I)=\rho_{U_i,\Sigma_i}^{-1}(I)$, the pair
$(\rho_{U, \Sigma}^{-1}(I), f_I)$ is a nice $\overline K$-chart of type 3.

Let $v$ be a simple point of $\mathscr V$. Then $v$ belongs to some $\mathscr V_i$, 
and we then have
$\rho_{U,\Sigma}^{-1}(v)=\rho_{U_i,\Sigma_i}^{-1}(v)=\{v\}$. 

Now let $v$ be a non-simple point of $\mathscr V$.
Choose $i$ so that $v$ belongs to $\widehat U_i$. 
Let $(f_v ,F_v, (f_\zeta)_{\zeta \in F_v})$ be the data associated with $v$ 
provided by the atlas $\mathscr A_i$. Let $F'_v$  be the union of $F_v$ and of the
finite set of all $\zeta\in C_{v,\widehat U\setminus \Sigma}$ that do not belong to 
$C_{v,\widehat U_i\setminus \Sigma_i}$. 
Then $\rr_{U,\Sigma}^{-1}(C_{v,\widehat U\setminus \Sigma}\setminus F'_v)
=\rr_{U_i,\Sigma_i}^{-1}(C_{v,\widehat U_i\setminus \Sigma}\setminus F_v)$, 
so $(
\rr_{U,\Sigma}^{-1}(C_{v,\widehat U\setminus \Sigma}\setminus F'_v), f_v)$
is a nice $\overline K$-chart of type 2. Now let $\zeta \in F'_v$. 
Then $\wrr^{-1}_{U,\Sigma}(\zeta)$ does not meet the union of the $\partial \widehat U_j$, 
so there exists $j$ such that $\wrr^{-1}_{U,\Sigma}(\zeta)$ is entirely contained 
in $\widehat U_j$, so it coincides with $\wrr^{-1}_{U_j,\Sigma_j}(\zeta)$. 
Now if $i=j$ then $\zeta \in F_v$ so  $(\wrr^{-1}_{U,\Sigma}(\zeta),f_\zeta))$
is a nice
$\overline K$-chart of type 1. Assume now that $j\neq i$ and let 
$(g_v ,G_v, (g_\xi)_{\xi \in G_v})$ be the data associated with $v$ 
provided by the atlas $\mathscr A_j$. If $\zeta\in G_v$ then
$(\wrr^{-1}_{U,\Sigma}(\zeta),g_\zeta))$ is a 
nice $\overline K$-chart of type 1. 
Assume now that $\zeta\notin G_v$.
As $F'_v$ is finite  $\rr_{U,\Sigma}^{-1}(\zeta)$
is $\overline K$-definable, so it contains a $\overline K$-point $u$
(this is obvious if $K$ is non-tivially valued; if $K$ is trivially valued then as $v$
is $\overline K$-definable the image $g_v(v)$ is necessarily the unique $\overline K$-definable, 
non-simple point of $\A^1$, namely $\eta_{0,1}$, and $g_v(\rr_{U,\Sigma}^{-1}(\zeta))$
is then an open disc of radius 1 centered at $\overline K$-point, whence the claim).
And
now $(\wrr^{-1}_{U,\Sigma}(\zeta), g_v-g_v(u))$ 
is a nice $\overline K$-chart of type 1, and we are done. 
\end{proof}

\begin{theo}[Triangulations of algebraic curves]\label{structure-theo}
Let $K$ be a valued field
and let $X$ be a generically smooth
algebraic $K$-curve. Let $U$ be a $K$-definable subset of $X$.

\begin{enumerate}[1] 
\item 
The stable completion $\widehat U$ possesses a
Nash-admissible $K$-skeleton
(or, equivalently, a $K$-triangulation); moreover this Nash-admissible $K$-skeleton
can be taken  to contain any prescribed
$K$-skeleton of $\widehat U$.

\item Let $L$ be a valued extension of $K$. 
A $K$-skeleton $\Sigma$
of $\widehat U$ is Nash-admissible 
as a $K$-skeleton if and only if it is
Nash-admissible as an $L$-skeleton. 
\end{enumerate}
\end{theo}

\begin{proof}
We proceed in several steps. 

\subsubsection{}\label{nash-admissible-algclos}
We first assume that $K=\overline K$ and we are
going to prove that 
there exists a  Nash-admissible $K$-skeleton on $\widehat U$.

\paragraph{Modification of the problem}
Let $X'$ be the normalization of $X$, let $S$ be the singular locus of $X$
(which consists of finitely many points by our generic 
smoothness assumption) and let $U'$
and $S'$ be the pre-images of
 $U$ and $S$ on $X'$. Let $U''$ be the set obtained from $U'$ by removing 
 the finite set $E$ of 
 isolated points of $\widehat{U'}$. 
 Assume that we have proven the existence of 
 Nash-admissible
 $K$-skeleton $\Sigma''$ of $\widehat {U''}$ containing $U''\cap S'$. Then $\Sigma':=\Sigma''\cup E$ is a Nash-admissible $K$-skeleton 
 of $\widehat {U'}$ containing $U'\cap S'$. Since $\Sigma'$ is
 Nash-admissible, any admissible 
 $K$-homotopy homotopy
 of $\widehat {U'}$ with target $\Sigma'$ fixes all simple points of $\Sigma'$, and in particular
 all points of $\Sigma'\cap S'$.
 It follows that this homotopy descends to 
 an admissible $K$-homotopy
 of $\widehat U$ with target the image $\Sigma$ of $\Sigma'$. So $\Sigma$
 is an admissible $K$-skeleton of $\widehat U$
 containing $S$ and as $X'\to X$ is an isomorphism outside $S$
 the Nash-admissibility of $\Sigma'$ implies that of $\Sigma$. 
 
 We can thus assume from now on that $X$ is smooth and
 that $\widehat U$ has no isolated
 simple points, and we want to
 prove the existence of a Nash-admissible $K$-skeleton $\Sigma$ of $\widehat U$ containing
 a prescribed finite set $S$ of simple points.

 \paragraph{}
 By Theorem \ref{theo-core} there exists a finite family $(U_i, f_i)$ of nice $K$-charts of
 $\widehat U$ such that $\widehat U=\bigcup \widehat {U_i}$. 
 By \ref{triangulation-chart} et sq., there exists for every $i$ a Nash-admissible
 $K$-skeleton $\Sigma_i$
 of $\widehat U_i$. Now by \ref{skeleton-adapted}
 there exists an admissible $K$-skeleton $\Tau$ of 
 $\widehat U$ which is adapted to every $U_i$, contains $S$ and contains each $\Sigma_i$.
 Fix an index $i$. As $\Tau$ is adapted to $U_i$, $\Tau\cap \widehat U_i$ is an
 admissible
 $K$-skeleton of $\widehat U_i$; as it contains $\Sigma_i$, it is Nash-admissible by
 Theorem \ref{properties-triangulation} (5). 
 It now follows from Lemma
 \ref{glue}
 that $\Tau$ is a Nash-admissible $K$-skeleton of $\widehat U$.

 \subsubsection{Proof of (2) for $L=\overline K$}
 If $\Sigma$ is Nash admissible as a $K$-skeleton, it is obviously
 Nash-admissible as a $\overline K$-skeleton. 
  Let us assume conversely that $\Sigma$ is Nash-admissible 
 as a $\overline K$-skeleton. Let $\mathscr V$ be a $\overline K$-triangulation of
 $\widehat U$ with skeleton $\Sigma$. Let $\mathscr V'$ denotes the union of all conjugates
 of $\mathscr V$ under $G$. Then $\mathscr V'$ is a finite $K$-definable subset of $\Sigma$; 
 as it contains $\mathscr V$, this is still a vertex set of $\Sigma$. 
 By Lemma-Definition \ref{definably-orientable} there exists a
 finite $K$-definable subset $\mathscr V''$ of $\Sigma$ containing $\mathscr V'$
 such that $(\Sigma, \mathscr V'')$ is $K$-definably orientable.
 As it contains $\mathscr V$, the set
 $\mathscr V''$ is 
 in view of Theorem \ref{properties-triangulation} (3) a $\overline K$-triangulation of $\widehat U$ with skeleton $\Sigma$; 
 together with the fact that
 $(\Sigma, \mathscr V'')$ is $K$-definably orientable, this
 implies that
 $\mathscr V''$ is a $K$-triangulation of $\widehat U$
 with skeleton $\Sigma$. 
 
 \subsubsection{Proof of (1)}
 Let $\Sigma$ be a $K$-skeleton of $\widehat U$. 
 By the particular case of (1) handled in 
 \ref{nash-admissible-algclos}, there exists 
 a Nash-admissible $\overline K$-skeleton $\Tau$ on $\widehat U$. 
 Let $\Tau'$ be the union of all conjugates of $\Tau$ under $G$. Then $\Tau'$
 is a $K$-definable skeleton of $\widehat U$, and so is $\Tau'\cup \Sigma$. 
 The latter is therefore contained in some admissible
 $K$-skeleton $\Upsilon$. As $\Upsilon$ is admissible and contains $\Tau$, it is
 Nash-admissible as a $\overline K$-skeleton by Theorem \ref{properties-triangulation} (5), hence
 it is Nash-admissible as a $K$-skeleton
 by statement (2)
 already proved for $L=\overline K$. 
 
 \subsubsection{Proof of (2) in general}
 As (2) holds for $L=\overline K$, it suffices to prove that $\Sigma$ is Nash-admissible as a
 $\overline K$-skeleton if and only if it is  Nash-admissible an $\overline L$-skeleton (where $\overline L$
 is a valued algebraic closure of $L$ containing $\overline K$), so we can assume that both $K$ and $L$
 are algebraically closed. 
 
 \paragraph{}
 If the valuation of $K$ is non-trivial, $K$ and $L$ are models of $\acvf$, and the result
 follows by model-completeness, since the set of data that define a triangulation is clearly ind-definable (one has to choose 
 the set $\mathscr V$ and then the collection of nice charts). 
 
 \paragraph{}It remains to handle the case where $K$ is trivially valued and we will only sketch it; details
 are left to the reader. First of all, 
 let us assume that $X$ is projective and smooth. Then $\widehat X(K)=X(K)\cup\{\eta\}$ where $\eta$ is its Gauss type --
 corresponding to the trivial valuation on $K(X)$. Using several $K$-definable finite maps $X\to \P^1$ (sufficiently many so that
 each $K$-points of $X$ lies in the étale locus of one of them) one shows essentially by the same kind of arguments as in 
 \ref{implicit-function-theorem} that $\{\eta\}$ is a $K$-triangulation of $\widehat X$ with skeleton $\{\eta\}$ and that each
 $K$-definable connected component of $\widehat X\setminus \{\eta\}$ contains exactly one $K$-point of $X$. 
It easily follows that any closed $K$-skeleton of $\widehat X$ either is
a finite union of $K$-points or contains $\eta$. It also
follows that if $U$ is not a finite subset of $X$, then $\widehat U$ contains $\eta$ and all
connected components of $\widehat X\setminus \eta$ except possibly finitely many of them $\Omega_1,\ldots, \Omega_r$ 
which are $K$-definable; and for each $i$ the intersection $\widehat U\cap \Omega_i$
can only be one of those three sets: $\emptyset$, $\{\omega_i\}$ or $\Omega_i\setminus \{\omega_i\}$, where $\omega_i$
is the unique $K$-point of $X$ lying in $\Omega_i$. 

Now if we do not assume anymore that $X$
is projective nor smooth we can use the above, together with standard arguments involving
a compactification of $X$ and its normalization, that a $K$-skeleton $\Sigma$ of $\widehat U$ is Nash-admissible 
as an $L$-skeleton if and only if the following hold: 
the intersection of $\Sigma$ with every connected component of $\widehat U$ is
non-empty and connected; the skeleton $\Sigma$ contains all singular points of $X$ lying on $U$ as well as all branches at infinity
of $\widehat U$;  the skeleton $\Sigma$ contains all $U$-branches at its simple points
(if these conditons are fulfilled, $\Sigma \cap \widehat X(K)$ is an $L$-triangulation of $\widehat U$ with skeleton $\Sigma$). 
But these conditions do not involve the field $L$, which ends the proof. 
 \end{proof}

\section{Tame henselian rationality and other applications}\label{applications}

%

\begin{theo}\label{definability}
Let $K$ be any valued field. Let 
$U$ be a $K$-definable subset of an algebraic $K$-curve $X$. 

\begin{enumerate}[1]
\item The pro-$K$-definable sets $\breve U$ and $\widetilde U$ are $K$-definable. 
\item For every closed $K$-skeleton $\Sigma$ of $\widehat U$, the 
$K$-definable set of connected components
of $\widehat U\setminus \Sigma$ (Proposition \ref{prop-connected-comp-definable} (2)) 
is $\RV$-internal. 
\item The set of non-simple points of $\widehat U$ whose residue curve has positive genus is finite and $K$-definable.
\end{enumerate}
\end{theo}

\begin{proof}
We prove each statement separately.

\subsubsection{Proof of (1)}
Let $Z$ be the normal projective
completion of the normalization of $X$. 
There is a natural pro-$K$-definable
isomorphism between $Z$ and $\breve Z\setminus \widetilde Z$, wich maps
a point $z$ to the type induced by the discrete valuation on $Z$ centered at $z$. 
Now $\breve U\setminus \widetilde U$ is cut out from $\breve Z\setminus \widetilde Z$
by finitely many definable conditions (the ones that define $U$), so it is $K$-definable. 
Hence it suffices to prove that $\widetilde U$ is $K$-definable.
As it is already pro-$K$-definable, it suffices to prove that $\widetilde U$
is $M$-definable for some model $M$ of $\acvf$ containing $K$. Therefore we can assume that $K$
is a model of $\acvf$.

First remind that 
$\widetilde{\P^1}$ is $K$-definable by 
\ref{projective-line-explicit}. 
By Theorem \ref{structure-theo} we know that $\widehat U$ admits
a Nash-admissible $K$-skeleton $\Sigma$. Let $\mathscr V$ be a $K$-triangulation 
with skeleton $\Sigma$. 
We then have 
\[U=\left(\coprod_{I\in \pi_0(\Sigma 
\setminus \mathscr V)}\rho_{U,\Sigma}^{-1}(I)\right)\coprod\left(\coprod_{v\in \mathscr V}
\rho_{U,\Sigma}^{-1}(v)\right).\]
For proving that $\widetilde U$ is $K$-definable
it suffices to know that for every summand $V$ of the above decomposition, 
$\widetilde V$ is $K$-definable. 
Assume first that $V=\rho_{U,\Sigma}^{-1}(I)$ for some connected component
$I$ of $\Sigma\setminus \mathscr V$. Then $V$ is $K$-definably Nash-isomorphic
to 
an open annulus $A= \{r<\abs T<R\}\subset \P^1$ with $r$ and $R$ in $\abs K$. 
Then $\widetilde A$ is relatively $K$-definable
in the $K$-definable set $\widetilde{\P^1}$, 
hence it is $K$-definable and so is $\widetilde V$.  
Assume now that $V=\rho_{U,\Sigma}^{-1}(v)$ for some vertex $v\in \mathscr V$. 
If $v\in U(K)$ then $V=\{v\}=\widetilde V$ and the result is obvious. Assume that $v$
is not a simple point and use the notation $F_v, f_v, f_\zeta$ of Definition 
\ref{defi-strong-adm} (4a) and (4b). Set 
$r=\abs {f(v)}$ and $C'_v=C_{v,\widehat U\setminus \Sigma}
\setminus F_v$, and let
$C'_{\eta_{0,r}}$ be the image of $C'_v$ in $C_{\eta_{0,r}}$. 
Set $W=\rr_{U,\Sigma}^{-1}(C'_v)$. 
Then $V$ is the disjoint union of $W$ and of the $V_\zeta:=\rr_{U,\Sigma}^{-1}(\zeta)$
for
$\zeta \in F_v$; each $V_\zeta$ is
$K$-definably isomorphic to an open disc through $f_\zeta$, so each
$\widetilde{V_\zeta}$ is $K$-pro-definably 
isomorphic to a relatively $K$-definable subset of $\widetilde{\P^1}$, so it is $K$-definable. 
It remains to show that $\widetilde W$ is $K$-definable.  Set $\Delta=\rr_{\A^1,\eta_{0,r}}^{-1}(C'_v)$. 
It follows once again from the $K$-definability of $\widetilde{\P^1}$
that $\widetilde \Delta$ is $K$-definable. But $\widetilde W$ can be now identified through
$(f_v,\trr_{W,v})$ to $\widetilde \Delta\times_{\widehat{C'_{\eta_{0,r}}}}\widehat {C'_v}$, so it is
$K$-definable. 

\subsubsection{Proof of (2)}
We know by Theorem \ref{structure-theo} that there exists a Nash-admissible $K$-skeleton 
$\Tau$ of $\widehat U$ containing $\Sigma$. It follows from 
Theorem \ref{properties-triangulation}
that $\Tau_{\RV,U}$ is $\RV$-internal. 
Now Proposition \ref{prop-connected-comp-definable} (1) 
show that $\pi_0(\widehat U\setminus \Sigma)$ is
the union of a $K$-definable 
subset of $\Tau_{\RV,U}$ 
and of a finite set, whence (2). 

\subsubsection{Proof of (3)}
By Theorem \ref{structure-theo}
there exists a $K$-triangulation of $\widehat U$; let $\Sigma$ be its skeleton. 
Let $\mathscr V'$ be the subset of $\mathscr V$ consisting of non-simple points whose residue curve has positive genus. 
Then $\mathscr V'$ is Galois-invariant, so it is $K$-definable. Let us show that $\mathscr V'$ is precisely 
the set of non-simple points of points of $\widehat U$ whose residue curve has positive genus, which
will end the proof. 
Let $x$ be a non-simple point of $\widehat U$ that does not belong to $\mathscr V'$; let is prove that $C_x$ 
has genus zero. We distinguish three cases. 

\paragraph{The case where $x\in\mathscr V$}
The genus of $C_x$ is then zero by the very definition of $\mathscr V'$. 

\paragraph{The case where $x\in \Sigma\setminus \mathscr V$}
Then $x$ lies on some connected component $I$ of $\Sigma \setminus \mathscr V$. 
By definition of a triangulation there exists some étale map from a $\overline K$-Zariski neighborhood of
$\rho_{U,\Sigma}^{-1}(I)$ to $\A^1$ inducing a Nash isomorphism between $\rho_{U,\Sigma}^{-1}(I)$ 
and an open annulus. The point $f(x)$ is then a non-simple point of $\A^1$,
so it is fo the form 
$\eta_{a,r}$. The map $f$ induces an isomorphism between $C_x$ and $C_{\eta_{a,r}}$, and the latter
is isomorphic to the projective line, so $C_x$ is of genus zero. 

\paragraph{The case where $x\notin \Sigma$}
The connected component of $\widehat U\setminus \Sigma$ that contains $x$
is equal to $\widehat V$ for some definable subset $V$ of $U$, and
by definition of a triangulation there exists some étale map from a Zariski neighborhood of
$V$ to $\A^1$ inducing a Nash isomorphism between $V$ 
and an open disc. The point $f(x)$ is then a non-simple point of $\A^1$, so it is fo the form 
$\eta_{a,r}$. The map $f$ induces an isomorphism between $C_x$ and $C_{\eta_{a,r}}$, and the latter
is isomorphic to the projective line, so $C_x$ is of genus zero. 
\end{proof}

\begin{rema}The finiteness statement in (3) can also be deduced from the genus inequality proved in \cite[Theorem 3.1]{gmp1}.
\end{rema}

Now we
will provide yet another proof of the semi-stable reduction theorem.
It might seem in some sense unsatisfactory, since we get it only in height 1, and have to leave at some point
the world
of stable completions and use
the genuine Berkovich formalism instead, and then rely on a lot of technical results 
in formal and rigid geometry.

But in a subsequent work we will use the formalism
and the results
of this current
paper for investigating models of
varieties over valuation rings of arbitrary height, and we plan
to build directly
from our triangulations
semi-stable models for curves,
remaining in the stable completion setting. We also intend to explore what can be done 
in higher dimensions using  the methods developed in the present paper.

\begin{theo}\label{semi-stable}
Let $K$ be a henselian 
valued field of height 1 and let $X$ be a smooth projective curve over $K$. 
There exists a finite extension $L$ of $K$ such that $X_L$ admits a semi-stable model
over the ring of integers of $L$. 
\end{theo}

\begin{proof}
Let us first assume that $K$ is complete, non-trivially valued and
algebraically closed. Let $\Sigma$ be a Nash-admissible $K$-skeleton 
of $\widehat X$. By Chapter 14 of \cite{hrushovski-l2016}, this skeleton $\Sigma$
gives rise to a skeleton $\Upsilon$ on the Berkovich analytification $X\an$ of the curve
$X$, and the $K$-definable deformation retraction of $\widehat X$ toward $\Sigma$ induces
a deformation retraction from $X\an$ toward $\Upsilon$. 
Let $U$ be a connected component of $X\an\setminus \Upsilon$. By the analytic Nullstellensatz 
$U$ has a $K$-point $x$, and it follows from the constructions in \cite{hrushovski-l2016}
that $U$ is the Berkovich analytification $V\an$ where $V$ is the $K$-definable
subset $\rr_{X,\Sigma}^{-1}(\rr_{X,\Sigma}(x))\cap X$. 
By definition of a Nash-admissible $K$-skeleton, there exists a $K$-regular étale map $f$
from a dense $K$-Zariski open subset $X'$ of $X$ to $\A^1$ that induces a $K$-definable
bijection between $V$ and an open disc $D\subset \A^1$. 
Then $f$ induces an étale map $V\an\to D\an$ and moreover by Lemma \ref{lemm-same-henselization}
one has $K(f(y))\h=K(y)\h$ for every $y\in V$. It follows that the completion of
$K(y)$ is equal to that of $K(f(y))$ when the valuation of $K(y)$ is of height 1, so 
$f$ induces an isomorphism $\hr{f(z)}\simeq \hr z$ for every Berkovich point
$z\in V\an$. Being étale, $f$ is then a local isomorphism around every point of
$V\an$. Moreover the bijectivity of $f\colon V\to D$ implies the bijectivity 
of $V\an \to D\an$ in view of the preservation of completed residue fields
and of the fact that fibers are finite (so that any pre-image on $V$ of a point
of $D$ with height 1 valuation still has height 1 valuation). 
Then $f\colon V\an \to D\an$ is a bijective local isomorphism, so it is an isomorphism. 

Therefore every connected component of $X\an\setminus \Upsilon$ is isomorphic to an open disc. 
The same kind of transfer argument also ensures that there is a finite family 
$\mathscr V$ of vertices
in $\Upsilon$ such that every connected component $I$ of $\Upsilon
\setminus \mathscr V$ is an open interval, whose pre-image under the retraction
$X\an \to \Sigma$ is isomorphic to an open annulus. 

These facts are classically equivalent to the semi-stable reduction for $X$ over $K$, see
for instance Chapter 6 of \cite{ducros2024} (this relies on highly non-trivial results like
Grauert's finiteness, Bosch's study of formal fibers, the algebraization of one-dimensional proper
formal schemes, \dots). Therefore $X$ has a semi-stable model. 

Now let us handle the general case. There is nothing to do if $K$ is trivially valued. 
If it is not, let $A$ be the ring of integers of $K$ and let
$\mathbb K$ be the completion of an algebraic closure of $K$. 
By the above, $X_{\mathbb K}$ admits a semi-stable model over the ring of integers 
$B$ of $\mathbb K$. 
By approximation there exists a finitely presented integral $A$-scheme $S$ having a $B$-point
and a proper semi-stable $S$-scheme $\mathscr X$ equipped 
with an isomorphism $\mathscr X_{S_K}\simeq X\times_K S_K$.  
By quantifier elimination in $\acvf$, $S$ also has a $B'$-point where
$B'$ is the ring of integers of the algebraic closure of $K$ in $\mathbb K$
(the reader who would like at this stage to see a purely scheme-theoretic proof
of the non-emptyness of $\mathscr X(B')$ 
may also refer to Theorem 7.4.8 of \cite{ducros2018}). So $X$ has a semi-stable model
over the ring of integers of an algebraic closure of $K$, and then over the ring of integers
of a finite extension of $K$ by approximation. 
\end{proof}

Finally we are going to answer positively 
the original question by Franziska Jahnke and Franz-Viktor Kuhlmann,
by providing a 
geometric and model theoretic proof
of 
Kuhlmann's henselian rationality theorem for tame henselian fields
 from
\cite{kuhlmann2019}
(see more comments on this  in the Introduction). 

We recall that a henselian valued field $K$ is called
\emph{tame} if every finite extension of $K$ is tamely ramified. This amounts
to requiring that $K$ be defectless and $\RV(K)$ be perfect (this implies that $K$
itself is perfect); and it is easily seen that $\RV(K)$ is perfect if and only if
$\k(K)$ is perfect and $\abs{K^\times}$ is $p$-divisible, where $p$ is the residue characteristic exponent. 
If the residue characteristic of $K$ is zero, $K$ is tame. 
%
%
%
%

\begin{theo}[Kuhlmann]\label{henselian-rationality}
Let $K$ be a tame henselian valued field and
let $L$ be a 
finitely generated extension of transcendence degree 1 over $K$,
equipped with an immediate extension $v$ of the valuation of $K$. There exists an element $t$ of $L$, transcendental
over $K$, such that the inclusion $(K(t), v|_{K(t)}) \hookrightarrow (L,v)$ induces an isomorphism
\[(K(t), v|_{K(t)})\h \simeq (L,v)\h.\]
\end{theo}

\begin{proof}

We first need to investigate the behavior of $(L,v)$ under finite
ground field extension. 

\subsubsection{}\label{behavior-scalar-extension}
Let $F$ be a finite extension of $K$.  
Write $L\h\otimes_K F=\prod L_i$, where the $L_i$ are finite separable
extensions
of $L\h$. 

Since $K$ is tame, $F$ is tamely ramified over $K$. 
It follows then from \cite[2.21]{ducros2013b}
that the natural map 
$\RV(L)\otimes_{\RV(K)}\RV(F)\to \prod \RV(L_i)$ is an isomorphism. 
But as $L$ is an immediate extension of $K$ one has
$\RV(L)=\RV(K)$, so $\prod \RV(L_i)=\RV(F)$. 
This means that $L\h\otimes_K F$ is a field, and is an immediate extension 
of $F$. Otherwise said $L\otimes_K F$ is a field,  
and the valuation on $L$ has a unique extension to
$L\otimes_K F$, which makes it an immediate extension of $F$. 

Let $\overline K$ be an algebraic closure of $K$. 
By the above and a colimit argument, $L\otimes_K \overline K$
is a field, to which 
$v$
has a unique extension $w$
and which makes $L\otimes_K \overline K$ an immediate extension of $\overline K$.

\subsubsection{}
Let $X$ be an
integral $K$-curve with function field $L$. The valuation $w$ defines a
quantifier-free
type on 
$X$ over $K$; let $x$ be a realization of this type in some model. 
We have seen that
$w$ is the  unique extension of
$v$
to $L\otimes_K \overline K$; it follows
that $\mathrm{tp}(x/\overline K)$ is invariant under $\mathrm{Gal}(\overline K/K)$. 
%
%
As $\Sigma_{\RV,X}$ is $\RV$-internal 
by Theorem \ref{properties-triangulation}, 
its point $\rr_{X,\Sigma}(x)$ is $\overline K\RV(K(x))$-definable, 
so it is $\overline K$-definable since $\RV(K(x))=\RV(K)$ because
$(L,v)$ is an immediate extension of $K$. 
Moreover $\rr_{X,\Sigma}(x)$ is
invariant under $\mathrm{Gal}(\overline K/K)$, 
since this is already the case for $\mathrm{tp}(x/\overline K)$. 
As a consequence $\rr_{X,\Sigma}(x)$ is $K$-definable, 
and so is $U:=\rr_{X,\Sigma}^{-1}(\rr_{X,\Sigma}(x))$. 
By definition of a Nash-admissible $K$-skeleton  
there exist a $\overline
K$-Zariski open subset
$X'$ of $X$ containing $U$
and a $\overline K$-regular étale
map  $f\colon X'\to \A^1$
that induces a Nash
isomorphism between $U$ and an open disc $D$.
Let $F$ be a finite extension of $K$ over which
$f$ and $X'$ are defined. 
Then $f$ is an element of $F[U]$ inducing 
a Nash
isomorphism between $U$ and $D$. 
As $K$ is tame, $F$ is a tamely ramified extension of $K$. 
By tame descent for open (poly)discs
(Theorem \ref{theo-tame-descent}) there exists
$g\in K[U]$ that induces 
a Nash
isomorphism between $U$ and an open
disc $\Delta$.

This implies that $K(g(x))\h=K(x)\h$. 
As $K(x)$ is isomorphic to $L$ as a field and $g$ is étale, 
$g(x)$ is transcendental over $K$. Identifying $K(x)$ with $L$ and setting
$t=g(x)$, we get the required statement. 
\end{proof}

\bibliographystyle{smfalpha}
\bibliography{aducros}

\end{document}